\documentclass{amsart}
\usepackage[margin = 1in]{geometry}
\usepackage{amsmath}
\usepackage{amssymb}
\usepackage{amsthm}
\usepackage{xcolor}
\usepackage{tikz}
\usepackage{tikz-cd}
\usepackage{mathrsfs}
\usepackage{stmaryrd}
\usepackage{colonequals}
\usepackage{hyperref}
\newcommand{\doots}{,\dots,}
\newcommand{\NN}{\mathbb N}
\newcommand{\PP}{\mathbb P}

\newcommand{\QQ}{\mathbb Q}
\newcommand{\OO}{\mathcal O}
\newcommand{\XXX}{\mathscr X}

\newcommand{\oF}{\overline{F}}

\newcommand{\CC}{\mathbb C}
\newcommand{\RR}{\mathbb R}
\newcommand{\PPP}{\mathcal P}

\newcommand{\ZZ}{\mathbb Z}

\newcommand{\Fv}{F_v}

\newcommand{\Gm}{\mathbb G_m}

\newcommand{\YYY}{\mathscr Y}

\newcommand{\AAA}{\mathbb A}

\newcommand{\XX}{\mathcal X}

\newcommand{\thickslash}{\mathbin{\!\!\pmb{\fatslash}}}

\DeclareMathOperator{\Spec}{Spec}
\DeclareMathOperator{\Proj}{Proj}
\DeclareMathOperator{\sProj}{\mathcal{P}roj}

\DeclareMathOperator{\End}{End}
\DeclareMathOperator{\Gal}{Gal}
\DeclareMathOperator{\age}{age}

\DeclareMathOperator{\un}{un}
\DeclareMathOperator{\orb}{orb}
\DeclareMathOperator{\lcm}{lcm}
\DeclareMathOperator{\Hom}{Hom}
\DeclareMathOperator{\GL}{GL}

\DeclareMathOperator{\ord}{ord}

\DeclareMathOperator{\Pic}{Pic}
\DeclareMathOperator{\Eff}{\overline{Eff}}
\DeclareMathOperator{\NS}{NS}

\DeclareMathOperator{\rig}{rig}

\DeclareMathOperator{\Aut}{Aut}

\DeclareMathOperator{\Id}{Id}
\DeclareMathOperator{\SL}{SL}
\DeclareMathOperator{\naive}{naive}
\DeclareFontFamily{U}{wncy}{}
\DeclareFontShape{U}{wncy}{m}{n}{<->wncyr10}{}
\DeclareSymbolFont{mcy}{U}{wncy}{m}{n}
\DeclareMathSymbol{\Sh}{\mathord}{mcy}{"58}
\newtheorem{conj}{Conjecture}[section]
\newtheorem{thm}[conj]{Theorem}
\newtheorem{lem}[conj]{Lemma}
\newtheorem{cor}[conj]{Corollary}

\theoremstyle{definition}
\newtheorem{rem}[conj]{Remark}

\newtheorem{mydef}[conj]{Definition}

\newtheorem{ex}[conj]{Example}
\begin{document}
	
	\title{The Stacky Batyrev--Manin conjecture and modular curves}
	
	\author{Ratko Darda}
	\address{Faculty of Engineering and Natural Sciences, Sabanc\i{} University, 34956 Istanbul, Turkey}
	\email{ratko.darda@sabanciuniv.edu} 
	
	\author{Changho Han}
	\address{Department of Mathematics, Korea University, Seoul 02841, Republic of Korea}
	\email{\nolinkurl{changho_han@korea.ac.kr}} 
	

\begin{abstract}
	Let $\XXX_0(N)$ be the Deligne--Rapoport modular stack of elliptic curves endowed with a cyclic rational $N$-isogeny over a number field~$F$. 
	Let $N\in\{1,2,3,4,5,6,7,8,9,10,12,13,16,18,25\},$ which are precisely the values for which the coarse moduli space of $\XXX_0(N)$ is isomorphic to $\PP^1$. 
	We show that the stacky Batyrev--Manin conjecture \cite{dardayasudabm} holds for the naive height on $\XXX_0(N)$ when $F=\QQ$. In the process, we give a concrete description of $\XXX_0(N)$ as a square root stack over a stacky curve. 
\end{abstract}

\maketitle

\section{Introduction}\label{sec:intro}
\subsection{Batyrev--Manin conjecture and stacks}
\subsubsection{} One of the most prominent conjectures about the distribution of rational points on varieties is the Batyrev--Manin conjecture. It predicts the asymptotic behavior of the number of rational points of bounded height on varieties which have ``many'' rational points. 
Under some assumptions, it conjectures \cite{Peyre, BatyrevTschinkel} that the number of points of height bounded by $B$ on {\it geometrically rationally connected varieties} grows as $CB^a\log(B)^{b-1}$ for certain explicit constants $C>0$, $a>0$, and $b\geq 1$.
The constants $a$ and $b$, which are the main focus of this paper, are expressed in terms of relative positions of the line bundle defining height and the canonical line bundle to the effective cone of the variety.

Recently, the Batyrev--Manin conjecture has been studied for {\it Deligne--Mumford stacks} (DM stacks for short) \cite{ellenberg_satriano_zureick-brown_2023, dardayasudabm, dardayasudabm2}. 
The motivation comes from the fact that some counting results and predictions, such as Malle's conjecture \cite{Malle} or counting elliptic curves, with or without a level structure, of bounded height \cite{boggesssankar, Molnar2023, phillips2024pointsboundedheightimages, 7isogeny}, often feature a formula of the same shape $CB^a\log(B)^{b-1}$. 
They turn out to be concerned with rational points of DM stacks. 
The main conjecture of \cite{dardayasudabm} gives a prediction for the number of rational points of bounded height for a reasonable class of DM stacks.

\subsubsection{} 
Defining a useful height on a Deligne--Mumford stack requires more data than a single line bundle. 
Namely, a line bundle on a DM stack does not produce a useful height as some power of it is a pullback of a line bundle from the stack's {\it coarse moduli space}. 
In this case, the height is (up to taking a power) the pullback of a height from the coarse moduli space and often fails the Northcott property. 
In \cite{dardayasudabm}, the problem is resolved by the use of {\it sectors}. 

Let $\XX$ be a smooth irreducible, separated DM stack over a number field $F$ with projective geometrically rationally connected coarse moduli space. 
Roughly, a sector of $\XX$ is a deformation equivalence class of a pair of a geometric point $x$ of $\XX$ and 
an injection $\mu_k \hookrightarrow \Aut(x)$ (with $k\geq 1$), where $\Aut(x)$ is the automorphism group scheme of $x$. 
If $k=1$, the resulting sector is called the {\it untwisted} sector, while the other sectors are called the {\it twisted} sectors. Now, the role of the N\'eron--Severi space of a variety, which parametrizes heights up to $O(1)$ quotients, is played by its orbifold analogue $$\NS_{\orb}(\XX)_{\QQ}:=\NS(\XX)_{\QQ}\oplus\bigoplus_{y \in \pi_0^*(\mathcal J_0(\XX))}\QQ[y],$$
where $\NS(\XX)$  is the N\'eron--Severi space of $\XX$ and $\pi_0^*(\mathcal J_0(\XX))$ is the (finite) set of  twisted sectors of $\XX$.  The $a$- and $b$-invariants are expressed in terms of relative positions of the element defining the height, call it $(L,x)$, and the orbifold analogue $K_{\XX,\orb}$ of the canonical line bundle \cite[Definition 9.1]{dardayasudabm} to the orbifold analogue  $\Eff_{\orb}(\XX)$ of the pseudo-effective cone \cite[Definition 8.1]{dardayasudabm}:
$$a(L,x):=\inf\{t\in\RR|\hspace{0,1cm}t\cdot (L,x)+K_{\XX,\orb}\in\Eff_{\orb}(\XX)\}$$ and $$b(L,x):=\text{codimension of the minimal face of }\Eff_{\orb}(\XX)\text{ containing }a(L,x)\cdot (L,x)+K_{\XX,\orb}.$$
Let us now state the main conjecture of \cite{dardayasudabm}. Given a stack $\mathcal Y$ and a ring $R$, denote by $\mathcal Y\langle R\rangle$ the set of $R$-isomorphism classes of $R$-points of $\mathcal Y$. 
\begin{conj}[{\cite[Conjecture 9.6]{dardayasudabm}}]\label{mainconjecture}
	Assume that $\XX\langle F\rangle$ is Zariski-dense in $\XX$.  Let $(L,x)\in\Eff_{\orb}(\XX)$ and let $H=H_{(L,x)}:\XX\langle F\rangle\to\RR_{>0}$ be the corresponding height\footnote{
		We do not impose the {\it adequacy} condition of \cite[Definition 5.2]{dardayasudabm}, which is required for \cite[Conjecture 5.6]{dardayasudabm} but redundant for \cite[Conjecture 9.10]{dardayasudabm2}; this is consistent with the sequel \cite{dardayasudabm2}, and several of our cases are non-adequate.
		}. There exist a thin set $T\subset \XX\langle F\rangle $ and a constant $C>0$ such that 
	$$\#\{\XX\langle F\rangle -T|\hspace{0,1cm}H(x)\leq B\}\sim_{B\to\infty}C B^{a(L,x)}\log(B)^{b(L,x)-1}.$$
\end{conj}
\subsection{Results} 
The main theorem of this paper proves Conjecture~\ref{mainconjecture}  for the Deligne--Rapoport modular curves $\XXX_0(N)$ over~$\QQ$ for the values of $N$ in 
\begin{equation} \label{eq:N_genus_0}
	\mathfrak N_0 \colonequals \{1,2,3,4,5,6,7,8,9,10,12,13,16,18,25\}
\end{equation} 
with respect to the {\it naive} height. 
Recall that for $N\in\ZZ_{>0}$, $\XXX_0(N)$ is  the Deligne--Mumford compactification of the moduli stack of elliptic curves endowed with a cyclic rational $N$-isogeny over $F$ (see Definition \ref{defx0n} for the precise definition).
The set $\mathfrak N_0$ is precisely the set of those values for which the coarse moduli space $X_0(N)$ of $\XXX_0(N)$ is geometrically rationally connected, i.e. isomorphic to $\PP^1$ (cf. \cite{Mazur78}). When $N=1$ the naive height\footnote{some sources use a different normalization, e.\ g.\ in \cite{AHPP25}, it is given by $\max(4|e^3|, 27|f^2|),$ but it has no effect on $a$- and $b$-invariants.} $H:\XXX_0(1)\langle \QQ\rangle\to\RR_{>0}$ is given as follows: the height of the curve $E:Y^2=X^3+eX+f,$ where $e,f\in\ZZ$ and for all $p$ prime, $p^4\nmid e$ or $p^6\nmid f$, is  $H_1(E):=\max(|e^3|, |f^2|)$. For $N>1$, the naive height on $\XXX_0(N)\langle \QQ\rangle$ is the pullback of the naive height from $\XXX_0(1)$ along the  morphism $J_N:\XXX_0(N)\to\XXX_0(1)$, which forgets the isogeny; so $H_N(E)=H_1(J_N(E))$. We drop $N$ from the index when it is clear from the context.   
The rational points of stacks $\XXX_0(N)$ of bounded  naive height have been counted in \cite{Brumer, Grant, countthreeisog, zbMATH07357690, phillips2024pointsboundedheightimages, Molnar2023, 7isogeny, boggesssankar, AHPP25}.  By Lemma~\ref{Zariski-dense}, one has that $\XXX_0(N)\langle \QQ\rangle$ is Zariski-dense in $\XXX_0(N)$, hence is covered by Conjecture~\ref{mainconjecture}.
\begin{thm}\label{thm:mainthm} 
	Let $N \in \mathfrak{N}_0$. 
	When $F=\QQ$,  Conjecture~\ref{mainconjecture} holds for $\XXX_0(N)$ with respect to the height~$H$.
\end{thm}
The definition of the naive height is extended to a general number field~$F$ in Definition~\ref{naive}. The following theorem calculates values of the $a$- and $b$-invariants for the naive height in the general number field case.
 \begin{thm}\label{valuesaandb} Let $M_{\naive}\in\NS_{\orb}(\XXX_0(N))_{\QQ}$ be the element defining $H$. For $N\in\mathfrak N_0$ and $M_{\naive}$, one has the following values of $a$- and $b$-invariants:
 	\begin{center}
 		\begin{tabular}{|c|c|c|}
 			\hline
 			$N$	& $a(M_{\naive})$& $ b(M_{\naive})$ \\
 			\hline
 			1	&  5/6      &     1    \\
 			\hline 
 			2	&   1/2      &    1 \\
 			\hline 
 			3	&    1/3   &  2        \\
 			\hline
 			4	&         1/3  & 1 \\
 			\hline
 			5	&     1/6  &     5-[F(i):F]\\
 			\hline
 			6	&       1/6     &  2  \\
 			\hline
 			7	&       1/6   &  2    \\
 			\hline
 			8	&        1/6    &   2   \\
 			\hline
 			9	&        1/6    &   2   \\
 			\hline
 			10	&       1/6    &   1   \\
 			\hline
 			12	&       1/6    &   1   \\
 			\hline
 			13	&        1/6    &   1   \\
 			\hline
 			16	&        1/6   &   1  \\
 			\hline
 			18	&        1/6   &   1  \\
 			\hline
 			25	&        1/6  &   1  \\
 			\hline
 		\end{tabular}
 	\end{center}
 \end{thm} 
 
 
 \begin{proof}[Proof of Theorem~\ref{thm:mainthm} assuming Theorem~\ref{valuesaandb}]
 The theorem follows from Theorem~\ref{valuesaandb} and the statements which are listed below:
 
 \begin{center}
 \begin{tabular}{|c|c|}
 	\hline
 	$N$	& $\text{Statement}$ \\
 	\hline
 	1	&  \text{\cite[Lemma 4.3]{Brumer}}         \\
 	\hline 
 	2	&  \text{\cite[Proposition 1]{Grant}}     \\
 	\hline 
 	3	&    \text{\cite[Proposition 4.2]{boggesssankar}}           \\
 	\hline
 	4	&         \text{\cite[Theorem 4.2]{zbMATH07357690}} \\
 	\hline
 	5	&    \text{\cite[Theorem 1.1]{AHPP25}}  \\
 	\hline
 	6	&       \text{\cite[Theorem 1.2.1]{Molnar2023}}     \\
 	\hline
 	7	&       \text{\cite[Theorem 1.2.2]{7isogeny}}       \\
 	\hline
 	8	&        \text{\cite[Theorem 1.2.1]{Molnar2023}}      \\
 	\hline
 	9	&        \text{\cite[Theorem 1.2.1]{Molnar2023}}      \\
 	\hline
 	10	&        \text{\cite[Theorem 1.2.4]{Molnar2023}}       \\
 	\hline
 	12	&       \text{\cite[Theorem 1.2.4]{Molnar2023}}       \\
 	\hline
 	13	&        \text{\cite[Theorem 1.2.4]{Molnar2023}}      \\
 	\hline
 	16	&       \text{\cite[Theorem 1.2.4]{Molnar2023}}   \\
 	\hline
 	18	&        \text{\cite[Theorem 1.2.4]{Molnar2023}}    \\
 	\hline
 	25	&        \text{\cite[Theorem 1.2.4]{Molnar2023}}   \\
 	\hline
 \end{tabular}
 \end{center}
 \end{proof}
 \begin{rem}
 	\normalfont Note that when $N=3$, Theorem~\ref{thm:mainthm} is false without removing a thin set. As it is clear from \cite[Proposition 4.2]{boggesssankar}, a sufficient thin set is given by $$\{(E,\phi)\in\XXX_0(3)\langle F\rangle|\hspace{0,1cm}j(E)=0\}.$$ 
 \end{rem}
 \subsection{Strategy of proof of Theorem~\ref{valuesaandb}}
 To prove Theorem~\ref{valuesaandb}, we first give a concrete description of $\XXX_0(N)$ as a {\it square root stack} over a stacky curve, i.e. over a DM stack of dimension $1$ which is generically a genuine curve. 
 Recall that given a line bundle $L$ on a stack~$\mathcal X$, the $n$-th root stack of $L$, denoted by $\sqrt[n]{L/\XX}$, is a stack defined in \cite{impose_tangency}, endowed with a morphism to $\mathcal X$ with the property that the pullback of $L$ is an $n$-th power in the Picard group of~$\sqrt[n]{L/\XX}$. 
 The {\it square root stacks} are $\mu_2$-gerbes, and are {\it trivial} $\mu_2$-gerbes if and only if $L$ is a square in the Picard group $\Pic(\mathcal X)$. Hence, the square root stacks are classified by $\Pic(\XX)/2\Pic(\XX)$. In \cite{Abramovich_Olsson_Vistoli}, one describes a notion of {\it rigidification} $\mathcal Y^{\rig}$ of a stack $\mathcal Y$, which essentially kills a subgroup scheme of the generic stabilizer group scheme of a stack. If $\mathcal X$ is a stacky curve and $L$ is a line bundle on $\mathcal X$, the generic stabilizer of $\sqrt[2]{L/\XX}$ is $\mu_2$ and the rigidification is just $\mathcal X$. 

 We have a canonical representable morphism $J_N:\XXX_0(N)\to \XXX_0(1)$ and an induced representable morphism $J_N^{\rig}:\XXX_0(N)^{\rig}\to \XXX_0(1)^{\rig}$. 
 Let $\Lambda$ be the Hodge line bundle on $\XXX_0(1)$. 
 There exists a unique line bundle $\Lambda^{\rig}$ on $\XXX_0(1)^{\rig}$ such that $\Lambda^{\otimes 2}$ is isomorphic to the pullback of $\Lambda^{\rig}$ under the rigidification morphism $\XXX_0(1)\to\XXX_0(1)^{\rig}$. We set $\Lambda_{N}:=J_N^*\Lambda$ and $\Lambda_{N}^{\rig}:=(J_N^{\rig})^*\Lambda^{\rig}$.
  We will express $\XXX_0(N)\cong\sqrt{\Lambda^{\rig}_N/\XXX_0(N)^{\rig}}$. Moreover, to give a concrete description of the stacky curve $\XXX_0(N)^{\rig}$, we use the following notion. Given an $r$-tuple of disjoint Cartier divisors $(C_1\doots C_r)$ on a curve $Z$, we denote by $\sqrt[n_1\doots n_r]{C_1\doots C_r/Z}\to Z$ the $(n_1\doots n_r)$-th root stack of the $r$-tuple,  which can be thought of as adding stackiness $\mu_{n_i}$ along $C_i$. 
   Let $K_6$ and $K_4$ be the divisors $X^2+3Y^2=0$ and $ X^2+Y^2=0$, respectively, on $\PP^1$.  Let $E_0,E_{1728}$ be the divisors on $\PP^1$ corresponding to $X=0$ and $X=1728Y$, respectively. 
 
\begin{thm}\label{isomclas} Assume that $N \in \mathfrak{N}_0$. Then there is a canonical isomorphism $\XXX_0(N)\cong \sqrt{\Lambda_N^{\rig}/\XXX_0(N)^{\rig}}$ corresponding to the line bundle $\Lambda_N$. The following table summarizes the rigidifications of $\XXX_0(N)$, under a suitable isomorphism of $\PP^1$ and the coarse space $X_0(N)$, as root stacks over $\PP^1$. 
\begin{center}
	\begin{tabular}{|c|c|c|}
		\hline
		$\XXX_0(N)$	& $\XXX_0(N)^{\rig}$& $ \text{trivial/non-trivial gerbe}$ \\
		\hline
		$\XXX_0(1)$	&  $\sqrt[3,2]{E_0, E_{1728}/\PP^1} $      &     \text{non-trivial}    \\
		\hline 
	$	\XXX_0(2)$	&   $\sqrt[2]{E_{1728}/\PP^1} $     &    non-trivial \\
		\hline 
	$	\XXX_0(3)$	&    $\sqrt[3]{E_{0}/\PP^1} $    &  trivial        \\
		\hline
	$\XXX_0(4)$	&         $\PP^1 $  & non-trivial \\
		\hline
		$\XXX_0(5)$&     $\sqrt[2]{K_4/\PP^1} $  &     \text{non-trivial}\\
		\hline
	$	\XXX_0(6)$	&       $\PP^1 $      &  \text{trivial}  \\
		\hline
	$	\XXX_0(7)$	&       $\sqrt[3]{K_6/\PP^1}$   &   \text{trivial}    \\
		\hline
	$	\XXX_0(8)$	&        $\PP^1 $   &    \text{trivial}  \\
		\hline
	$	\XXX_0(9)$	&        $\PP^1 $   &    \text{trivial}   \\
		\hline
	$	\XXX_0(10)$	&        $\sqrt[2]{K_4/\PP^1} $    &  \text{non-trivial}   \\
		\hline
	$	\XXX_0(12)$	&        $\PP^1 $    &    \text{trivial}   \\
		\hline
	$	\XXX_0(13)$	&        $\sqrt[3,2]{K_6, K_4/\PP^1} $    &   \text{non-trivial}   \\
		\hline
	$	\XXX_0(16)$	&        $\PP^1 $  &    \text{trivial}  \\
		\hline
	$	\XXX_0(18)$	&        $\PP^1 $   &    \text{trivial}  \\
		\hline
	$	\XXX_0(25)$	&       $\sqrt[2]{K_4/\PP^1} $   &   \text{non-trivial}  \\
		\hline
	\end{tabular}
\end{center}
\end{thm}
\medskip

We now describe the sectors of $\XXX_0(N)$. 
We recall that a sector of a connected smooth DM $F$-stack $\XX$ is an equivalence class of a pair $(x,g)$ with $x\in \XX(\oF)$ and $g:\mu_k\hookrightarrow\Aut(x)$ and $k\geq 1$, where $(x_1,g_1:\mu_k\hookrightarrow \Aut(x_1))$ and $(x_2,g_2:\mu_k\hookrightarrow\Aut(x_2))$ are equivalent if there exists a connected $F$-scheme $S$, two geometric points $y_1,y_2:\Spec(\oF)\to S$, a morphism $y:S\to \XX$ and an injection $h:(\mu_k)_S\to\Aut(x)$ and isomorphisms of the restriction $(y,h)$ at $y_i$ with  $(x_i,g_i)$ for $i=1,2$. 
A sector is untwisted if it is the equivalence class of $(x,\Id_x)$;
otherwise, it is twisted. 
The choice of the line bundle $\Lambda_N$ induces an identification of automorphism group schemes of points of $\XXX_0(N)$ with $\mu_m$ for $m\in\{2,4,6\}$. 
Let $\mathcal K_{6}$ and $\mathcal K_{4}$ be the closed substacks of $\XXX_0(N)$ where the automorphism group scheme is of order $6$ and $4$, respectively. 
Before stating the next theorem, we recall the definition of the age function on the set $\pi_0\mathcal J_0\XX$ of sectors of $\XX$. 
Given $(x,\iota:\mu_k\hookrightarrow\Aut(x))$ with $x\in\XX(\oF)$, we have a canonical representation of $\mu_k$ on the fiber of the tangent space $T_{\XX}(x)$. We decompose it as a sum of the irreducible representations $\otimes_{i=1}^{\dim(\XX)}\tau^{a_i}$ with $a_i\in \ZZ\cap [0\doots k-1]$, where $\tau:\mu_k\hookrightarrow\Gm$ is the canonical inclusion. Then the age of $(x,\iota)$ is defined to be $\frac{1}{k}\sum_{i=0}^{\dim(\XX)}a_i$. The function $\age$ descends to the sectors of $\XX$. In Paragraph~\ref{subsubsec:wtproj_stack}, we obtain the following explicit descriptions of sectors of $\XXX_0(N)$ and their ages:  
\begin{thm}\label{thmonsectors}
For every $N \in \mathfrak{N}_0$, the DM stack $\XXX_0(N)$ has the following sectors:
\begin{itemize}
	\item the untwisted sector, that we denote by $0$ and which is of age $0$;
	\item the twisted sector denoted by $1/2$, which is the unique twisted sector of age $0$;  
	\item for $0< \zeta<1$ with $\zeta\in\frac16\ZZ-\frac12\ZZ$ and for $T_6\in\mathcal K_6$, the twisted sector $(T_6,\zeta)$ which is the equivalence class of the pair $(x,\mu_6\xrightarrow{6\zeta}\mu_6)$ where $x\in T_6(\oF)$ is a point. The age of $(T_6,\zeta)$ is $\{4\zeta\} $. Here $\{\cdot\}$ denotes the fractional part of a real number.
	\item for $0< \zeta<1$ with $\zeta\in\frac14\ZZ-\frac12\ZZ$ and for $T_4\in\mathcal K_4$, the twisted sector $(T_4,\zeta)$ which is the equivalence class of the pair $(x,\mu_4\xrightarrow{4\zeta}\mu_4)$ where $x\in T_4(\oF)$ is a point. The age of $(T_4,\zeta)$ is $\{2\zeta\} $. 
\end{itemize}
\end{thm}
We now describe explicitly how the element $K_{\XXX_0(N),\orb}$ of $\NS_{\orb}(\XX)_{\QQ}$ plays the role of the canonical line bundle. 
We identify $\NS(\XXX_0(N))_{\QQ}$ with $\QQ$ via $\deg:\NS(\XXX_0(N))_{\QQ}\xrightarrow{\sim}\QQ$. 
\begin{cor}  
For every $N \in \mathfrak{N}_0$, one has that $K_{\XXX_0(N),\orb} \in \NS_{\orb}(\XXX_0(N))_{\QQ}$ is given by $$((-1+\varepsilon_2(N)/4+\varepsilon_3(N)/3),(\age(y)-1)_{y \in \pi_0^*J_0\XXX_0(N)})$$
where for $u=2,3$, one sets
$$\varepsilon_u(N) :=\begin{cases}
	0,&\text{ if } u^2|N\\
	\prod_{p|N}\bigg(1+ \bigg(\dfrac{u^2-7u+9}{p}\bigg)\bigg),&\text{otherwise.}
\end{cases} $$
\end{cor}


To prove Theorem~\ref{valuesaandb}, the computations of $a(M_{\naive})$ and $b(M_{\naive})$ are divided into two cases: {\it adequate} and {\it non-adequate} ones in the sense of \cite[Definition 5.2]{dardayasudabm}.
The distinction comes from the location of the point of the intersection of $tM_{\naive}+K_{\XXX_0(N),\orb}$ with $\Eff_{\orb}(\XXX_0(N))$. The adequate cases, which are simpler to handle, are treated in Theorem~\ref{aandb}. They correspond to  $N \in \{1,2,3,4,5,6,8,9\}$, where the formula for $a(M_{\naive})$ and $b(M_{\naive})$ is deduced from \cite[Conjecture 5.6]{dardayasudabm}.

The rest of the cases correspond to non-adequate cases, which we treat by giving a full or partial description of the cone $\Eff_{\orb}(\XXX_0(N)). $ 
The case $N\in\{12, 16,18\}$ is treated by Theorem~\ref{thm:N=12,16,18}, the case $N=7$ by Theorem~\ref{thm:N=7},  the case $N\in\{10,25\}$ by Theorem~\ref{thm:N=10,25}, and the case $N=13$ by Theorem~\ref{thm:N=13}.

\subsection{Organization}
In Section~\ref{sec:prelim}, we recall cyclotomic stacks, the stacky Proj construction, root stacks, and the sectors of DM stacks.
Section~\ref{sec:modular_curve} starts with the definition and geometric properties of modular curves $\XXX_0(N)$.
After that, the naive height of $\XXX_0(N)$ is studied in Section~\ref{subsec:naive_height}, which concludes with the description of the element defining the naive height in $\NS_{\orb}(\XXX_0(N))$ (Corollary~\ref{naivforx0n}).
In Section~\ref{subsec:cano_elt}, we consider the orbifold canonical divisor $K_{\XXX_0(N),\orb}$.
In Section~\ref{subsec:adequate}, we classify the adequate cases and then prove Theorem~\ref{valuesaandb} for these cases.
Section~\ref{subsec:local_situation} considers local computations needed to understand stacky curves over $\XXX_0(N)$, which is a preparation for the descriptions of the orbifold pseudoeffective cones $\Eff_{\orb}(\XXX_0(N))$ in Section~\ref{subsec:orb_peff_cone}.
In the remaining subsections of Section~\ref{sec:non-adequate}, Theorem~\ref{valuesaandb} is proved for the non-adequate cases.
\subsection{Notation}
Let $F$ be a fixed number field. 
By $M_F$ we denote the set of places of~$F$. 
By~$M_F^{0}$, respectively by~$M_F^{\infty}$, we denote the set of finite, respectively, infinite places of~$F$. 
For a finite place~$v$ we denote by~$q_v$ the cardinality of the residue field at~$v$. 
Given a scheme~$S$, we will use $(\mu_n)_S$ to denote the finite group scheme of $n$-th roots of unity. 
If~$S=\Spec(K)$, with~$K$ a field, we may also write $(\mu_n)_K$ for $(\mu_n)_S$. 
We may drop~$S$ or~$K$ from the index when they are clear from the context. 
When~$K$ is a field having all $n$-th roots of unity, we may identify the finite group scheme $(\mu_n)_{K}$ with the Galois module $\mu_n(K)$. 
By~$\zeta_n$ we denote a (fixed) generator of~$\mu_n$. We may denote by~$i$ the element~$\zeta_4$. Given a Deligne--Mumford stack~$X$, we will denote by $|X|$ the corresponding topological space. 
\subsection{Acknowledgments} 
The authors would like to thank Mohammad Sadek for many insightful discussions and for performing some MAGMA computations.
The authors also would like to thank Santiago Arango-Pi\~neros, Runxuan Gao, Oana Padurariu, Sun Woo Park, Tim Santens, and Takehiko Yasuda for helpful comments and conversations.
The first author has received a funding from the European Union’s Horizon 2023 research and innovation programme under the Maria Skłodowska-Curie grant agreement 101149785. The second author is supported by Korea University grants and the National Research Foundation of Korea(NRF) grant funded by the Korea government(MSIT) (RS-2025-24535254).

\section{Preliminaries}\label{sec:prelim}

\subsection{Cyclotomic stacks and the stacky Proj construction}\label{sec:stacky_proj}
Let us recall the stacky Proj construction; see \cite[Example 10.2.8]{Olsson16} or \cite[\S 2]{AHPP25} for details. 
Given a graded ring $R = \bigoplus_{d \geq 0} R_d$, the grading specifies a $\Gm$-action on $\Spec R$. 
This action fixes the point $V(R_+)$ corresponding to the irrelevant ideal $R_+ \colonequals \bigoplus_{d \ge 0} R_d$; when $R_0$ is a field or a ring of integers, we abbreviate $V(R_+)$ as the $\Spec R_0$-point $0$ in $\Spec R$. 
Then the {\it stacky Proj} of $R$ is the quotient stack
\[\sProj R \colonequals [(\Spec R - V(R_+))/\Gm].\] 
Its coarse moduli space is $\Proj R$, the usual Proj of $R$.
On $\sProj R$, the line bundle $\OO_{\sProj R}(1)$ comes from the associated classifying morphism $\sProj R \to B\Gm$ corresponding to the $\Gm$-torsor 
\[(\Spec R - V(R_+)) \to \sProj R.\]
If $R$ is a finitely generated graded $F$-algebra with $R_0=F$, then $\OO_{\sProj R}(1)$ is ample, i.e. there exists $m > 0$ such that $\OO_{\sProj R}(m)$ is isomorphic to the pullback of an ample line bundle on a quasiprojective $F$-variety $\Proj R$.

\begin{ex}\label{ex:wtproj_stack} 
	Suppose that a given graded ring $R$ is free and finitely generated in positive degrees, i.e. $R = \ZZ[x_0,\dotsc,x_n]$ where for every $i$, $x_i \in R_{w_i}$ for some $w_i >0$.
	Then $\mathcal{P}_\ZZ(w_0,\dotsc,w_n) \colonequals \sProj R$ the weighted projective stack over $\Spec \ZZ$.
	Using the definition of $\sProj$, $\mathcal{P}_\ZZ(w_0,\dotsc,w_n)=[(\AAA^{n+1}_\ZZ - \{0\})/\Gm]$, where $\Gm$-action on $\AAA^{n+1}_\ZZ - \{0\}$ is given by $t \cdot (x_0,\dotsc,x_n) = (t^{w_0}x_0,\dotsc,t^{w_n}x_n)$.
	Given a base scheme $S$, $\mathcal{P}_S(w_0,\dotsc,w_n)$ is the pullback of $\mathcal{P}_\ZZ(w_0,\dotsc,w_n)$ on $S$; when there is no ambiguity on $S$, we abbreviate $\mathcal{P}_S(w_0,\dotsc,w_n)$ as $\mathcal{P}(w_0,\dotsc,w_n)$.
\end{ex}


\medskip

Given a $\QQ$-stack $\mathcal X$, \cite[Lemma 2.7]{AHPP25} explains necessary and sufficient conditions for $\mathcal X$ to be a stacky Proj, which involves definitions from \cite[\S 2]{stable_with_twist}; so let us summarize those definitions over $\Spec F$.
A separated $F$-stack $\mathcal X$ of locally finitely presentation over $\Spec F$ is called a {\it cyclotomic stack} if for every geometric point $\overline{x}$ of $X$, the stabilizer group space $\Aut_{\mathcal X}(\overline{x})$ (or $\Aut(\overline{x})$ for short) is isomorphic to $(\mu_n)_{\kappa(\overline{x})}$ for some $n \ge 1$.
A line bundle $\mathcal{L}$ on a cyclotomic $F$-stack $\mathcal X$ is {\it uniformizing} if the associated classifying morphism $\mathcal X \to B\Gm$ is representable; equivalently, for every geometric point $\overline{x}$ of $\mathcal X$, $\Aut(\overline{x})$-action on $\mathcal{L}_{\overline x}$ is faithful (c.f. \cite[Proposition 2.3.10]{stable_with_twist}).
Finally, a uniformizing line bundle $\mathcal L$ on a proper cyclotomic $F$-stack $\mathcal X$ is {\it polarizing} if there exists $m \ge 1$ such that $\mathcal L^m \cong c^*M$, where $c : \mathcal X \to X$ is the coarse moduli map and $M$ is an ample line bundle on $X$. A mild generalization of \cite[Lemma 2.7]{AHPP25} is as follows:

\begin{lem}\label{lem:stacky_proj}
	Suppose that $\mathcal X$ is a proper geometrically connected cyclotomic $F$-stack with a polarizing line bundle $\mathcal L$. Then $\mathcal X \cong \sProj R_{\mathcal L}$, where $R_{\mathcal L} \colonequals \oplus_{n \ge 0} H^0(\mathcal L^n)$ is called the section ring of $\mathcal L$.
\end{lem}

\begin{proof}
	The proof of \cite[Lemma 2.7]{AHPP25} follows when $\QQ$ is replaced by $F$. 
\end{proof}

\begin{rem}\label{rem:closed_embed_wtproj_stack}
	Suppose that $\mathcal X$ satisfies the conditions of Lemma~\ref{lem:stacky_proj} and the graded $F$-algebra $R_\mathcal{L}$ admits a graded resolution by a free graded $F$-algebra $R$ generated in positive degrees.
	Then the graded resolution $R \twoheadrightarrow R_\mathcal{L}$ induces a closed embedding of $\mathcal X$, induced by $\mathcal L$ into a weighted projective stack $\mathcal{P}(w_0,\dotsc,w_n) \cong \sProj R$ for appropriate $w_0,\dotsc,w_n >0$; in other words, the pullback of $\OO_{\mathcal{P}(w_0,\dotsc,w_n)}(1)$ on $\mathcal X$ is $\mathcal L$.
	
	If the generic stabilizer of $\mathcal X$ is $\mu_d$, then this closed embedding must factor through a coordinate substack $\mathcal{P}(w_{j_0},\dotsc,w_{j_m})$, where $\{j_0,\dotsc,j_m\} \subset \{0,\dotsc,n\}$ such that $d$ divides $w_{j_i}$ for every $0 \le i \le m$.
	As a result, $R_{\mathcal L}$ is $d\NN$-graded.
\end{rem}
In \cite[Appendix A]{Abramovich_Olsson_Vistoli}, the notion of \textit{rigidification} of a stack is introduced, which has many applications in the theory of algebraic stacks.
Given a geometrically connected cyclotomic $F$-stack $\mathcal X$, the inertia stack $\mathcal I_{\mathcal X} \to \mathcal X$ (called $\mathcal I \mathcal X$ in loc. cit.) is a proper subgroup $\mathcal X$-scheme of $(\Gm)_{\mathcal X}$. 
So the generic stabilizer group scheme of $\mathcal X$ extends to a flat subgroup $\mathcal X$-scheme $(\mu_n)_{\mathcal X}$ for some $n \ge 1$.
Then for every $\ell \ge 1$ that divides $n$, \cite[Theorem A.1]{Abramovich_Olsson_Vistoli} implies that the \textit{$\mu_\ell$-rigidification} $\mathcal X \thickslash (\mu_\ell)_{\mathcal X}$ (or $\mathcal X \thickslash \mu_\ell$ for short) is the stackification of the prestack (as a fibered category), whose objects are objects of $\mathcal X$ and whose morphisms are $\mu_\ell$-equivalence classes of morphisms of $\mathcal X$.
In other words, when $\mathcal X \cong \sProj \oplus_{d \in n\NN}  R_d$, then $\mathcal X \thickslash \mu_\ell \cong \sProj \oplus_{d \in \frac{n}{\ell}\NN} R_{d \ell}$ when $\ell$ divides $n$.

\begin{rem}\label{rem:...}
	When $\mathcal X$ is a cyclotomic $F$-stack with the generic stabilizer $\mu_n$, then we define \textit{rigidification} $\mathcal X^{\rm rig}$ of $\mathcal X$ to be $\mathcal X \thickslash \mu_n$. 
	In this case, if $\mathcal L$ is a uniformizing line bundle on $\mathcal X$, then let $\mathcal L^{\rig}$ be a line bundle on $\mathcal X^{\rig}$ where the pullback of $\mathcal L^{\rig}$ by the rigidification $\mathcal X \to \mathcal X^{\rig}$ is isomorphic to $\mathcal L^n$.
	Then $\mathcal L^{\rig}$ is a uniformizing line bundle on $\XX^{\rig}$, which is also polarizing if $\mathcal L$ is a polarizing line bundle on $\XX$.
\end{rem}

\subsection{Root stacks}\label{subsec:root_stacks}

Here, we summarize the notion of root stacks to describe several classes of cyclotomic $F$-stacks as root stacks. The definitions are from \cite[Appendix B]{AGV08} (also independently developed by \cite{impose_tangency}), but because the definition is local on lisse-\'etale sites, we will state the construction on stacks.

Given an algebraic $F$-stack $\XX$ and a line bundle $\mathcal L$, the {\it $n^{\rm th}$ root stack} of $\mathcal L$ on $\XX$ is defined as the fiber product $\sqrt[n]{\mathcal L/\XX} \colonequals \XX \times_{B\Gm} B\Gm$ of $\XX \to B\Gm$ corresponding to $\mathcal L$ and the $n^{\rm th}$-power morphism $B\Gm \to B\Gm$. 
Then for every $F$-scheme $T$, any $T$-point of $\sqrt[n]{\mathcal L/\XX}$ is given by a triple $(x,M,\varphi)$ where $x \colon T \to \XX$ is an $F$-morphism, $M$ is a line bundle on $T$, and $\varphi \colon M^n \to x^*\mathcal L$ is an isomorphism.
Observe that $\sqrt[n]{\mathcal L/\XX}$ is a $\mu_n$-gerbe over $\XX$.

Given an effective Cartier divisor $\mathcal D$ on $\XX$, pick a global section $s_{\mathcal D}$ of $\OO_{\XX}(\mathcal D)$ that cuts out $\mathcal D$. 
Then the {\it $n^{\rm th}$ root stack} of $\mathcal D$ on $\XX$ is defined as the fiber product $\sqrt[n]{\mathcal D/\XX} \colonequals \XX \times_{[\AAA^1/\Gm]} [\AAA^1/\Gm]$ of $\XX \to [\AAA^1/\Gm]$ corresponding to $(\mathcal L,s_{\mathcal D})$ and the $n^{\rm th}$-power morphism $[\AAA^1/\Gm] \to [\AAA^1/\Gm]$; 
the isomorphism class of $\sqrt[n]{\mathcal L/\XX}$ is independent of the choice of $s_{\mathcal D}$.
Similarly, given a fixed choice $s_{\mathcal D}$, a $T$-object of $\sqrt[n]{\mathcal D/\XX}$ is instead a quadruple $(x,M,u,\varphi)$ where $u \in H^0(M)$ and $\varphi(u^n)=x^*s_{\mathcal D}$.
Finally, when $\mathcal D_1, \dotsc, \mathcal D_m$ are mutually disjoint effective Cartier divisors of $\XX$, then $\sqrt[n_1,\dotsc,n_m]{\mathcal D_1,\dotsc,\mathcal D_m/\XX}$ is the iterative root stack $\sqrt[n_1]{\mathcal D_1/\sqrt[n_2]{\mathcal D_2/\dotsb /\sqrt[n_m]{\mathcal D_m/\XX}}}$; note that the isomorphism class of this stack remains the same under permutation of $(n_1,\mathcal D_1),\dotsc, (n_m,\mathcal D_m)$.
Because our stacks of interests are DM stacks of dimension one, the mutually disjoint condition of divisors $\mathcal D_1, \dotsc, \mathcal D_m$ is not restrictive.

\medskip

Let's relate root stack construction to rigidification of cyclotomic $F$-stacks:

\begin{lem}\label{lem:root_stack_of_rigidification}
	Let $\XX$ be a cyclotomic $F$-stack with the generic stabilizer $\mu_n$, and let $\mathcal L$ be a uniformizing line bundle on $\XX$.
	Then $\mathcal L$ induces an isomorphism $\XX \xrightarrow{\sim} \sqrt[n]{\mathcal L^{\rig}/\XX^{\rig}}$.
\end{lem}
\begin{proof}
	Observe that the rigidification morphism $\pi \colon \XX \to \XX^{\rig}$ is a $\mu_n$-gerbe by definition.
	Because $\pi^*\mathcal L^{\rig} \cong \mathcal L^n$, this isomorphism induces a morphism $h \colon \XX \to \sqrt[n]{\mathcal L^{\rig}/\XX^{\rig}}$ of $\mu_n$-gerbes over $\XX^{\rig}$.
	Then for every smooth $F$-covering $U \to \XX^{\rig}$ by an $F$-scheme $U$, the pullback $h_U \colon \XX \times_{\XX^{\rig}} U \to \sqrt[n]{\mathcal L^{\rig}/\XX^{\rig}} \times_{\XX^{\rig}} U$ is an isomorphism by \cite[Lemma 12.2.4]{Olsson16}. 
	Therefore, $h$ is also an isomorphism by descent.
\end{proof}

\begin{cor}\label{cor:rigid-by-pol-LB_base-change}
	Let $\XX$ and $\mathcal Y$ be geometrically irreducible cyclotomic $F$-stacks with generic stabilizers $\mu_n$.
	Suppose that $f \colon \XX \to \mathcal Y$ is a representable dominant morphism that induces isomorphism of generic stabilizers.
	Let $f^{\rig} \colon \XX^{\rig} \to \mathcal Y^{\rig}$ be the induced morphism of rigidifications.
	If $\mathcal L$ is a uniformizing line bundle on $\mathcal Y$, then the following diagram is Cartesian:
	\[\begin{tikzcd}
		\XX \arrow[r, "f"] \arrow[d, "\pi_{\XX}"']
		& \mathcal{Y} \arrow[d, "\pi_{\mathcal{Y}}"] \\
		\mathcal{X}^{\rig} \arrow[r, "f^{\rig}"']
		& \mathcal{Y}^{\rig}
	\end{tikzcd}\]
\end{cor}
\begin{proof}
	Observe that $f^{\rig}$ is dominant and representable.
	By Lemma~\ref{lem:root_stack_of_rigidification}, $\mathcal Y \cong \sqrt[n]{\mathcal L^{\rig}/\mathcal Y^{\rig}}$ and $\pi_{\mathcal Y}$ is a $\mu_n$-gerbe.
	As $f^*\mathcal L$ is a uniformizing line bundle on $\XX$, $\pi_\XX$ is a $\mu_n$-gerbe with $\XX \cong \sqrt[n]{(f^*\mathcal L)^{\rig}/\XX^{\rig}}$ by Lemma~\ref{lem:root_stack_of_rigidification} again.
	Note that $(f^*\mathcal L)^{\rig} \cong (f^{\rig})^*\mathcal L^{\rig}$, which implies that $\XX \cong \sqrt[n]{(f^{\rig})^*\mathcal L^{\rig}/\XX^{\rig}}$.
	The formation of root stack commutes with base change, so $\XX \cong \XX^{\rig} \times_{\mathcal Y^{\rig}} \sqrt[n]{\mathcal L^{\rig}/\mathcal Y^{\rig}} \cong \XX^{\rig} \times_{\mathcal Y^{\rig}} \mathcal Y$.
\end{proof}

\subsection{Sectors of some Deligne--Mumford curves}
\subsubsection{}We start by recalling definition of sectors and their ages from \cite{dardayasudabm}. Let~$X$ be a finite type smooth Deligne--Mumford stack over~$X$.
\begin{mydef}
We denote by $\mathcal J_0X$ its cyclotomic inertia stack, which parametrizes pairs $(x,\mu_k\hookrightarrow\Aut(x))$ for some $k\geq 1$. A connected component of~$\mathcal J_0X$ is called a sector of~$X$. The connected component corresponding to $k=1$ is called the untwisted sector, while the other components are called twisted sectors.
\end{mydef}
\begin{mydef} Let $d\geq 1$ be an integer. 	Let~$\rho$ be a $d$-dimensional representation of a cyclic group~$\mu_n$, where $n\geq 1$. There exists unique integers $a_1\doots a_n\geq 1$ such that $\rho\cong \tau^{a_1}\oplus\cdots \tau ^{a_d}$, with~$\tau$ given by the embedding $\mu_n\hookrightarrow F^{\times}, \zeta_n\mapsto \zeta_n.$ We define a number $$\age(\rho):=\frac{1}{n}\sum_{i=1}^d a_i.$$
\end{mydef}
\begin{mydef}
	Consider a vector bundle~$V$ on~$X$. Given an algebraically closed $\overline F$-field~$K$ an
	 $x\in (\mathcal J_0X)(K)$ corresponding to $(x',\mu_{n}\hookrightarrow\Aut(x'))$, we denote by~$\rho(x)$ the induced action of $\mu_n$ on~$V(x)$. We define $$|\mathcal J_0X|\to \QQ_{\geq 0},\hspace{1cm}x\mapsto \age(\rho(x)).$$ This by \cite[Lemma 2.22]{dardayasudabm}, induces a function$$\age_V:\pi_0\mathcal J_0X\to\QQ_{\geq 0}.$$
	If~$V=T_X$, we write $\age$ for $\age_{T_X}$.
\end{mydef}
\subsubsection{}The following language is convenient when working with functorial properties.  Let us set $$\widehat{\mu}=\varprojlim_n\mu_n$$
where the transition maps are given by $k|m$ then $\mu_m\to\mu_k$ is given by $x\mapsto x^{m/k}$. One has that~$\widehat{\mu}$ has a structure of a group scheme. By \cite[Proposition 2.11]{dardayasudabm}, one has that the stack $\mathcal J_0X$ is canonically equivalent with the stack parametrizing the pairs $(x, \widehat{\mu}\to \Aut(x))$ of a point~$x$ of~$X$ and a homomorphism $\widehat{\mu}\to \Aut(x)$. With this description it is immediate \cite[Corollary 2.14]{dardayasudabm}, that given a morphism $f:Y\to X$ of smooth finite-type Deligne--Mumford stacks, we obtain an induced pointed map $\pi_0\mathcal J_0Y\to\pi_0\mathcal J_0X$.
\begin{mydef}\label{defofagelimproj} 
	Let $k$ be an algebraically closed field over $\overline F$. 
	A finite dimensional $k$-representation of $\widehat{\mu}$ is an $k$-homomorphism $\rho:\widehat{\mu}\to \GL_r$, for certain $r\geq 1$ called the dimension of representation, which factorizes through the canonical homomorphism $\widehat{\mu}\to\mu_{\ell}$ for some $\ell\geq 1$. 
	Let $\rho_{\ell}:\mu_{\ell}\to \GL_r$  be the induced representation. We define $$\age(\rho):=\age(\rho_{\ell}).$$
\end{mydef}
\begin{lem}Let $\rho$ be an $r$-dimensional $k$-representation. One has that $\age(\rho)$ does not depend on the choice of~$\ell$.
\end{lem}
\begin{proof}Choose~$\ell$ as in Definition~\ref{defofagelimproj} minimal possible.. Suppose that $\ell_0$ is such that $\rho$ factorizes through $\widehat{\mu}\to\mu_{\ell_0}$. Then $\ell|\ell_0$. We compare $\age(\rho_{\ell_0})$ and $\age(\rho_{\ell})$. Let $\tau_{\ell}:\mu_{\ell}\hookrightarrow\GL_1$ and $\tau_{\ell_0}:\mu_{\ell}\hookrightarrow\GL_1$ be given by the inclusions. Write $\rho_{\ell}$ as $$\rho_{\ell}\cong\bigoplus_{i=1}^r\tau_{\ell}^{a_i}.$$ Then $$\rho_{\ell_0}\cong \bigoplus\tau_{\ell_0}^{a_i\cdot (\ell_0/\ell)}.$$
	We have that $$\age(\rho_{\ell})=\frac{1}{\ell}\sum_{i=1}^r a_i=\frac{1}{\ell_0}\sum_{i=1}^r a_i\cdot (\ell_0/\ell).$$
	The claim follows.
\end{proof}
\begin{lem}\label{pullbackage}
	Suppose that $f:X\to Y$ is a morphism of smooth finite type $F$-Deligne--Mumford stacks. Let $\mathcal X\in\pi_0\mathcal J_0X $ and let $\mathcal Y=(\pi_0\mathcal J_0f)(\mathcal X).$ Let~$V$ be a vector bundle on~$Y$. Then $$\age_{f^*V}(\mathcal X)=\age_V (\mathcal Y).$$
\end{lem}
\begin{proof}
	Let~$k$ be an algebraically closed field over $\overline F$ such that there exists a $k$-point $x\in\mathcal X$.
	The point~$x$ corresponds to a point $x\in X(k)$ and a $k$-homomorphism $\widehat{\mu}\to \Aut(x).$ We have a commutative diagram
	\[\begin{tikzcd}
		{\Aut(x)} & {\GL(f^*V(x))=\GL(V(f(x)))} \\
		{\Aut(f(x))} & {\GL(V(f(x)))}
		\arrow[from=1-1, to=1-2]
		\arrow[from=1-1, to=2-1]
		\arrow["{=}", from=1-2, to=2-2]
		\arrow[from=2-1, to=2-2]
	\end{tikzcd}\]
	Let~$\rho$ be the induced representation $\rho:\widehat{\mu}\to \GL((f^*V)(x))$ from $\Aut(x)\to \GL((f^*V)(x))$. Then $\rho$ coincides with the representation $\Aut(x)\to\Aut(f(x))\to\GL(V(f(x)))$. We obtain $$\age_V(\mathcal X)=\age_{f^*V}(x)=\age({\rho})=\age_V(f(x))=\age_V(\mathcal Y)$$
	as claimed.
\end{proof}
\begin{cor}\label{etaleage}
Suppose we have $f:X\to Y$ as in Lemma~\ref{pullbackage} which is also \'etale at a point $x$ of $X$. Let $\mathcal X\in\pi_0\mathcal J_0X$ be any sector above $x\in |X|$, and let $\mathcal Y=(\pi_0\mathcal J_0f)(\mathcal X)$. Then $$\age(\mathcal X)=\age (\mathcal Y).$$
\end{cor}
\begin{proof}
As $f$ is \'etale at $x$, then $f^*T_Y \cong T_X$ on a neighborhood of $x$. Then the assertion follows from Lemma~\ref{pullbackage}. 
\end{proof}

\subsubsection{Weighted projective stacks} \label{subsubsec:wtproj_stack}
Let $n\geq 1$ be an integer. Suppose that $a_1\doots a_n\geq 1$ are given integers. 
Recall the weighted projective stack $\mathcal P(a_1\doots a_n):=[(\AAA^n-\{0\})/\Gm]$. 
For any $F$-scheme~$S$ and any morphism $f:S\to\PPP(a_1\doots a_n)$, we have a canonical identification $\Aut(f)=\mu_{n(f)}$ where $n(f)|\lcm (a_i)$. 
Fix a bijection $$\delta:\bigcup_{n\geq 1}\Hom(\widehat{\mu},\mu_n)\xrightarrow{\sim}\QQ/\ZZ.$$  
Observe that when $n | k$, then the canonical surjection $\widehat \mu \to \mu_n$ factors through another canonical surjection $\widehat \mu \to \mu_k$.
If $\theta_n \colon \widehat\mu\to\mu_n$ is the canonical surjection, then let $\delta(\theta_n) \colonequals \frac1n \in \QQ/\ZZ$.
In general, a homomorphism $\iota:\widehat{\mu}\to\mu_n$ is a composite of $\theta_n$ and a homomorphism $\mu_n\to \mu_n$ given by $x\mapsto x^k$. Then we define $\delta(\iota)=k/n$.  The image of $\Hom(\widehat{\mu},\mu_n)$ via $\delta$ is the group $(\frac 1n\ZZ/\ZZ)$.

\begin{lem} \label{shift}
	Suppose that $\mathcal Z$ is a closed substack of a weighted projective stack.  Consider a function $$|\mathcal J_0\mathcal Z|\to\QQ/\ZZ,\hspace{1cm}(x,\phi:\widehat{\mu}\to \Aut(x)=\mu_{n(x)})\mapsto \delta{(\phi)}.$$
	This function descends to a function~$\gamma:\pi_0\mathcal J_0\mathcal Z\to\QQ/\ZZ.$
\end{lem}
\begin{proof}
Suppose that $S\to\mathcal J_0\mathcal Z$ is a morphism with~$S$ connected $F$-scheme corresponding to 
$$(f:S\to \mathcal Z, \phi:\widehat{\mu}\to\Aut(f)=\mu_{n(f)}).$$ 
Then the value of the function at any geometric point of~$S$ is~$\delta(\phi)$, so the claim follows.
\end{proof}
If~$z$ is a geometric point of~$\mathcal Z$ realized as $\widetilde z:\Spec(K)\to\mathcal Z$ with~$K$ an algebraically closed $F$-field, we write $\Aut(z)=\Aut(\widetilde z)$.
\begin{lem} \label{secgroup}Denote by $\omega:|\mathcal J_0\mathcal Z|\to|\mathcal Z|$ the canonical morphism. Given~$z\in |\mathcal Z|$, let $$(\pi_0\mathcal J_0\mathcal Z)_z:=\{\mathcal Y\in\pi_0\mathcal J_0\mathcal Z \; | \; \hspace{0,1cm}\exists y \in |\mathcal Y| \text{ s.t. } \omega(y)=z\}.$$
We have a bijection $$\frac{1}{\#\Aut(z)}\ZZ/\ZZ\xrightarrow{\sim}(\pi_0\mathcal J_0\mathcal Z)_z,\hspace{1cm}x\mapsto [z,\delta^{-1}(x)],$$ where $[\cdot]$ denotes the corresponding connected component.
\end{lem}
\begin{proof} 
	For surjectivity of $\gamma$, suppose that $z$ is  realized as $\widetilde z:\Spec(K)\to\mathcal Z$ with~$K$ an algebraically closed $F$-field. We verify that the map is surjective: if $\mathcal Y\in(\pi_0\mathcal J_0\mathcal Z)_z$, let $y\in\mathcal Y$ be such that $\omega(y)=z$. By increasing $K$ if needed, we realize $y$ as a $K$-point. 
	The point $y$ then corresponds to a pair $(y_0,f)$ of a $K$-point of $\mathcal J_0 \mathcal Z$, where $y_0 \cong \widetilde z$ and $f:\widehat{\mu}\to \Aut(y_0)\cong \Aut(\widetilde z)=\Aut(z)=\mu_{\#\Aut(z)}$. 
	Hence $\delta(f)$ maps to $\mathcal Y$, proving surjectivity. 
	
	For injectivity, suppose that $\frac{k_1}{\#\Aut(z)}$ and  $\frac{k_2}{\#\Aut(z)}$ map to the same component. The value of $\gamma$ at these components is $\frac{k_1}{\#\Aut(z)}$ and $\frac{k_2}{\#\Aut(z)}$, respectively. Hence $k_1=k_2$ and so $\gamma$ is injective.
\end{proof}

\section{Modular curves}\label{sec:modular_curve}
We recall some properties of the moduli stacks of our interest.
Let~$\YYY_0(1)$ be the moduli stack of elliptic curves over~$F$.
Let~$\XXX_0(1)$ be the Deligne--Mumford compactification of~$\YYY_0(1)$. 
It is well known that
$$\XXX_0(1)\cong\PPP(4,6),$$
so $\XXX_0(1)^{\rig}\cong \PPP(2,3)$. 
Its coarse moduli space is $X_0(1) \cong \PP^1$. 
The pullback of the line bundle $\OO_{\PP^1}(1)$ is the line bundle $\OO_{\PPP(4,6)}(12)$ on~$\XXX_0(1)$, and the Hodge bundle $\Lambda$ on $\XXX_0(1)$ corresponds to $\OO_{\PPP(4,6)}(1)$, which is a polarizing line bundle.

Let $N\geq 1$ be an integer. The following is a special case of \cite[IV.3.2]{DeligneRapoport}:
\begin{mydef}
Let $\YYY_0(N)$ be the moduli $F$-stack of elliptic curves equipped with cyclic order $N$ subgroup schemes.
\end{mydef}
So for every $F$-scheme $S$, every $S$-point of $\YYY_0(N)$ is of the form $(E \to S, H \subset E[N])$, where $E \to S$ is an elliptic curve defined over $S$ and $H$ is a $S$-flat cyclic subgroup scheme in $E[N]$ of order $N$;
so $H$ is \'etale locally isomorphic to $\ZZ/N\ZZ$ as a group scheme.
Then any point of~$\YYY_0(N)$ has a non-trivial automorphism group containing $\mu_2$ generated by the involution $[-1]$ sending $P \in E$ to $-P$. 

There is a canonical morphism $\YYY_0(N)\to \YYY_0(1)$ given by $(E,\phi)\mapsto E$, which forgets the specified cyclic order $N$ subgroup scheme of $E[N]$.  This morphism is finite and \'etale.
The open immersion $\YYY_0(1) \hookrightarrow \XXX_0(1)$ is affine because it is representable and the induced open immersion $Y_0(1) \hookrightarrow X_0(1)$ of coarse moduli spaces is affine.
Thus, $\YYY_0(N)\to\XXX_0(1)$ is affine as well. 

Using the relative normalization of an affine (so representable) morphism $f \colon \mathcal C \to \mathcal S$ of algebraic stacks as in \cite[Tag 0822]{stacks-project}, we obtain the following definition of $\XXX_0(N)$:
\begin{mydef}[{\cite[IV.3.3]{DeligneRapoport}}] \label{defx0n}
Define $\mathscr X_0(N)$ to be the normalization of $\XXX_0(1)$ with the respect to the affine morphism $\YYY_0(N)\to \YYY_0(1)\to\XXX_0(1)$. 
\end{mydef} Clearly, one can identify $\YYY_0(N)$ with the preimage of $\YYY_0(1)$ in $\XXX_0(N)$.  We denote by $Y_0(N)$ and $X_0(N)$ the coarse moduli spaces of~$\mathscr Y_0(N)$ and $\XXX_0(N)$ respectively. 
The stacks $\YYY_0(N)$ and $\XXX_0(N)$ are smooth by \cite[IV.6.7]{DeligneRapoport}, and their coarse spaces are irreducible (see discussion in \cite[Section 1.1]{YangModular}).  
Thus $\YYY_0(N)$ and $\XXX_0(N)$ are irreducible as well. 

\subsubsection{}
	There exists a unique line bundle $\Lambda^{\rig}$ on $\XXX_0(1)^{\rig}$, such that its pullback for the rigidification morphism $\XXX_0(1)\to\XXX_0(1)^{\rig}$ is $\Lambda^2$. 
	One has a canonical isomorphism $\XXX_0(1)\xrightarrow{\sim}\sqrt[2]{\Lambda^{\rig}/\XXX_0(1)^{\rig}}$ from Lemma~\ref{lem:root_stack_of_rigidification} induced by $\Lambda^2\xrightarrow{=}(\XXX_0(1)\to\XXX_0(1)^{\rig})^*(\Lambda^{\rig}).$ 
	Let $\Lambda_N^{\rig}:=(J_N^{\rig})^*(\Lambda^{\rig})$ and $\Lambda_N:=J_N^*(\Lambda)$. 
	Let $\eta^{\rig}_N:\Spec(\oF(X_0(N)))\to\XXX_0(N)^{\rig}$ be the generic point of the stacky curve $\XXX_0(N)^{\rig}$.
\begin{lem}
	The stack~$\XXX_0(N)$ is isomorphic to $\sProj R_{J_N^*\Lambda}$ and its generic stabilizer is $\mu_2$.
	Furthermore, $\XXX_0(N) \cong \sqrt[2]{\Lambda_N^{\rig}/\XXX_0(N)^{\rig}}$.
\end{lem}
\begin{proof} 
	It is clear that generic stabilizer of $\XXX_0(N)$ is $\mu_2$.
	As $\Lambda$ is a polarizing line bundle on $\XXX_0(1)$, we have $\XXX_0(1) \cong \sProj \Lambda$ by Lemma~\ref{lem:stacky_proj}.
	So $\XXX_0(1) \cong \sqrt[2]{\Lambda^{\rig}/\XXX_0(1)^{\rig}}$.
	Because $\XXX_0(N) \to \XXX_0(1)$ is a finite morphism that induces isomorphism of generic stabilizers, Corollary~\ref{cor:rigid-by-pol-LB_base-change} implies that $\XXX_0(N) \cong \XXX_0(N)^{\rig} \times_{\XXX_0(1)^{\rig}} \XXX_0(1)$.
	Since the formation of root stack commutes with base change, $\XXX_0(N) \cong \sqrt[2]{\Lambda_N^{\rig}/\XXX_0(N)^{\rig}}$ by the definition of the polarizing line bundle $\Lambda_N$.
	Therefore, $\XXX_0(N) \cong \sProj R_{J_N^*\Lambda}$ by Lemma~\ref{lem:stacky_proj}.
\end{proof}

\begin{lem}\label{Zariski-dense}
	The set $\XXX_0(N)\langle F\rangle$ is Zariski-dense in $\XXX_0(N)$.
\end{lem}
\begin{proof}
	There exists a non-empty open $U\subset X_0(N)\cong\PP^1$, such that $\XXX_0(N) \times_{X_0(N)} U \cong U\times B\mu_2$. 
	As $U\langle F \rangle$ is Zarski-dense in $U$, the assertion follows.
\end{proof}

\begin{lem}\label{degcano}
	The degree of $J_N \colon \XXX_0(N)\to\XXX_0(1)$ is equal to the number of cyclic subgroups of order~$N$ of~$(\ZZ/N\ZZ)^2$. This number is $$\kappa(N):=N\cdot\prod_{p|N}\bigg(1+\frac{1}{p}\bigg).$$
\end{lem}
\begin{proof}
The first part is clear. 
Also, the number of cyclic subgroups of order~$N$ is equal to the index $[\SL_2(\ZZ):\Gamma_0(N)]$ where $\Gamma_0(N)$ consists of elements of $\SL_2(\ZZ)$ that's upper-triangular in $\SL_2(\ZZ/N\ZZ)$.
This number is exactly $\kappa(N)$ as above by \cite[Proposition 1.43]{shimura-book}.
\end{proof}

Because $\XXX_0(N)\to\XXX_0(1)$ is representable, by \cite[Theorem 10.1]{Arithmetic_of_Elliptic}, the only points of~$\XXX_0(N)$ where the stabilizer may not be $\mu_2$ lie inside the fibers of $\XXX_0(N)\to\XXX_0(1)$ where the $j$-invariant is $0$ or $1728$;
in those cases, the stabilizer group can be $\mu_6$ or $\mu_4$ respectively.
The following lemma is a slight restatement of \cite[Lemma 3.5]{valcritstack}, where $\sqrt[n]{\Spec R}$ for a DVR $R$ with the special point $0$ means $\sqrt[n]{0/\Spec R}$. 

 \begin{lem}\label{rootstackextension}
	Let $\phi:R\to R'$ be a morphism of DVRs, with ramification index~$e$.  Let $K$ and $K' $ be the the corresponding fields of fractions. Let $n\geq 1$ be an integer.   There exists a unique representable morphism $$ \sqrt[m]{\Spec(R')}\to\sqrt[n]{\Spec(R)}$$
	extending $\Spec(K')\to \Spec(K)$ for some $m\geq 1$. The only~$m$ for which the morphism exists is $m=n/\gcd (n,e)$, and in this case  its restriction to the closed fiber is the morphism $B\mu_m\to B\mu_n$, induced by $$\mu_m\to \mu_n,\hspace{1cm}x\mapsto x^{e/\gcd(n,e)}.$$
\end{lem}
\begin{proof}
	The existence for $m=n/\gcd(n,e)$ is proven in \cite[Proof of Lemma 3.5]{valcritstack}, as well as its description to the closed fiber. The uniqueness is a part of the general result \cite[Theorem 3.4]{valcritstack}.
\end{proof}

We end the paragraph by proving the following lemmas.
\begin{lem}\label{cmsunr}
Suppose that $\widehat{ E}\in\XXX_0(N)(\overline{F})$ satisfies that $j(J_N(\widehat{ E}))\in \{0,1728\}$. The following three properties are equivalent: 
\begin{enumerate}
	\item  the map on the coarse moduli spaces $\PP^1\cong X_0(N)\to X_0(1)\cong \PP^1$ is unramified at~$r(\widehat{ E})$, where $r:\XXX_0(N)\to X_0(N)$ is the coarse moduli map; 
	\item  $J_N$ is \'etale at~$\widehat{ E}$ and $\Aut(\widehat{ E})=\mu_6$ (respectively $\Aut(\widehat E)=\mu_4$) if $j(J_N(\widehat{ E}))=0$ (respectively $j(J_N(\widehat{ E}))=1728$);
 	\item $J_N^{\rig}:\XXX_0(N)^{\rig}\to \XXX_0(1)^{\rig}$ is \'etale at $s_{0,N}(\widehat{ E})$, where $s_{0,N}:\XXX_0(N)\to \XXX_0(N)^{\rig}$ 
 	is the rigidification and  $\Aut(s_{0,N}(\widehat{ E}))=\mu_3$ (resp., $\Aut(s_{0,N}(\widehat E))=\mu_2$) if $j(J_N(\widehat{ E}))=0$ (resp., $j(J_N(\widehat{ E}))=1728$).
\end{enumerate}
\end{lem}
\begin{proof}
The equivalence between (2) and (3) follows because the rigidifications $s_{0,N}:\XXX_0(N)\to\XXX_0(N)^{\rig}$ and $s_1:\XXX_0(1)\to\XXX_0(1)^{\rig}$ are \'etale at~$\widehat{ E}$ and~$J_N(\widehat{ E})$ by \cite[Theorem A.1]{Abramovich_Olsson_Vistoli}. We show the equivalence between (1) and (3). Let $\overline{J_N}:X_0(N)\to X_0(1)$ be the morphism between coarse moduli spaces induced by $J_N$. Let $\widehat{R}$ be the local ring at~$r(\widehat{ E}) $ and $R$ the local ring at $\overline{J_N}(r(\widehat{ E}))$. Let $m_1:=\#\Aut(s_{0,N}(\widehat{ E}))$ and $m_2:=\#\Aut(s_1(J_N(\widehat{ E})))$. We have a commutative diagram
\[\begin{tikzcd}
	{\sqrt[m_1]{\Spec(\widehat R)}} & {\Spec(\widehat{R})} \\
	{\sqrt[m_2]{\Spec(R)}} & {\Spec(R)}
	\arrow[from=1-1, to=1-2]
	\arrow[from=1-1, to=2-1]
	\arrow[from=1-2, to=2-2]
	\arrow[from=2-1, to=2-2]
\end{tikzcd}\]
where the left vertical morphism is induced from~$J_N^{\rig}$ and so representable. By Lemma~\ref{rootstackextension}, this means that $m_1=m_2/\gcd(m_2,e)$ where $e$ is the ramification degree of $R\to \widehat{R}$. By \cite[Page 97]{Schoeneberg}, one has that $e\in\{1,2,3\}$. As~$m_2\in\{2,3\}$, this means that that $\overline{J_N}$ is unramified at~$r(\widehat{ E})$ if and only if $m_1=m_2$, which gives equivalence between (1) and (3). 
\end{proof}

\begin{lem}\label{stacky-points-xon}
There are $$\varepsilon_3(N):=\begin{cases}
0,&\text{ if } 9|N\\
\prod_{p|N}\bigg(1+ \bigg(\dfrac{-3}{p}\bigg)\bigg),&\text{otherwise.}
\end{cases} $$ points in~$\XXX_0(N)(\overline F)$ with automorphism group~$\mu_6$ and $$\varepsilon_2(N):=\begin{cases}
0&\text{ if } 4|N\\
\prod_{p|N}\bigg(1+\bigg(\dfrac{-1}{p}\bigg)\bigg) &\text{otherwise.}
\end{cases}$$ points in~$\XXX_0(N)(\overline F)$ with automorphism group~$\mu_4$. Below we give the table with concrete values:
\begin{center}
		\begin{tabular}{|c|c|c|}
				\hline
				$N$	& $\varepsilon_2(N)$& $ \varepsilon_3(N)$ \\
				\hline
				1	&  1       &     1    \\
				\hline 
				2	&   1       &    0  \\
				\hline 
				3	&    0    &  1        \\
			\hline
			4	&         0  & 0 \\
			\hline
			5	&      2  &     0\\
			\hline
		6	&       0      &  0  \\
			\hline
			7	&       0   &  2    \\
			\hline
		8	&        0    &   0   \\
			\hline
				9	&        0    &   0   \\
			\hline
				10	&        2    &   0   \\
			\hline
			12	&        0    &   0   \\
			\hline
			13	&        2    &   2   \\
			\hline
			16	&        0   &   0  \\
			\hline
			18	&        0   &   0  \\
			\hline
			25	&        2  &   0  \\
			\hline
			\end{tabular}
	\end{center}
\end{lem}
\begin{proof}
	By Lemma~\ref{cmsunr} we need to count the number of unramified points above $j=0$ and $j=1728$. Now the claim follows from \cite[IV, \S~7 - 8]{Schoeneberg}.
\end{proof}
\subsubsection{}
We are ready to prove the description of the sectors of $\XXX_0(N)$ from Theorem~\ref{thmonsectors}.  Let $$N\in\{1,2,3,4,5,6,7,8,9,10,12,13,16,18,25\}.$$  We have a canonical morphism of stacks $J_N:\XXX_0(N)\to\XXX_0(1)$ which is representable.
\begin{proof}[Proof of Theorem~\ref{thmonsectors}]
First, consider $N=1$ case. 
Recall that the Hodge bundle $\Lambda$ induces an identification $\XXX_0(1)\cong \PPP(4,6)$ with $\Lambda $ corresponding to $\mathcal O(1)$.  
The untwisted sector is $(x,\widehat{\mu}\xrightarrow{y\mapsto 1}\Aut(x))$. 
We also have the twisted sector $(x,\widehat{\mu}\to\mu_{12}\xrightarrow{y\mapsto y^6}\mu_2)$ for any geometric point $x$ with $\#\Aut(x)=2$. 
Assume that $\#\Aut(x)=6,$ hence $x=[0,1]$ and $\Aut(x)=\mu_6$. 
We have $6$ such sectors by Lemma~\ref{secgroup}, namely the ones containing $(x, \widehat{\mu}\to\mu_{12}\xrightarrow{x\mapsto x^{2\ell}}\mu_6)$ for $\ell=0, 1, 2, 3, 4, 5$. 
When $\ell=0,3$, $\gamma$-values of the corresponding sectors are precisely $0$ and $1/2$. 
One treats analogously the sectors corresponding to $\#\Aut(x)=4$. 

Let's calculate the ages. 
The anticanonical line bundle on $\PPP(4,6)$ is $\mathcal O(1)^{\otimes 10}$. 
The linearization describing $\mathcal O(1)$ is $(x,y,a)\mapsto (t^4x,t^6y, t^1a)$ on the trivial line bundle $(\mathbb A^2-\{0\})\times \mathbb A^1\to (\mathbb A^2-\{0\})$. 
Then the action of the stabilizer $\mu_6$ of $[0,1]$ on the fiber $\mathcal O(1)([0,1])$ is given by $\zeta_6\cdot a=\zeta_6 a$;
this action corresponds to the canonical inclusion $\mu_6\to\overline F^*=\Gm(\overline F)$. 
Now let's calculate the age of the sector containing $([0,1],\widehat{\mu}\to\mu_{12}\xrightarrow{y\mapsto y^{2\ell}}\mu_6)$. 
When $\ell=0$, the ages with respect to $\mathcal O(1)$ and $\mathcal O(1)^{\otimes 10}$ are $0$. 
When $\gcd(\ell,6)=1$, then the induced inclusion of $\mu_k$ for some $k\geq 1$ to the stabilizer of $[0,1]$ is given by $\mu_6\to\mu_6, x\mapsto x^{\ell}$. 
The age with respect to $\mathcal O(1)$ is thus $\ell/6$. 
It is immediate that the age with respect to $\mathcal O(10)$ is $\{10\ell/6\}=\{4\ell/6\}=\{2\ell /3\}$. 
Assume that $\ell=2$. 
Then $\widehat{\mu}\to\mu_{6}$ is given by $x\mapsto x^{2}$. The induced inclusion $\mu_k\to\mu_6$ is given by $\mu_3\to\mu_6, x\mapsto x$. 
The age with respect to $\mathcal O(1)$ is $1/3$, while the age with respect to $\mathcal O(10)=\{10/3\}=1/3$. 
Similarly, when $\ell=4$, the induced morphism is $\mu_3\to\mu_6, x\mapsto x^{2}$. 
The age with respect to $\mathcal O(1)$ is $2/3$, while the age with respect to $\mathcal O(10)$ is $2/3$. 
Finally, when $\ell=3$, the induced morphism is $\mu_2\to\mu_6$, which has the age $1/2$ with respect to $\mathcal O(1)$ and the age $0$ with respect to $\mathcal O(10)$. 
The case when the automorphism group is of order $4$ is treated similarly. The proof in the case $N=1$ is completed.

We now treat the general case.	
The stack $\XXX_0(N)$ has the untwisted sector, denoted by $0$, whose age is $0$.
Choose $x\in \XXX_0(N)(\oF)$. 
Suppose that $j(J_N(x)) \neq 0,1728$. 
Let $U$ be the complement of $j=0$ and $j=1728$ in $\PP^1 \cong X_0(1)$. 
Then $\Aut(x)$ is $\mu_2$. 
In this case, $x$ maps to a non-stacky point of $\XXX_0(N)^{\rig}$. 
Let $s_2:\widehat{\mu}\to\mu_2$ be the canonical surjection defined over $\Spec(\ZZ)$. 
The points $(x,\widehat{\mu}\xrightarrow{x\mapsto 1}\mu_2)$ and $(x,\widehat{\mu}\xrightarrow{(s_2)_{\oF}}\mu_2)$ of $\mathcal J_0\XXX_0(N)(\oF)$ map to points of $\mathcal J_0\XXX_0(1)(\oF)$ which lie in distinct connected components, so they lie in distinct connected components of $\mathcal J_0\XXX_0(N)$. 
The point $(x,\widehat{\mu}\xrightarrow{x\mapsto 1}\mu_2)$ lies in the sector $0$. We show that $(x_1,\widehat{\mu}\xrightarrow{(s_2)_{\oF}}\mu_2)$ and $(x_2,\widehat{\mu}\xrightarrow{(s_2)_{\oF}}\mu_2)$, with $x_1$ and $x_2$ not mapping to $j=0,1728$, lie in the same component. Indeed, let $\overline{x_1}$ and $\overline{x_2}$ be the images in $U_{\oF}$ of $x_1$ and $x_2$. We have that $\XXX_0(N)|_U\cong U\times B\mu_2$ because $\Pic(U)=\{0\}$.  Consider $(U\xrightarrow{(\Id_U,1)} U\times B\mu_2\to \XXX_0(N),((\widehat{\mu})_U\xrightarrow{(s_2)_U}(\mu_2)_U)$. Restricting it to $\overline{x_i}$ gives $(x_i,\widehat{\mu}\xrightarrow{(s_2)_{\oF}}\mu_2)$ for $i=1,2$, which means that they are in the same connected component. We call this sector $1/2$. Now it maps to the sector $1/2$ of $\XXX_0(1)$, and Corrolary~\ref{etaleage} implies that its age is $0$. 

Consider now  a geometric point $x$ above $j=0$ with automorphism group scheme $\mu_2$.  Let $S$ be the closed substack of $\XXX_0(N)$ with automorphism group scheme larger than $\mu_2$. Note that $S$ is finite, i.\ e.\ $S\langle \oF\rangle$ is finite. Let $y$ be another finite closed substack of $\XXX_0(N)$ which is geometrically a singleton, disjoint from $\{x\} \cup S$. Set $S':=S\cup\{y\}$. One has that $V:=(\XXX_0(N)-S')^{\rig}$ is an open subscheme of $\PP^1$ and not equal to $\PP^1$. Hence $\Pic(V)=\{0\}$ and $\XXX_0(N)|_V\cong V\times \PP^1$. Now as above, we have that $(x,(\widehat{\mu})_{\oF}\xrightarrow{s_{\oF}}(\mu)_{\oF})$ is in the sector $1/2$. The same arguments apply if $x$ is a point with automorphism group $\mu_2$ lying above $j=1728$.

Consider now a geometric point $x$ above $j=0$, with automorphism group isomorphic to $\mu_6$. Clearly, the point $(x,\widehat{\mu}\xrightarrow{x\mapsto 1}\Aut(x))$ belongs to the untwisted sector. Consider the point $(x,\widehat{\mu}\xrightarrow{\phi}\Aut(x))$, where $\phi$ is the only homomorphism with the image of size $2$.  Let $S_0:=S-\{x\}$ and let $S_0':=S_0\cup\{y\}$ with $y$ as above. Then $(\XXX_0(N)-S_0')^{\rig}$ is isomorphic to an open substack $\mathcal U$ of the stacky curve which is $\AAA^1_{\oF}$ with stackiness $\mu_3$ at $0$. Let us write $X|_Y$ for the fiber product $X\times_{f,Z}Y$ when $f:X\to Z$ is clear from the context. As $\Pic(\mathcal U)$ is $\mathbb Z/3\mathbb Z$ \cite[Theorem 1.1]{lopez2023picardgroupsstackycurves} in which every element is a square, we obtain that $\XXX_0(N)|_{(\XXX_0(N)-S_0')^{\rig}}\cong\mathcal U\times B\mu_2$ so that the point $x$ corresponds to $(0,1)$. The homomorphism $\phi$ becomes $(\widehat{\mu}\xrightarrow{x\mapsto (1,s_2(x))}\mu_3\times \mu_2).$ We remove finitely many non-stacky points of $\mathcal U$, if needed, so that $\mathcal U$ admits an \'etale cover by $U'$ which is a non-empty open subscheme of $\AAA^1$. We have a morphism $a:U'\to \mathcal U\to\mathcal U\times B\mu_2\to \XXX_0(N)$ and $\Aut(a)=\mu_2$. Now $(x,\phi)$ is a restriction of $(a, (s_2)_{U'}) $ and so is $(x',(s_2)_{\oF})$ for some $x'$ with $\Aut(x')\cong \mu_2$. We deduce that $(x,\phi)$ belongs to the sector $1/2$.

Assume that $x$ is a geometric point of $\XXX_0(N)$ with automorphism group scheme isomorphic to $\mu_6$ and that $\widehat{\mu}\to \Aut(x)$ has image of at least $3$ elements. The induced homomorphism $\Aut(x)\to \Aut(J_N(x))\cong\mu_6$ is an isomorphism.  We deduce from Lemma~\ref{secgroup} that the sectors of $(x, \widehat{\mu}\to \Aut(x))$ are distinct and moreover distinct from the sectors $0$ and $1/2$. If $N=3$, there exists precisely one possibility for $x$, so the sectors are described in this case. If $N=7,13$ we have $2$ geometric points of $\XXX_0(N)$ with automorphism group scheme isomorphic to $\mu_6$. Moreover, they are defined over $F$ if and only if $\omega\in F$, otherwise they form a connected closed substack of dimension $0$. We deduce that in the case $\omega\in F$, we have two sectors of $\XXX_0(N)$ above the sector $j/6$, where $j\in\{1,2,4,5\}$, and if $\omega\not\in F$ we have a single sector above $j/6$. In either case, by Corollary~\ref{etaleage}, the age is equal to the age of the corresponding sector $j/6$ of $\XXX_0(1)$, which is $\{-2j/3\}$ by \cite[Example 5.9]{dardayasudabm}.

Assume that $x$ is a geometric point with automorphism group $\mu_4$. Clearly $(x,\widehat{\mu}\xrightarrow{x\mapsto 1}\Aut(x))$ belongs to the trivial sector. Consider the point $(x,\widehat{\mu}\xrightarrow{\psi}\Aut(x))$, where $\psi$ is the only homomorphism with the image of size $2$.  Let $S_1:=S-\{x\}$ and let $S_1':=S_1\cup\{y\}$ with $y$ as before. Then $(\XXX_0(N)-S_0')^{\rig}$ is isomorphic to an open substack $\mathcal V$ of the stacky curve which is $\AAA^1_{\oF}$ with stackiness $\mu_2$ at $0$. As $\Pic(\mathcal V)=\mathbb Z/2\mathbb Z$ by \cite[Theorem 1.1]{lopez2023picardgroupsstackycurves}, we are in one of the following two situations: $\XXX_0(N)|_{(\XXX_0(N)-S_0')^{\rig}} \cong \mathcal V\times B\mu_2$ or $\XXX_0(N)|_{(\XXX_0(N)-S_0')^{\rig}} \cong \PPP(2,4)|_{\mathcal V}$. In the first case, one can use arguments as before. Assume we are in the second case.  The point $(x,\psi)$ corresponds to the point $([0:1],\widehat{\mu}\xrightarrow{(s_2)_{\oF}}\mu_2\hookrightarrow\mu_4)$. Let $\widetilde{V}$ be the preimage of $\mathcal V$ in $\AAA^{2}-\{0\}$ for $\AAA^2-\{0\})\to \PPP(2,4)\to\PPP(1,2)$. The induced morphism $d:\widetilde V\to\PPP(2,4)|_{\mathcal V}\to \XXX_0(N)$ has automorphism group scheme $\mu_2$. The restriction of $(d,\widehat{\mu}\xrightarrow{(s_2)_{\widetilde V}}\mu_2)_{\widetilde V}) $ at $(0,1)$ is the point $(x,\widehat{\mu}\xrightarrow{\psi}\Aut(x))$ while the restriction at some other point is a point $(x_0,\widehat{\mu}\xrightarrow{(s_2)_{\oF}}\mu_2)_{\oF})$. The associated sector is $1/2$.

The case $x$ has automorphism group scheme isomorphic to $\mu_4$ and $\widehat{\mu}\to\mu_4$ is surjective is treated similarly to the case $\widehat{\mu}\to\mu_6$ with the image of size at least $3$. The proof is completed.
\end{proof}
We may write below for the sector $(\mathcal K_*, t)$ only $t$.
\subsection{Naive height}\label{subsec:naive_height}
 We recall the definition of the naive height $H:\XXX_0(N)\langle F\rangle \to\RR_{>0}$. When $N=1$, these heights are called {\it toric heights} in \cite{darda:tel-03682761}. For $v\in M_F^0$, we define $|\cdot|_v$ by $|x|_v=q_v^{-v(x)} $, where $q_v$ is the norm of the maximal ideal $\mathfrak m_v$ of the ring of integers $\OO_v$ of $F_v$. For $v$ a real place let $|\cdot|_v$ be the usual absolute value on $F_v\cong \RR$. For $v$ a complex place, let $|\cdot|_v$ be the square of the absolute value on $F_v^2\cong \CC$.  The Hodge line bundle defines an isomorphism $\XXX_0(1)\cong \PPP(4,6)$. For $v$ finite, we let $$\mathcal D_v:=\OO_v^2-\mathfrak m_v^4\times \mathfrak m_v^6,$$where $\mathfrak m_v$ is the maximal ideal of the ring of integers~$\OO_v$ of $F_v$. We define a function $r_v:F_v^2-\{0\}\to\ZZ$, given by $$r_v(x_1,x_2)=\max\bigg(\bigg\lceil\frac{-v(x_1)}{4}\bigg\rceil,\bigg\lceil\frac{-v(x_2)}{6}\bigg\rceil\bigg).$$
 It follows from \cite[Lemma 3.3.4.4 and Lemma 4.4.3.1]{darda:tel-03682761} that $\pi_v^{r_v(x_1,x_2)}\cdot (x_1,x_2)\in\mathcal D_v$. We set $$f_v(x_1,x_2)=q_v^{-12r_v(x_1,x_2)}.$$ For $v$ infinite we define $$f_v:F_v^2-\{0\}\to\RR_{\geq 0},\hspace{1cm}(x_1,x_2)\mapsto \max(|x_1^3|_v,|x_2^2|_v).$$
\begin{mydef}\label{naive}
We define a naive height on $\XXX_0(1)$ by $$H: \XXX_0(1)\langle F\rangle\to\RR_{> 0},\hspace{1cm}(x_1,x_2)\mapsto \prod_{v\in M_F}f_v(x_1,x_2).$$
We define a naive height on $\XXX_0(N),$ which by abuse of notation is also denoted by $H$: $$H:\XXX_0(N)\langle F\rangle\to\RR_{>0},\hspace{1cm}x\mapsto H(J_N(x)).$$
\end{mydef}
\begin{rem}
\normalfont By \cite[Example 4.4.3.5]{darda:tel-03682761}, we have when $N=1$ and $F=\QQ$, that $$H(x_1,x_2)=\max(|x_1^3|, |x_2^2|)$$ if the coordinates $x_1,x_2$ are in $\ZZ$ with the condition that for all $p$ prime, one has that $p^4\nmid x_1\text{ or }p^6\nmid x_2$.
\end{rem}

	Let $v$ be a finite place not extending the places $2$ and $3$ and let $K/F_v^{\un}$ be the unique extension of degree $12$. The extension is totally ramified and Galois. The Galois group $\Gal(K/F_v^{\un})$ canonically identifies with $\mu_{12}$. We deduce a canonical action of $\mu_{12}$ on $\Spec(\OO_K)$. Fix $12$-th root of unity $\zeta_{12}$ and a $12$-th root $\sqrt[12]{\pi_v}$ of $\pi_v$. Note that the stack $[\Spec(\OO_K)/\mu_{12}]$ has one closed point and one open point. The automorphism group of the closed points canonically identifies with $\mu_{12}$ while the open point corresponds to the open immersion $\Spec(F^{\un}_v)\to [\Spec(\OO_K)/\mu_{12}]$.  Let $\psi_v:\PPP(4,6)\langle F_v^{\un}\rangle\to \pi_0\mathcal J_0\XXX_0(1)$ be the residue map \cite[Definition 2.17]{dardayasudabm}. Note that the map $\pi_0\mathcal J_0\XXX_0(1)\to (\frac14\ZZ\cup\frac16\ZZ)\cap[0,1[$ defined by $(\mathcal K_4,t)\mapsto t$, by $(\mathcal K_6,t)\mapsto t$, by $1/2\mapsto 1/2$ and $0\mapsto 0$, is a bijection. We may use this map to denote the sectors.
	\begin{lem}
		Suppose that $x=(x_1,x_2)\in\XXX_0(1)\langle F_v^{\un}\rangle\cong\PPP(4,6)\langle F_v^{\un}\rangle$.  One has that $$\psi_v(x_1,x_2)=1-\min(v(x_1)/4,v(x_2)/6)-r_v(x_1,x_2).$$
	\end{lem}
	\begin{proof} One can extend $r_v$ to $\Fv^{\un} $ in the obvious way. We define $\mathcal D_v^{\un}$ to be $(\mathcal O_v^{\un})^2-\mathfrak m_v^{\un})^4\times (\mathfrak m_v^{\un})^6,$ with $\mathcal O_v$ the ring of the integers in $F_v^{\un}$ and $\mathfrak m_v^{\un}$ its maximal ideal.  One has $(x_1^{\circ},x_2^{\circ}):=\pi_v^{r_v(x_1,x_2)}\cdot (x_1,x_2)\in\mathcal D^{\un}_v$ similarly as above. Let $m=\min({v(x_1^{\circ})}/{4},{v(x_2^{\circ})}/{6})$. We  have that
		$$m=\min({v(x_1)}/{4},{v(x_2)}/{6})+r_v(x_1,x_2).$$
		As $r_v(x_1^{\circ},x_2^{\circ})=0$, it is sufficient to show that $$\psi_v(x_1^{\circ},x_2^{\circ})=1-\min(v(x_1^{\circ})/4,v(x_2^{\circ})/6) $$ for $(x_1,x_2)\in\mathcal D_v^{\un}$. 
	We search for character $\rho:\mu_{12}\to \Gm$ corresponding to $[\Spec(\OO_K)/\mu_{12}]\to\PPP(4,6)$ that's induced by a $\rho$-equivariant morphism $\Spec(\OO_K)\to\AAA^2-\{0\}$ given by $(x_1,x_2)$. 
		 
		 Assume that $3v(x_1)<2v(x_2)$. The point $(x_1,x_2)$ is $K$-isomorphic to the restriction of the point $h:\Spec(\OO_K)\to\AAA^2-\{0\}$ where $h^*$ satisfies $X\mapsto (\sqrt[12]{\pi_v})^{-48m} x_1$, $Y\mapsto (\sqrt[12]{\pi_v})^{-72m}x_2$. 
		 We have a {commutative diagram below:} 
		 \[\begin{tikzcd}
		 	{\sqrt[12]{\pi_v}} & {((\sqrt[12]{\pi_v})^{-48m} x_1,(\sqrt[12]{\pi_v})^{-72m}x_2)} \\
		 	{\sqrt[12]{\pi_v}\zeta_{12}} &
		 	{((\sqrt[12]{\pi_v})^{-48m}\zeta_{12}^{-48m} x_1,(\sqrt[12]{\pi_v})^{-72m}\zeta_{12}^{-72m}x_2)}
		  	\arrow["h", maps to, from=1-1, to=1-2]
		 	\arrow[maps to, from=1-1, to=2-1]
		 	\arrow[ maps to, from=1-2, to=2-2]
		 	\arrow["h", maps to, from=2-1, to=2-2]
		 \end{tikzcd}\]  with vertical maps given by $\mu_{12}$ and $\Gm$-actions respectively.
		 We deduce $\rho(\zeta_{12})^4=\zeta_{12}^{-48m}$ and $\rho(\zeta_{12})^6=\zeta_{12}^{-72m}$. Thus $\rho(\zeta_{12})^2=\zeta_{12}^{-24m}$ from where we get that $\rho(\zeta_{12})=\zeta_{12}^{-12m}$ or $\rho(\zeta_{12})=\zeta_{12}^{6-12m}$. Note that as the automorphism group of the point $((\sqrt[12]{\pi_v})^{-48m}\zeta_{12}^{-48m} x_1,(\sqrt[12]{\pi_v})^{-72m}\zeta_{12}^{-72m}x_2)$ given by the reduction $(\AAA^2-\{0\})(\OO_K)\to(\AAA^2-\{0\})(\overline {\kappa_v})$, where $\overline {\kappa_v}$ is the separable closure of the residue field $\kappa_v$ at $v$, is $\mu_6$.  Hence, the induced morphism $\mu_{12}\to \mu_6$ is given by $x\mapsto x^{-12m}$. The associated sector is  $\{-12m/12\}=1-m$
		  The case $3v(x_1)>2v(x_2)$ is treated similarly. Assume that $3v(x_1)=2v(x_2)$. Assume at least on of the inequalities $v(x_1)<4, v(x_2)<6$ is valid, we get either $v(x_1)=v(x_2)=0$ or $v(x_1)=2, v(x_2)=3. $ In the first case the associated sector is the untwisted sector, which is $0$ in our notation. In the second case, as above, we get that $\rho(\zeta_{12})=-1$. This implies that the associated sector is $1/2$. The proof is completed.
	\end{proof} 
{	\begin{lem}
		The element of $\NS_{\orb}(\XXX_0(1)) $ defining the naive height is $(\Lambda, (12-12t)_{t}) $ for $t\in\pi_0\mathcal J_0\XXX_0(1)$.
	\end{lem}}
	\begin{proof}
		Consider the morphism $j:\XXX_0(1)\to X_0(1)\cong\PP^1$. The pullback of the line bundle $\OO(1)$ is the $12$-th power of the Hodge line bundle $\Lambda^{\otimes 12}$. Endow it with the adelic metric which is the pullback of an adelic metric defining a Weil's height corresponding to $\OO(1)$. Consider the section $j^*X$ of $\Lambda^{\otimes 12}$. There exists a finite set of places $S$ containing infinite places and all places above $2$ and $3$, such that for $v\not\in S$, one  has whenever $X\neq 0$ that $$||(j^*X)(x_1,x_2)||_v=||X(j(x_1,x_2))||_v=|x_1^3|_v\cdot \max(|x_1^3|_v,|x_2^2|_v)^{-1}=|x_1|^3_v\cdot q_v^{12\min(v(x_1)/4,v(x_2)/6)}.$$
		As above, by abuse of notation let us write $t$ for the sector $(*,t)$. Write temporarily $q_v=1$ for $v\in S$.  We will show that in the domain $j^*X\neq 0$, the functions $(x_1,x_2)\mapsto H(x_1,x_2)=\prod_v f_v(x_1,x_2)$ and $(x_1,x_2)\mapsto \prod_{v}||(j^*X)(x_1,x_2)||_v^{-1}$ have quotients bounded from below and above by strictly positive constants. Indeed the quotient of the first by the latter is given by  
		$$(x_1,x_2)\mapsto\prod_{v} \frac{{f_v(x_1,x_2)}}{|x_1^3|_v\cdot \max(|x_1^3|_v,|x_2^2|_v)}q_v^{12-12\psi_v(x_1,x_2)}=\prod_{v\not\in S}\frac{q_v^{-12r_v(x_1,x_2)+12-12\psi_v(x_1,x_2)}}{q_v^{12\min (v(x_1)/4, v(x_2)/6)}}\prod_{v\in S}\frac{f_v(x_1,x_2)}{\max(|x_1^3|_v,|x_2^2|_v)^{-1}},$$ which is just $$(x_1,x_2)\mapsto \prod_{v\in S}\frac{f_v(x_1,x_2)}{\max(|x_1^3|_v, |x_2^2|_v}. $$ Now fixing $x_2$ and letting $x_1\to 0 $ gives $\max(|x_1^3|_v, |x_2|_v^2)\to |x_2|_v^2$ and $f_v(x_1,x_2)=q_v^{-12\max(\lceil -v(x_1)/4\rceil,\lceil-v(x_2)/6\rceil)}\to q_v^{-12\lceil -v(x_2)/6\rceil}.$ Similarly fixing $x_1$ and letting $x_2\to 0$ gives $\max(|x_1^3|_v,|x_2^2|_v)\to |x_1^3|_v$ and $f_v(x_1,x_2)\to q_v^{-12\lceil -v(x_1)/4\rceil}.$  
	Hence in the domain $j^*X\neq 0$, we have that the two functions $(x_1,x_2)\mapsto \prod_{v}f_v(x_1,x_2)= H(x_1,x_2)$ and $(x_1,x_2)\mapsto \prod_{v} ||(j^*X)(x_1,x_2)||^{-1}\cdot q_v^{12-12\psi_v(x_1,x_2)}$ have bounded quotients from above and below by strictly positive constants. The latter is the height function in the sense of \cite[Definition 4.3]{dardayasudabm} for the element $(\Lambda, (12-12t)_{t})$ restricted to the domain $j^*X\neq 0$. Analogous conclusions are valid in the domain $j^*Y\neq 0$. The claim follows.
	\end{proof}
		\begin{cor}\label{naivforx0n}
		The naive height $H:\XXX_0(N)\langle F\rangle\to\RR_{\geq 0}$ corresponds to the element $(J_N^*\Lambda, (12-12t)_{k\in J_N^{-1}(t)}) $ where by abuse of notation $J_N$ is used for the canonical map $J_N:\pi_0\mathcal J_0\XXX_0(N)\to\pi_0\mathcal J_0\XXX_0(1)$. If for some $t\in\pi_0\mathcal J_0\XXX_0(1)$ one has that $J_N^{-1}(t)=\emptyset$, the corresponding coordinate is understood to be ignored.
\end{cor}
\subsection{Canonical element}\label{subsec:cano_elt}
\begin{mydef}
Given a smooth Deligne--Mumford stack~$X$ over~$F$, we set $$K_{X,\orb}:=(K_X, (\age(y)-1)_{y\in\pi^*_0\mathcal J_0X})$$ where~$K_X$ is the canonical line bundle on~$X$ and $\age$ is defined in \cite[Definition 2.2]{dardayasudabm}.
\end{mydef}
\begin{lem}\label{degcanonical}
One has that
$$\deg(K_{\XXX_0(N)})=-1+\frac{\varepsilon_3(N)}{3}+\frac{\varepsilon_2(N)}{4}.$$
\end{lem}
\begin{proof}
We have a rigidification morphism $\XXX_0(N)\to\XXX_0(N)^{\rig}$ and $\XXX_0(N)^{\rig}$ has trivial generic stabilizer. Moreover, the coarse moduli morphism $\XXX_0(N)\to X_0(N)$ factorizes through $\XXX_0(N)^{\rig}$. Applying \cite[Page 58]{VZB} gives: 
\begin{align*}\deg(K_{\XXX_0(N)^{\rig}})&=\deg(K_{X_0(N)})+\sum_{x\in\XXX_0(N)^{\rig}}\bigg(1-\frac{1}{\# \Aut(x)}\bigg) .\\
	&=-2+\sum_{x\in\XXX_0(N)^{\rig}(\overline{F})}\bigg(1-\frac{1}{\# \Aut(x)}\bigg).
	\end{align*}
The rigidification $p:\XXX_0(N)\to\XXX_0(N)^{\rig}$ is \'etale by \cite[Theorem A1]{Abramovich_Olsson_Vistoli}, hence $K_{\XXX_0(N)}=p^*K_{\XXX_0(N)^{\rig}}$ and thus $\deg(K_{\XXX_0(N)})=\frac12\cdot\deg(K_{\XXX_0(N)^{\rig}})$. 
We obtain \begin{align*}\deg(K_{\XXX_0(N)})&=\frac{1}{2}\bigg(-2+\sum_{x\in\XXX_0(N)^{\rig}(\overline{F})}\bigg(1-\frac 1{\#\Aut(x)}\bigg)\bigg)\\
	&=-1+\sum_{x\in\XXX_0(N)^{\rig}(\overline{F})}\bigg(\frac12-\frac1{2\#\Aut(x)}\bigg)\\
	&=-1+\sum_{x\in\XXX_0(N)(\overline{F})}\bigg(\frac12-\frac1{\#\Aut(x)}\bigg)\\
	&=-1+\sum_{\substack{x\in \XXX_0(N)(\overline F)\\j(J_N(x))=0}}\frac13+\sum_{\substack{x\in \XXX_0(N)(\overline F)\\j(J_N(x))=1728}}\frac14\\
	&=-1+\frac{\varepsilon_3(N)}3+\frac{\varepsilon_2(N)}4.
\end{align*}
 Now Lemma~\ref{cmsunr} implies that the number of the points in~$\XXX_0(N)(\overline F)$ with $\#\Aut(x)=6 $ is $\varepsilon_3(N)$ while the number of points in~$\XXX_0(N)(\overline F)$ with $\# \Aut(x)=4$ is $\varepsilon_2(N)$.
\end{proof}
\begin{lem}
	The pullback $J_N^*(\Lambda^{\otimes 12})$ for the canonical morphism $J_N:\XXX_0(N)\to \XXX_0(1)$ has degree $$\deg(J_N^*(\Lambda^{\otimes 12}))=\frac{\kappa(N)}{2}.$$
\end{lem}
\begin{proof}
By \cite[VI.4.4]{DeligneRapoport}, one has that $\deg(\Lambda)=\frac1{24}$. Lemma~\ref{degcano} gives $$\deg(J_N^*(\Lambda^{\otimes 12}))=\kappa(N)\cdot \frac{12}{24}=\frac{\kappa(N)}{2}.$$
\end{proof}
\begin{mydef}
The anticanonical naive height is the pullback of the naive height, scaled so that the line bundle is the anticanonical line bundle on~$\XXX_0(N)$.
\end{mydef}
\begin{cor}\label{ant-form}
The anticanonical naive height corresponds to the element 
$$\bigg(1-\frac{\varepsilon_3(N)}3-\frac{\varepsilon_2(N)}4,\bigg(\frac{24(1-y)\cdot (1-\frac{\varepsilon_3(N)}{3}-\frac{\varepsilon_2(N)}{4}}{\kappa(N)}\bigg)_{t\in J_N^{-1}(y)}\bigg),$$
the rational orbifold N\'eron-Severi space of~$\XXX_0(N)$.
If for certain $y\in\pi_0\mathcal J_0\XXX_0(1)$ one has that $J_N^{-1}(y)$ is empty then the corresponding coordinate is understood to be ignored.
\end{cor}
\begin{proof}
	This follows immediately from Corollary~\ref{naivforx0n} and the previous calculations.
\end{proof}
\subsection{Adequate cases}\label{subsec:adequate}
\begin{thm}\label{aandb}
	Let~$c:\pi_0^*\mathcal J_0\XXX_0(N)\to\QQ_{\geq 0}$ be the raising function  corresponding to the anticanonical naive height. 	Suppose that for every $y\in\pi_0^*\mathcal J_0\XXX_0(N)$, one has that $\age(y)+c(y)\geq 1$. We call this case adequate.
	\begin{enumerate}
		\item The cases $N=1,2,3,4,5,6, 8,9$ are adequate.
		\item  The expected $a$- and $b$-invariants in the adequate cases are given in the table below: %
		\begin{center}
			\begin{tabular}{|c|c|c|}
				\hline
				$N$	& $a$& $ b$ \\
				\hline
				1	&  5/6       &     1    \\
				\hline 
				2	&   1/2       &    1  \\
				\hline 
				3	&    1/3    &  2       \\
				\hline
				4	&         1/3  & 1 \\
				\hline
				5	&      1/6  &     5-[F(i):F] \\
				\hline
				6	&       {1/6}      &  {2}  \\
				\hline
			
				8	&       1/6    &   2   \\
				\hline
				9	&        1/6   &   2  \\
				\hline
			%
			\end{tabular}
				\end{center}
		\end{enumerate}
		\end{thm}
		\begin{proof}
			We will use \cite[Conjecture 5.6]{dardayasudabm} for the adequate cases. The $a$-invariant in this expected to be $$a=\frac{1-\varepsilon_3(N)/3-\varepsilon_2(N)/4}{\kappa(N)/2},$$ while the $b$-invariant is $$b:=\{y\in\pi_0^*\mathcal J_0\XXX_0(N)|\hspace{0,1cm}\age(y)+c(y)=1\}+1.$$
			For $\kappa(N)=N\cdot\prod_{p|N}\big(1+\frac1p\big), $ we have the following table:
			$$	\begin{tabular}{|c|c|}
				\hline
				$N$	& $\kappa(N)$ \\
				\hline
				1	&  1           \\
				\hline
				2	&  3           \\
				\hline 
				3	&    4         \\
				\hline
				4	&         6   \\
				\hline
				5	&    6 \\
				\hline
				6	&       12  \\
				\hline
				
				8	&    12   \\
				\hline
				9	&      12  \\
				\hline
				%
			\end{tabular}$$
			 The case $N=1$ is adequate by \cite[Example 5.9]{dardayasudabm}. The corresponding $a$- and $b$-invariants, have been calculated in {\it id.\ cit.\ }Assume that $N\in\{4,6,8,9\}$. From Lemma~\ref{stacky-points-xon}  we see that $\varepsilon_2(N)=\varepsilon_3(N)=0$. Hence, we have only one twisted sector of age equal to~$0$. 
			In particular $\kappa(N)\leq 12$  
			and so $0+24\cdot\frac{1}2\cdot \frac{1}{\kappa(N)}\geq 1$ which by Lemma~\ref{ant-form} implies that~$c$ is adequate. In this case we clearly have $$a=\frac{1-\varepsilon_3(N)/3-\varepsilon_2(N)/4}{\kappa(N)/2}=\frac{2}{\kappa(N)}.$$ The $b$-invariant is clearly~$1$ if $\kappa(N)<12$ and~$2$ if $\kappa(N)=12$.  Let us treat the case~$N=2$. By Lemma~\ref{stacky-points-xon}, we have that $\varepsilon_2(2)=1$ and $\varepsilon_3(2)=0$. In particular, the only $\overline F$-point of~$\XXX_0(2)$ with the automorphism group scheme not equal to~$\mu_2$, has the automorphism group scheme equal to~$\mu_4$ and is defined over~$F$. The $j$-invariant of the corresponding elliptic curve is~$1728$. We deduce that~$\XXX_0(2)$ has~$4$ sectors lying above the sectors $0, 1/4, 1/2, 3/4$ of~$\XXX_0(1)$ and having the same ages as them that is, the ages $0, 1/2, 0, 1/2$, respectively.  For any twisted sector~$t$, we have that $$\age(t)+12\cdot \bigg(\frac{1-1/4}{\kappa(2)/2}\bigg)\cdot (1-J_2(t))=\age(t)+6\cdot (1-J_2(t))=\begin{cases}
				3&J_2(t)=1/2\\
				5&J_2(t)=1/4\\
				2&J_2(t)=3/4.
			\end{cases} $$
			We obtain that the case $N=2$ is adequate, that $a= (1-1/4)/(\kappa(2)/2)=1/2$ and that $b=1$. Suppose that $N=3$. Then $\varepsilon_3(3)=1$, while $\varepsilon_2(3)=0$. In particular, the only $\overline F$-point of $\XXX_0(3)$ with the automorphism group scheme not equal to~$\mu_2$, has the automorphism group scheme equal to~$\mu_6$ and is defined over~$F$. The  $j$-invariant of the corresponding elliptic curve is~$0$.  We deduce that $\XXX_0(3)$ has~$6$ sectors lying above the sectors $0, 1/6, 1/3, 1/2, 2/3, 5/6$ and having the same ages as them, that is the ages $0, 2/3, 1/3, 0, 2/3, 1/3$. For any twisted sector~$t$, we have that
			$$\age(t)+12\cdot \bigg(\frac{1-1/3}{\kappa(3)/2}\bigg)\cdot (1-J_3(t))=\age(t)+4\cdot (1-J_3(t))=\begin{cases}
				4&J_3(t)=1/6\\
				3&J_3(t)=1/3\\
				2&J_3(t)=1/2\\
				2& J_3(t)=2/3\\
				1& J_3(t)=5/6.
			\end{cases} $$
			 We obtain that the case $N=3$ is adequate, that $ a=(1-1/3)/(3-\phi(3)/2)=1/3$ and that $b=2$. Assume now $N=5$. By Lemma~\ref{stacky-points-xon}, we obtain that $\varepsilon_2(5)=2$, while $\varepsilon_3(5)=0$. We deduce that the only $\overline F$-points of~$\XXX_0(5)$ with the automorphism group scheme not equal to~$\mu_2$ have the automorphism group scheme equal to~$\mu_4$ and that the $j$-invariant of corresponding elliptic curves is~$1728$. Moreover, by \cite[Lemma 2.8]{AHPP25}, we have that~$\XXX_0(N)$ is  is isomorphic to the substack  of $\PPP(4,4,2)$ given by $X^2+Y^2=Z^4.$ We deduce that the closed substack given by the points with automorphism group schemes not equal to  $\mu_2,$ that is with the automorphism group scheme isomorphic to~$\mu_4$, is given by $Z=0$. The two $\overline F$-points of this stack are $(i,1,0)$ and $(-i,1,0)$. We deduce that we have only one closed point of~$X_0(5)$ lying above the point $j=1728$ of~$X_0(1)$ if $i\not\in F$ and precisely two points if $i\in F$. Suppose that $i\not\in F$ (respectively, that $i\in F$). We deduce that $\XXX_0(5)$ has $4$  (respectively, $8$) sectors lying above the sectors $0, 1/4, 1/2, 3/4$ of $\XXX_0(1)$ such that above each mentioned sector of~$\XXX_0(1)$ there exists precisely~$1$ sector (respectively, $2$ sectors, except for the sector $1/2$, above which we have only one sector) of~$\XXX_0(5)$. For any twisted sector~$t$, we have that $$\age(t)+12\cdot\bigg(\frac{1-1/2}{\kappa(5)/2}\bigg)\cdot (1-J_5(t))=\age(t)+2\cdot (1- J_5(t))=\begin{cases}
			 5/2,	&J_5(t)=1/4\\
			 1,	& J_5(t)=1/2\\
			 1	&J_5(t)=3/4.
			 \end{cases} $$
			 Hence $a=1/6$. If $i\not\in F$, we have that $b=3$ and if $i\in F$, then $b=4$. The verification has been completed.
		\end{proof}
 \subsubsection{}For the sake of completeness, we express $\XXX_0(2), \XXX_0(3)$ and $\XXX_0(4)$ as root stacks.  One has that $\XXX_0(2)^{\rig}\cong \sqrt[2]{E_{1728}}$ where $E_{1728}:X-1728Y=0$, because $\XXX_0(2)^{\rig}$ has only one stacky point which is defined over $F$. In particular, one has that $\XXX_0(2)^{\rig}\cong\PPP(2,1)$. We deduce from \cite[Section 7.3]{mann} that $\Pic(\XXX_0(2)^{\rig})=\ZZ$. Note that $\deg(\Pic(\XXX_0(2)^{\rig}))=\frac{1}{2}\ZZ,$ as every effective Cartier divisor has degree $k/2$ for some $k\geq 0$. Now $(J_2^{\rig})^*(\Lambda^{\rig})$ is not a square, because its degree is $3/2.$ Taking the root with respect to $(J_2^{\rig})^*(\Lambda^{\rig})$ gives:
 \begin{cor} One has that $\XXX_0(2)\cong \PPP(4,2)$ and, in particular, one has that $\XXX_0(2)\to\XXX_0(2)^{\rig}$ is the only (up to isomorphism) non-trivial square root stack over $\XXX_0(2)^{\rig}\cong \PPP(1,2)$.
 	\end{cor}
 	Similarly, one has that $\XXX_0(3)^{\rig}\cong \PPP(1,3)$ hence $\Pic(\XXX_0(3)^{\rig})=\ZZ$. Note that $\deg(\Pic(\XXX_0(3)^{\rig}))=\frac{1}{3}\ZZ,$ as every effective Cartier divisor has degree $k/3$ for some $k\geq 0$. Now $(J_3^{\rig})^*(\Lambda^{\rig})$ is  a square, because its degree is $2/3, $ which is twice the degree of $\OO(\mathcal K_3^{\rig})$.
 \begin{cor}
 One has that $\XXX_0(3)\cong B\mu_2\times \XXX_0(3)^{\rig}\cong B\mu_2\times \PPP(1,3)$ and, in particular, that $\XXX_0(3)\to\XXX_0(3)^{\rig}$ is the trivial square root stack.
\end{cor}
For $N=4$, we have that $\XXX_0(4)^{\rig}\cong \PP^1$. Hence one has that $\XXX_0(4)\cong \PPP(2,2)$ if and only if $(J_4^*)^{\rig}(\Lambda^{\rig})$ is of odd degree. But its degree is $\kappa(4)/6=1$. We obtain:
\begin{cor}
One has that $\XXX_0(4)\cong \PPP(2,2)$ and, in particular, one has that $\XXX_0(4)\to\XXX_0(4)^{\rig}$ is the only (up to isomorphism) non-trivial square root stack over $\PP^1$.
\end{cor}
	\section{Non-adequate cases}\label{sec:non-adequate}
	In this section we treat the remaining cases.
\subsection{Local situation} \label{subsec:local_situation}
We want to treat the other cases. The choice of line bundle $\Lambda_N^{\rig}$ on $\XXX_0(N)^{\rig}$ induces an isomorphism of formal neighborhoods of $\oF$-points on $\XXX_0(N)^{\rig}$ with $\sqrt[m]{\Spec(\oF[[T]])}$ with $m\in\{1,2,3\}$, where $\sqrt[m]{\Spec(\oF[[T]])}$ denotes the $m$-th root stack with respect to the Cartier divisor $m\cdot(T)$. The closed point of $\sqrt[m]{\Spec(\oF[[T]])}$ has the automorphism group scheme {\it canonically }isomorphic to $\mu_m$. Given formal neighborhood $U$ of an $\oF$-point in $\XXX_0(N)^{\rig}$, the choice of line bundle $\Lambda_N$ gives an isomorphism $J_N^{-1}(U)$ with the root stack $\sqrt{L/\sqrt[m]{\Spec(\oF[[T]])}}$ with $L$ corresponding to $\Lambda_N^{\rig}$.

 Let $n\geq 1$. Given a morphism $\Spec(\oF((T)))\to V$ with $V$ some of the $\oF[[T]]$-Deligne--Mumford stacks from below, we find the extension of the morphism $\sqrt[n]{\Spec(\oF[[T]])}\to V$. 
\begin{lem} Consider the morphism $\AAA^1\to\AAA^1, x\mapsto x^m$. Let $[x^m]:[\AAA^1/\Gm]\to [\AAA^1/\Gm]$ be the induced morphism from $\Gm$-morphism $x\mapsto x^m$, where the action of $\Gm$ on $\AAA^1$ is $(t,x)\mapsto tx$.
	The diagram
	\[\begin{tikzcd}
		{\PPP(1,m)\hspace{0,4cm}} & {[\AAA^1/\Gm]} \\
		{\PP^1} & {[\AAA^1/\Gm]}
		\arrow["{(\mathcal O(1)_{\PPP(1,m)},x)}"', from=1-1, to=1-2]
		\arrow[from=1-1, to=2-1]
		\arrow["{[x^m]}"'', from=1-2, to=2-2]
		\arrow["{(\mathcal O(1),x)}"', from=2-1, to=2-2]
	\end{tikzcd}\]
	is Cartesian.
\end{lem}
\begin{proof}
	To give a morphism to the fiber product consists to give $(f:T\to\PP^1,L,s,\alpha)$ with $f:T\to\PP^1$ a morphism, $(L,s)$ a line bundle over $T$ endowed with $s\in\Gamma(T,L)$ and $\alpha:(L^{\otimes m},s^{\otimes m}) \xrightarrow{\sim}(f^*\OO(1) ,f^*(x))$. Let $X\to T$ be the total space of $f^*\mathcal O(1)$. It is a $\Gm$-torsor for the action induced by $\Gm\xrightarrow{x\mapsto x^{-1}}\Gm$ and the standard action of $\Gm$ on $X$. We choose coordinate system $\PP^1=U\cup V$ with $V:=x\neq 0$ and $U:=y\neq 0$. We have  trivializations $(\OO(1)|_{V},x|_V)\xrightarrow{\sim}(V\times \AAA^1, 1)$ and  $(\OO(1)|_{U},x|_U)\xrightarrow{\sim}(U\times \AAA^1, 1).$ We obtain trivializations  $(X|_{f^{-1}(V)},s|_{f^{-1}(V)})\xrightarrow{\sim}(f^{-1}(V)\times\AAA^1,1)$ and $(X|_{f^{-1}(U)},s|_{f^{-1}(U)})\xrightarrow{\sim}(f^{-1}(U),1).$ We have two $\Gm$-equivariant morphisms  $(f^{-1}(V)\times (\AAA^1-\{0\}))\to (\AAA^1-\{0\})\times (\AAA^1-\{0\})$ given by $(r,t)\mapsto (f(r),t)$ and $(f^{-1}(U)\times (\AAA^1-\{0\}))\to (\AAA^1-\{0\})\times (\AAA^1-\{0\})$ given by $(r,t)\mapsto (f(r)^{-1},t).$  They glue to a morphism $X\to \AAA^2-\{0\}$ which is $\Gm$-equivariant. { Conversely, starting with a morphism $g: T\to \PPP(1,m)$ with $T$ an $F$-scheme, we take  $f:=(\PPP(1,m)\to\PP^1)\circ g$, $L:=g^*(\OO(1)_{\PPP(1,m)})$, $s:=g^*x$ and $\alpha$ induced from commutative diagrams. The two processes are inverse  one to another. } 
\end{proof}
We obtain a Cartesian diagram:
\begin{equation}\label{diaenum}\begin{tikzcd}
		{\sqrt[m]{\Spec(\oF[[T]])} } & {\PPP(1,m)} \\
		{\Spec (\oF[[T]])} & {\PP^1}
		\arrow["{x=0}"', from=1-1, to=1-2]
		\arrow[from=1-1, to=2-1]
		\arrow[from=1-2, to=2-2]
		\arrow["{x=0 }"', from=2-1, to=2-2]
\end{tikzcd}\end{equation} The Cartesian diagram~(\ref{diaenum}) identifies the automorphism group of the closed point of $\sqrt[m]{\Spec(\oF[[T]])}$ with~$\mu_m.$
\begin{lem}Let $n\geq 1$. The canonical map $$\Pic\bigg(\sqrt[n]{\Spec(\overline F[[T]])}\bigg)\to \Pic(B\mu_n)=\Hom(\mu_n,\Gm)=\ZZ/n\ZZ $$ is an isomorphism.  A morphism $\sqrt[m]{\Spec(\oF[[T]])}\to\sqrt[n]{\Spec(\oF((T)))},$ maping the stacky point to the stacky point, induces a morphism $\Pic(\sqrt[n]{\Spec(\oF((T)))})\to \Pic(\sqrt[m]{\Spec(\oF((T)))})$ which is precisely the Cartier dual homomorphism $\ZZ/n\ZZ=\Hom(\mu_n,\Gm)\to\Hom(\mu_n,\Gm)=\ZZ/m\ZZ$ to the homomorphism of automorphism group schemes $\mu_m\to\mu_n$.
\end{lem}
\begin{proof}
	Consider the smooth stacky curve $\mathcal P(1,n)$. Then $\sqrt[n]{\Spec(\overline F[[T]])}$ is the fiber product of the coarse moduli space map and $\Spec(\overline F[[T]])\to\PP^1$ given by the formal neighborhood of $0$ in $\PP^1$. We know by \cite[Theorem 1.1]{lopez2023picardgroupsstackycurves} that $\Pic(\mathcal P(1,n))\to \ZZ/n\ZZ$ is surjective. Clearly this homomorphism factorizes through one in the statement, hence the latter is also surjective. The injectivity follows from \cite[Proposition 6.2]{Olsson12}. The second claim is immediate.
\end{proof}
The diagram~(\ref{diaenum}) induces an equivalence $$[\Spec(\oF[[T]][Y,Y^{-1}])/_{(1,m)}\Gm]\xrightarrow{\sim}\sqrt[m]{\Spec(\oF[[T]])}$$ where the indices denote the weights of the action.  The pullback of the line bundle $\OO(1)_{\PPP(1,m)}$ for $$\sqrt[m]{\Spec(\oF[[T]])}\xrightarrow{x=0}\PPP(1,m)$$ induces a line bundle on $\sqrt[m]{\Spec(\oF[[T]])}$. The square of the line bundle $\OO(1)_{\PPP(2,m)}$ is isomorphic to the pullback of $\OO(1)_{\PPP(2,3)}$. The pullback of the line bundle $\OO(1)_{\PPP(2,m)}$ for the base change $Z:=\sqrt[m]{\Spec(\oF[[T]])}\times _{x=0,\PPP(1,m)}\PPP(2,m) \to\PPP(2,m)$ yields an identification 
\[\begin{tikzcd}
	Z & {[\Spec(\oF[[T]][Y,Y^{-1}])_{(2,2m)}/\Gm]} \\
	{\sqrt[m]{\Spec(\oF[[T]])}} & {[\Spec(\oF[[T]][Y,Y^{-1}])_{(1,m)}/\Gm]}
	\arrow["{=}", from=1-1, to=1-2]
	\arrow[from=1-1, to=2-1]
	\arrow[from=1-2, to=2-2]
	\arrow["{=}", from=2-1, to=2-2]
\end{tikzcd}\]
The following lemmas describe the extension of a morphism $\Spec(\oF((T)))$ to a representable morphism $\sqrt[m]{\Spec(\oF[[T]])}$.
\begin{lem}\label{imagemu2}
	Consider the morphism $\Spec (\oF((T)))\to \Spec(\oF((T)))\times B\mu_2$ corresponding to $f\in T\oF[[T]]-\{0\}$ and $g\in H^1(\oF((T)),\mu_2)$. The morphism extends uniquely to a representable morphism of stacks $$\sqrt[m]{\Spec(\oF[[T]])}\to \Spec(\oF[[T]])\times B\mu_2.$$
	with $m|2$. Moreover,  we have that $m=1$ if and only if $g=0$.
\end{lem}
\begin{proof}
	The existence and uniqueness follow from \cite[Theorem 3.1]{valcritstack} applied to the commutative diagram:
	\[\begin{tikzcd}
		{\Spec(\oF((T)))} & {\Spec(\oF[[T]])\times B\mu_2} \\
		{\Spec(\oF[[T]])} & {\Spec(\oF[[T]]).}
		\arrow[from=1-1, to=1-2]
		\arrow[from=1-1, to=2-1]
		\arrow[from=1-2, to=2-2]
		\arrow[from=2-1, to=2-2]
	\end{tikzcd}\]
	We get a representable morphism $\sqrt[m]{\Spec(\oF[[T]])}\to \Spec(\oF[[T]])\times B\mu_2$ which implies that $m|2$.
		Write $$\Spec(\oF((T)))=[\Spec(\oF((T))[W,W^{-1}])/_{(0,1)}\Gm].$$ Write $\Spec(\oF((X)))\times B\mu_2=[\Spec(\oF((X))[Y,Y^{-1}])/_{(0,2)}\Gm].$ The morphism $(f,g)$ corresponds to $\Gm\to\Gm$-equivariant morphism 
		$$\Spec(\oF((T))[W,W^{-1}])\to\Spec(\oF((X))[Y,Y^{-1}]),\hspace{1cm}X\mapsto f(T), Y\mapsto \widetilde g(T)W^2 $$
		where $\widetilde g(T)=1$ if $g=0$ and $\widetilde g(T)=T$ if $g=1$. Endow $\Spec(\oF((T_0))[W_0,W_0^{-1}]) $ with the action of $\Gm$ given by $t\cdot (T_0,W_0)=(t^{1}T_0,t^{-2}W_0)$.  We have an isomorphism 
		$$\Spec(\oF((T_0))[W_0,W_0^{-1}])\to \Spec(\oF((T))[W,W^{-1}]),\hspace{1cm} T\mapsto T_0^{2}W_0^{1}, W\mapsto T_0.$$
		Composing it with $(f,g)$ we get the morphism
		$$\Spec(\oF((T_0))[W_0,W_0^{-1}])\to \Spec(\oF((X))[Y,Y^{-1}]),\hspace{1cm} X\mapsto f(T_0^{2}W_0^{1}), Y\mapsto \widetilde g(T_0^{2}W_0^{1})T_0^2.$$
		Note that $X\mapsto f(T_0^{2}W_0^{1}), Y\mapsto \widetilde g(T_0^{2}W_0^{1})T_0^2$ defines a $\Gm$-equivariant morphism $\Spec(\oF[[T_0]][W_0,W_0^{-1}])\to \Spec (\oF[[X]][Y,Y^{-1}])$. The map on automorphism group schemes at $T_0=0$ is given by $\mu_2\to\mu_2, W_0\mapsto W_0^{1}$ if $g=1$ and $\mu_2\to\mu_2, W_0\mapsto 1$ if $g=0$. The claim follows.
	\end{proof}
{	We will identify $\mu_2\times\mu_3$ with $\mu_6$ via the isomorphism induced by the canonical inclusions $\mu_2\subset\mu_6$ and $\mu_3\subset\mu_6$. We construct an isomorphism $\gamma:\sqrt[3]{\Spec(\oF[[X_0]])}\times B\mu_2\xrightarrow{\sim} [\Spec(\oF[[X]][Y,Y^{-1}])/_{(2,6)}\Gm].$ Indeed, set~$\gamma$ to be the one given by $\Gm^2\to\Gm, (t_0,p_0)\mapsto t^2_0p_0$-equivariant morphism  $$\Spec(\oF[[X_0]][Y_0,Y_0^{-1},Z_0,Z_0^{-1}])\to \Spec(\oF[[X]][Y,Y^{-1}]),\hspace{1cm}X\mapsto X_0Y_0Z_0  , Y\mapsto Y_0^4Z_0^3.$$
Indeed, let us show that an inverse is provided by $\Gm\to \Gm^{2}, t\mapsto (t^2,t^3)$-equivariant morphism $$X_0\mapsto X, Y_0\mapsto Y, Z_0\mapsto Y.$$ 
Indeed, one of the two composites of the morphisms is given by $(X_0,Y_0,Z_0)\mapsto (X_0Y_0Z_0,Y_0^4Z_0^3,Y_0^4Z_0^3)$ which is 
isomorphic to the identity via multiplying by $(Y_0Z_0,Y_0^2Z_0).$ The other composite is given by $(X,Y)\mapsto (XY^2,Y^7)$ which is isomorphic to the identity via multiplication by $Y.$ Let us find the induced homomorphism on automorphism group schemes of the closed points. Set $X_0=0, Y_0=1, Z_0=1$, the corresponding  automorphism group scheme is given by $\mu_3\times \mu_2$. The image of $(\zeta_3,\zeta_2)$ is given by $\zeta_3^4\zeta_2^2=\zeta_3\zeta_2,$ hence with the above identification, the induced isomorphism is $\mu_6\xrightarrow{=}\mu_6$. }
	\begin{lem}\label{localmu6} Consider the stack $\sqrt[3]{\Spec(\oF[[T]])}$. Let $f\in T\oF[[T]]-\{0\}$ and $g\in H^1(\oF((T)),\mu_2)$. Let $e=\ord_T(f)$. Consider the morphism $$\Spec (\oF((T)))\xrightarrow{(f,g)} \Spec(\oF((T)))\times B\mu_2\to \sqrt[3]{\Spec(\oF[[T]])}\times B\mu_2$$ where the second morphism is the open immersion. The morphism extends uniquely to a representable morphism of stacks $$\sqrt[m]{\Spec(\oF[[T]])}\to \sqrt[3]{\Spec(\oF[[T]])}\times B\mu_2$$
		for $m=6/\gcd(6,e) $ if $g=0$ and for $m=6/\gcd(6,3+e)$ if $g=1$.  The induced homomorphism is given by $\mu_m\to\mu_6, x\mapsto x^e$ if $g=0$ and $x\mapsto x^{e+3}$ if $g=1$. 
	\end{lem}
	\begin{proof}
		The uniqueness  is immediate from the separatedness of the target $\sqrt[3]{\Spec(\oF[[T]])}\times B\mu_2$. We will show that there exists an extension $\sqrt[6]{\Spec(\oF[[T]])}\to\sqrt[3]{\Spec(\oF[[T]])}\times B\mu_2$ and that the map on the automorphism group scheme is given by $x\mapsto x^e$ if $g=0$ and $x\mapsto x^{e+3}$ if $g=1$. 
					%
					We have that $$\Spec(\oF((T)))=[\Spec(\oF((T))[W,W^{-1}])/_{(0,1)}\Gm].$$
					The morphism $f:\Spec(\oF((T)))\to\Spec(\oF((T)))$ corresponds to $\Gm$-equivariant morphism $$\Spec(\oF((T))[W,W^{-1}])\to \Spec(\oF((T))[W,W^{-1}]),\hspace{1cm}T\mapsto f(T), W\mapsto W $$
					Let $\Gm$ act on $\Spec(\oF((X_0))[Y_0,Y_0^{-1}])$ by $t\cdot (X_0,Y_0)=(t^{1}X_0,t^{3}Y_0).$ We have an open immersion \begin{align*}[\Spec(\oF((T))[W,W^{-1}])/_{(0,1)}\Gm]=\Spec(\oF((T)))&\to\sqrt[3]{\Spec(\oF[[X_0]])}=[\Spec(\oF[[X_0]][Y_0,Y_0^{-1}])/_{(1,3)}\Gm],\\X_0\mapsto W, \hspace{0,1cm}& Y_0\mapsto W^3T.
					\end{align*} The morphism $f:\Spec(\oF((T)))\to\Spec(\oF((T)))\to\sqrt[3]{\Spec(\oF[[X_0]])}$ corresponds to $\Gm$-equivariant morphism $$\Spec(\oF((T)))[W,W^{-1}] \to \Spec(\oF[[X_0]][Y_0,Y_0^{-1}]),\hspace{1cm} X_0\mapsto W, Y_0\mapsto W^3f(T)$$ 
					We have an identification $$\sqrt[3]{\Spec(\oF[[X_0]]}\times B\mu_2=[\Spec(\oF[[X_0]][Y_0,Y_0^{-1},Z_0,Z_0^{-1}])/\Gm^2],\hspace{1cm} (t_0,p_0)\cdot(X_0,Y_0,Z_0)=(t_0^{1}X_0,t_0^{3}Y_0,p_0^{2}Z_0).$$
					The morphism $(f,g):\Spec(\oF((T)))\to [\Spec(\oF[[X_0]][Y_0,Y_0^{-1},Z_0,Z_0^{-1}])/\Gm^2]$ corresponds to the $\Gm\to \Gm^2, t\mapsto (t,t)$-equivariant morphism
					$$\Spec(\oF((T))[W,W^{-1}])\to \Spec(\oF[[X_0]][Y_0,Y_0^{-1},Z_0,Z_0^{-1}]),\hspace{1cm}X_0\mapsto W, Y_0\mapsto W^3f(T), Z_0\mapsto \widetilde g(T)W^2,$$
					with $\widetilde g(T)=1$ if $g=0$ and $\widetilde g(T)=T$ if $g=1$.
						The composite $\gamma\circ (f,g)$, where $\gamma$ is defined just before the lemma, corresponds to the $\Gm\to \Gm, t\mapsto t^2$-equivariant morphism
					$$\Spec(\oF((T))[W,W^{-1}]) \to[\Spec(\oF[[X]][Y,Y^{-1}])/\Gm],\hspace{1cm}X\mapsto \widetilde g(T)W^4, Y\mapsto f(T)\widetilde g(T)^3W^{12}.$$
					We have an isomorphism $$[\Spec(\oF((T_1))[W_1,W_1^{-1}])/\Gm]\xrightarrow{\sim}\Spec(\oF((T)))=[\Spec(\oF((T))[W,W^{-1}])/_{(0,1)}\Gm],$$ where the action to the left is given by $t\cdot (T_1,W_1)=(t^{1}T_1,t^{6}W_1)$ induced by the $\Gm$-equivariant isomorphism $$\Spec(\oF((T_1))[W_1,W_1^{-1}])\to \Spec(\oF((T))[W,W^{-1}]),\hspace{1cm}T\mapsto W_1T_1^{-6}, W\mapsto T_1.$$ Composing it with $\gamma\circ (f,g)$ yields the morphism $\Spec(\oF((T_1))[W_1,W_1^{-1}])/\Gm\to [\Spec(\oF((X))[Y,Y^{-1}])/\Gm]$ given by the $\Gm\to\Gm, t\mapsto t^2$-equivariant morphism $$\Spec(\oF((T_1))[W_1,W_1^{-1}])\to  \Spec(\oF((X))[Y,Y^{-1}]),\hspace{0,6cm}X\mapsto \widetilde g(W_1T_1^{-6})T_1^4, Y\mapsto f(W_1T_1^{-6})\widetilde g(W_1T_1^{-6})^3T_1^{12}.$$
					We write $f(T)=T^ef_0(T)$ with $f_0(T)$ not divisible by $T$. The morphism becomes $$\lambda: X\mapsto \widetilde g(W_1T_1^{-6})T_1^4,\hspace{0,2cm} Y\mapsto f_0(W_1T_1^{-6})\cdot (W_1T_1^{-6})^e\widetilde g(W_1T_1^{-6})^3T_1^{12}=\widetilde g(W_1T_1^{-6})^3f_0(W_1T_1^{-6})\cdot W_1^eT_1^{12-6e}.$$
					If $g=0$, the morphism is $2$-isomorphic to 
					$$X\mapsto T_1^{2e}, Y\mapsto f_0(W_1T_1^{-6})\cdot W_1^e$$
					which extends to a morphism $$[\Spec(\oF[[T_1]][W_1,W_1^{-1}])/_{(2,6)}\Gm]\to [\Spec(\oF[[X]][Y,Y^{-1}])/_{(2,6)}\Gm].$$
					$$X\mapsto \widetilde g(W_1T_1^{-6})T_1^{2e}, Y\mapsto \widetilde g(W_1T_1^{-6})f_0(W_1T_1^{-6})\cdot W_1^e$$
					with the induced morphism $\mu_6\to\mu_6$ being given by $W_1\mapsto W_1^e$. Now, we obtain that $m=\gcd(6,6/e)$ and that $\mu_m\to\mu_6$ is given by $x\mapsto x^e$. If $g=1$, that is $\widetilde g(W_1T_1^{-6})=W_1T_1^{-6}$, the morphism $\lambda$ is given by
					$$X\mapsto W_1T_1^{-2}, Y\mapsto W_1^{3+e}T_1^{-18}f_0(W_1T_1^{-6})T_1^{12-6e}.$$
					It is  $2$-isomorphic with
					$$X\mapsto W_1T_1^{2e}, Y\mapsto W_1^{3+e}$$ which extends to a morphism $$[\Spec(\oF[[T_1]][W_1,W_1^{-1}])/_{(2,6)}\Gm]\to [\Spec(\oF[[X]][Y,Y^{-1}])/_{(2,6)}\Gm].$$ The induced morphism on automorphism group schemes is $\mu_6\to \mu_6, W_1\mapsto W_1^{3+e}$ if $g=1$. Now we get that $m=6/\gcd(e+3,6)$ and $\mu_m\to \mu_6$ is given by $x\mapsto x^{e+3}$. The proof is completed.
					\end{proof}
		Thanks to \cite[Theorem 1.1]{lopez2023picardgroupsstackycurves}, there exists a line bundle on $\PPP(1,2)$ which maps to a non-trivial line bundle in $\Pic(\sqrt{\Spec(\oF[[X]])})$. Now it is clear that $\OO(1)_{\PPP(1,2)}$ must be such a line bundle. Write $L:=\OO(1)|_{\sqrt{\Spec(\oF[[X]])}}$ for the restriction. We have that $ \sqrt{\Spec(\oF[[X]])}=[\Spec(\oF[[X]][Y, Y^{-1}])/_{(1,2)}\Gm]$. Consider the stack $[\Spec(\oF[[T]][W,W^{-1}])/_{(2,4)}\Gm]$. Let $L_1=\OO(1)_{[\Spec(\oF[[T]][W,W^{-1}])/_{(2,4)}\Gm]}$ be the restriction of $\OO(1)_{\PPP(2,4)}$. The identity $L_1^{\otimes 2}=p^*L$, where $$p:[\Spec(\oF[[T]][W,W^{-1}])/_{(2,4)}\Gm]\to[ \Spec(\oF[[X]][Y,Y^{-1}])/_{(1,2)}\Gm]$$ is the canonical morphism, induces an isomorphism  $$\sqrt{L/([\Spec(\oF[[X]][Y,Y^{-1}])/_{(1,2)}\Gm])}\xrightarrow{\sim} [\Spec(\oF[[T]][W,W^{-1}])/_{(2,4)}\Gm].$$
		 On the automorphism group schemes (as subgroup schemes of $\Gm$) at the closed point is given by $\mu_4\xrightarrow{=}\mu_4$. 
		 Furthermore, one has an isomorphism $$\rho:\Spec(\oF((T_0)))\times B\mu_2\xrightarrow{\sim}[\Spec(\oF((T))[W,W^{-1}])/_{(2,4)}\Gm].$$
		 Indeed, one can identify $\Spec(\oF((T_0)))\times B\mu_2$ with $[\Spec(\oF((T_0))[W_0,W_0^{-1}]/_{(0,2)}\Gm]$. Then we have a $\Gm$-equivariant isomorphism $\Spec(\oF((T_0))[W_0,W_0^{-1}])\xrightarrow{\sim}\Spec(\oF((T))[W,W^{-1}]$, $(T_0,W_0)\mapsto (W_0,W_0^2T_0)$ where the action with weights $(0,2)$ corresponds to the action with the weights $(2,4)$.
	{		\begin{lem}\label{localmu4} 
		Let $f:=T_0^ef_0(T_0)$ be such that $e\geq 1$ and $f_0$ not divisible by~$T$.  Let $g\in \ZZ/2\ZZ=H^1(\oF((T_0)),\mu_2)$. 
	The morphism $$\Spec(\oF((T_0)))\xrightarrow{(f,g)}\Spec(\oF((T_0)))\times B\mu_2\xrightarrow{\rho}\Spec(\oF((T))[W,W^{-1}])/_{(2,4)}\Gm\subset [\Spec(\oF[[T]][W,W^{-1}])/_{(2,4)}\Gm] $$
	extends uniquely to a representable morphism $$\lambda:\sqrt[m]{\Spec(\oF[[T_0]])}\to [\Spec(\oF[[T]][W,W^{-1}])/_{(2,4)}\Gm],$$ for $m=4/\gcd(e,4)$ if $g=0$ and for $m=4/\gcd (e+2,4)$ if $g=1$. The induced homomorphism $\mu_m\to\mu_4$ is given by $x\mapsto x^e$ if $g=0$ and $x\mapsto x^{e+2}$ if $g=1$ 
  					\end{lem}}
						\begin{proof}
						The uniqueness is clear. With identification $\Spec(\oF((T_0)))=[\Spec(\oF((T_0))[X,X^{-1}])/_{(0,1)}\Gm]$, the morphism $ (f,g)$ becomes the one induced by the $\Gm$-equivariant morphism
						$$[\Spec(\oF((T_0))[X,X^{-1}])/_{(0,1)}\Gm]\rightarrow [\Spec(\oF((T_0))[W_0,W_0^{-1}])/_{(0,2)}\Gm],\hspace{0,7cm} T_0\mapsto f(T_0), W_0\mapsto \widetilde g(T_0)X^2,$$
						where $\widetilde g(T_0)=1$ if $g=1$ and $\widetilde g(T_0)=T_0$ if $g\neq 1$. 
						The morphism $\rho\circ (f,g)$ is given by the $\Gm$-equivariant morphism $$[\Spec(\oF((T_0))[X,X^{-1}])/_{(0,1)}\Gm]\rightarrow [{\Spec(\oF((T))[W,W^{-1}])/_{(2,4)}\Gm} ], \hspace{0,4cm}T\mapsto \widetilde{g}(T_0)X^2 , W\mapsto \widetilde g(T_0)^2 X^4 f(T_0) .$$ 
							Consider the stack $[\Spec(\oF((T_1))[W_1,W_1^{-1}])/\Gm]$ where the action is $t\cdot (T_1,W_1)=(t^{1}T_1,t^{4}W_1).$ This stack is in fact a scheme, and we will provide an isomorphism $$[\Spec(\oF((T_1))[W_1,W_1^{-1}])/_{(1,4)}\Gm]\xrightarrow{\sim}\Spec \oF((T_0))=[\Spec (\oF((T_0))[W_0,W_0^{-1}])/_{(0,1)}\Gm].$$ Indeed, the isomorphism is provided by the $\Gm$-equivariant isomorphism $$\Spec(F((T_1))[W_1,W_1^{-1}])\xrightarrow{\sim}\Spec(\oF((T_0))[X,X^{-1}]),\hspace{1cm} T_0\mapsto T_1^{-4}W_1^{1}, X\mapsto T_1.$$
							Now $[\Spec(\oF((T_1))[W_1,W_1^{-1}])/_{(1,4)}\Gm]\to [\Spec(\oF((T))[W,W^{-1}]/_{(2,4)}\Gm]$ is given by the $\Gm$-equivariant morphism 
						$$T\mapsto \widetilde g(T_1^{-4}W_1) T_1^{2}, W\mapsto  \widetilde g(T_1^{-4}W_1)^2T_1^{4}f(T_1^{-4}W_1).$$
							We have  $f(T_0)=T_0^ef_0(T_0)$ with $f_0(T_0)$ not divisible by $T_0$. The morphism is $1$-isomorphic to:
							$$T\mapsto    \widetilde g(T_1^{-4}W_1)  ,\hspace{0,5cm} W\mapsto \widetilde g(T_1^{-4}W_1)^2T_1^{-4 e}W_1^{ e}f_0(T_1^{-4}W_1),$$
							which is $1$-isomorphic with
							$$T\mapsto \widetilde g(T_1^{-4}W_1)\cdot T_1^{2e}, W\mapsto \widetilde g(T_1^{-4}W_1)^2 W_1^{ e}f_0(T_1^{-4}W_1). $$
							Assume that $g=1$, that is $\widetilde g=1$.
							The morphism clearly extends to a morphism $$[\Spec([[T_1]][W_1,W_1^{-1})/_{(1,4)}\Gm]\to [\Spec(\oF[[T]][W,W^{-1}])/_{(2,4)}\Gm] $$ and the morphism on the automorphism group schemes at the closed points is given by $\mu_4\to \mu_4, W_1\mapsto W_1^e$. We obtain that $m=4/\gcd(e,2)$ and that $\mu_m\to \mu_4$ is given by $x\mapsto x$. Assume that $g\neq 1$, that is $\widetilde g(T_0)=T_0$. The morphism is $1$-isomorphic to $$T\mapsto W_1T_1^{2e}, W\mapsto W_1^{2+e} f_0(T_1^{-4}W_1). $$	The morphism clearly extends to a morphism $$[\Spec([[T_1]][W_1,W_1^{-1})/_{(1,4)}\Gm]\to [ \Spec(\oF[[T]][W,W^{-1}])/_{(2,4)}\Gm] $$ and the morphism on the automorphism group schemes at the closed points is given by $\mu_4\to \mu_4, W_1\mapsto W_1^{2+e}$. We obtain that $m=4/\gcd(e+2,2)$ and that $\mu_m\to \mu_4$ is given by $x\mapsto x^{-1}$.
							
						 The proof is completed.
								\end{proof}
								\subsubsection{Stacky curves on $\XXX_0(N)$}
							 {Let us fix a generic section $\theta_N\in \Lambda_N$.} It induces an isomorphism $\OO_{\Spec(\oF((T)))}\xrightarrow{\sim}\Lambda_N|_{\eta_N^{\rig}}$ and hence an isomorphism $\Spec(\oF(X_0(N)))\times B\mu_2\xrightarrow{\theta_N}\sqrt{\Lambda_N^{\rig}|_{\eta^{\rig}_N}}$, which, by abuse of notation, we have also denoted by $\theta_N$.
								\begin{lem}\label{relnormfg}
									Let $C$ be a smooth projective irreducible $\overline F$-curve.   Given a pair $(f,g)\in \oF(C)\times H^1(\oF(C),\mu_2)$ with $f$ non-constant, the morphism $$\iota(f,g):\Spec(\oF(C))\xrightarrow{(f,g)}\Spec(\oF(T))\times B\mu_2\xrightarrow{\sim}\Spec(\oF(X_0(N)))\times B\mu_2\xrightarrow{\theta_N}\sqrt{\Lambda_N^{\rig}|_{\eta^{\rig}_N}}\to\sqrt{\Lambda_N^{\rig}}\cong \XXX_0(N)$$ is affine. The relative normalization $\iota(f,g)^{\nu}:\mathcal C\to\XXX_0(N)$ of $\XXX_0(N)$ with respect to $\iota(f,g)$ is representable and $\mathcal C$ has $C$ for its coarse moduli space. 
									
									Conversely, for any surjective representable $\oF$-morphism $h:\mathcal C\to \XXX_0(N)$, with $\mathcal C$ a stacky curve having $C$ for its coarse moduli space, there exists a unique pair $(f,g)\in \oF(C)\times H^1(\oF(C),\mu_2)$ with $f$ non-constant, and an isomorphism $\iota(f,g)^{\nu}\implies h$.
								\end{lem}
								\begin{proof}
									Let us show the morphism $\iota(f,g)$ is  affine. The morphism $g:\Spec(\oF(C))\to B\mu_2$ is affine as its base change by a morphism from a scheme to $B\mu_2$ is a $\mu_2$-torsor over this scheme. The morphism $f$ is affine, hence the product $(f,g)$ is affine. We are left to verify why the inclusion $\sqrt{\Lambda^{\rig}_{\eta_N^{\rig}}}\to \sqrt{\Lambda^{\rig}_N} $ is affine. This morphism is a base change of the inclusion $\Spec(\oF(X_0(N)))\to\XXX_0(N)^{\rig}$ so to show it is affine, it will be sufficient to show that the latter morphism is affine. Let $\mathcal U\subset\XXX_0(N)^{\rig}$ be the largest open of $\XXX_0(N)$ which is a scheme. Then $\mathcal U=\XXX_0(N)^{\rig}\cong\PP^1$ or $\mathcal U$ is affine subvariety of $\PP^1$. In the first case, the claim follows because the inclusion of the generic point to the projective line is clearly affine (cover the line by two affine open subvarieties) and in the latter case $\Spec(\oF(X_0(N))) \to\XXX_0(N)^{\rig}$ factorizes through $\mathcal U$, hence is affine.  
									Now the first claim follows because relative normalizations are representable morphisms. Let us prove the second claim. Let $\eta_C:\Spec(\oF(C))\to\mathcal C$ be the generic point.   For $(f,g)$ we choose the element given by the identification $$\Spec(\oF(C))\xrightarrow{h|_{\eta_C}}\XXX_0(N)|_{\eta^{\rig}_{N}}\cong\sqrt{(\Lambda^{\rig}_N)_{\eta^{\rig}_N}}\xrightarrow{\theta^{-1}_N}\Spec(\oF(X_0(N)))\times B\mu_2\xrightarrow{\sim}\Spec(\oF(T))\times B\mu_2.$$ Let $\eta_{\mathcal C}:\Spec(\oF(C))\to \mathcal C$ be the generic point. Now, by the universal property of the relative normalization, there exists an isomorphism $\iota(f,g)^{\nu}\implies h$. 
								\end{proof}
								Recall that $\mathcal K_4$ and $\mathcal K_6$ are the closed substacks of $\XXX_0(N)$ where the automorphism group is isomorphic to $\mu_4$ and $\mu_6$, respectively. The two substacks are of dimension $0$. One has that $\mathcal K_4$ (respectively, $\mathcal K_6$) is connected if and only $i\in F$, (respectively, $\omega\in F$) and in case it is not, then it has precisely two connected components.  Write $(\mathcal K_4)_{F(i)}=\mathcal K_4^{-}\cup \mathcal  K_4^{+}$, (respectively, $(\mathcal K_6)_{\oF(\omega)}=\mathcal K_6^{-} \cup\mathcal K_6^+$) for the splitting of $\mathcal K_4$, (respectively, $\mathcal K_6$) over $F(i)$ (respectively, over $F(\omega)$). Write $K_4$ and $K_6$ for the closed points of $X_0(N)$ corresponding to $\mathcal K_4$ and $\mathcal K_6$ for the coarse moduli morphism $\XXX_0(N)\to X_0(N)\cong \PP^1$. Write $(K_4)_{F(i)}=K_4^-\cup K_4^+$ and $(K_6)_{F(\omega)}=K_6^-\cup K_6^+$ with $\mathcal K_i^*$ corresponding to $K_i^* $ for $(i,*)\in\{4,6\}\times \{+,-\}$. We use the following notation: if $\mathcal X$ is a smooth DM stack of dimension $1$ and $x\in \mathcal X(\oF)$ is a point, we denote by $V(x)$ the formal neighborhood of $x$ in $\mathcal X. $ When $\mathcal X$ is not a scheme, for our purposes it means the fiber product of the formal neighborhood of the image of $x$ in the coarse moduli space.
								\begin{lem}\label{associated_sector}
									Let $C$ be a smooth projective irreducible $\oF$-curve.  Let $(f,g)\in \oF(C)\times H^1(\oF(C),\mu_2)$ with $f$ non-constant and let $\iota(f,g)^{\nu}:\mathcal C\to \XXX_0(N)$ be the  representable morphism from a smooth stacky curve $\mathcal C$ having $C$ for coarse moduli space given by Lemma~\ref{relnormfg}. Let $\pi:\mathcal C\to C$ be the coarse moduli morphism. For $y\in C(\oF)$, denote by $e_{f,y}$ the ramification of $f:C\to \PP^1\cong  X_0(N)$ at $y$. The following properties are valid:
									\begin{enumerate}
										\item For any point $x\in  C(\oF)$ which satisfies that $f(x)\not\in K_4\cup K_6$, one has that the associated sector of $\iota(f,g)^{\nu}$ at $\pi^{-1}(x)$ is $0$ if $g=0\in H^1(V(x),\mu_2) $ and is $1/2$ otherwise;
										\item for $*\in\{+,-\}$ and for any point $x\in C(\oF)$ with $f(x)\in K_4^{*}$ , we have that the associated sector of $\iota(f,g)^{\nu}$ at $\pi^{-1}(x)$ is $(\mathcal K_4^{*},\{(e_{f,x})/4\})$ if $g=0\in H^1(V(x),\mu_2)$ and $(\mathcal K_4^{*},\{(e_{f,x}+2)/4\})$ otherwise;
										\item for $*\in\{+,-\}$ and  for any point $x\in  C(\oF)$ with $f(x)\in K_4^*$, we have that the associated sector of $\iota(f,g)^{\nu}$ at $\pi^{-1}(x)$ is $(\mathcal K_6^*,\{(e_{f,x})/6\})$ if $g=0\in H^1(V(x),\mu_2)$ and $(\mathcal K_6^*,\{(e_{f,x}+3)/6\})$ otherwise;
									\end{enumerate}
								\end{lem}
								\begin{proof}  The choice of the line bundle $J_N^*(\Lambda)$ on $\XXX_0(N)$ identifies the commutative diagram
									\[\begin{tikzcd}
										{\mathscr X_0(N)} & {\mathscr X_0(N)^{\rig}} \\
										{\mathscr X_0(1)} & {\mathscr X_0(1)^{\rig}}
										\arrow[from=1-1, to=1-2]
										\arrow[from=1-1, to=2-1]
										\arrow[from=1-2, to=2-2]
										\arrow[from=2-1, to=2-2]
									\end{tikzcd}\]
									with the commutative diagram
									\[\begin{tikzcd}
										{\sqrt{(J_N^{\rig})^*(\mathcal O(1)_{\mathscr P(2,3)})}} & { \sqrt[2,3]{K_4,K_6}} \\
										{{\mathscr P(4,6) }} & {\mathscr P(2,3)}
										\arrow[from=1-1, to=1-2]
										\arrow[from=1-1, to=2-1]
										\arrow[from=1-2, to=2-2]
										\arrow[from=2-1, to=2-2]
									\end{tikzcd}\]
									Here the following convention is used: if $K_6=\emptyset$, then $\sqrt[2,3]{K_4,K_6}$ is interpreted as $\sqrt[2]{K_4}$ and similarly if $K_4=\emptyset$. If both are empty simultaneously, we interpret $\sqrt[2,3]{K_4,K_6}$ as $\PP^1$.  By denoting by $r$ the rigidification morphism whatever the origin is, taking the formal neighborhoods yields a commutative diagram
									\[\begin{tikzcd}
										{V(\pi^{-1}(x))} & {V({\iota(f,g)^{\nu}(\pi^{-1}(x)) })} & {V({r(\iota(f,g)^{\nu}(\pi^{-1}(x)))})} \\
										& {V(J_N(\iota(f,g)^{\nu}(\pi^{-1}(x))))} & {V(J^{\rig}_N(r(\iota(f,g)^{\nu}(\pi^{-1}(x)))))}
										\arrow[from=1-1, to=1-2]
										\arrow[from=1-2, to=1-3]
										\arrow["{J_N}", from=1-2, to=2-2]
										\arrow["{J_N^{\rig}}"', from=1-3, to=2-3]
										\arrow[from=2-2, to=2-3]
									\end{tikzcd}\]
									Assume that $f(x)\not\in K_4\cup K_6$.    The choice of the line bundle $(\mathcal O(1)_{\PPP(4,6)})|_{V(\iota(f,g)^{\nu})}$ yields the diagram
									\[\begin{tikzcd}
										{\sqrt[m]{\Spec(\oF[[T]])}} & {\Spec(\oF[[T]])\times B\mu_2} & {\Spec(\oF[[T]])} \\
										& {\Spec(\oF[[T]])\times B\mu_2} & {\Spec(\oF[[T]])}
										\arrow[from=1-1, to=1-2]
										\arrow[from=1-2, to=1-3]
										\arrow["{}", from=1-2, to=2-2]
										\arrow["{}"', from=1-3, to=2-3]
										\arrow[from=2-2, to=2-3]
									\end{tikzcd}\]
									for some $m|2$. 
									Now observing that the composite of the left horizontal and left vertical morphism is representable and using Lemma~\ref{imagemu2} implies the first claim.  Assume that $f(x)\in K_4$. The indices in the notation below, when the quotient is taken by $\Gm$, represent the weights of the action. The choice of line bundle $(\mathcal O(1)_{\PPP(4,6)})|_{V(\iota(f,g)^{\nu})}$ yields the diagram:
									\[\begin{tikzcd}
										{\sqrt[m]{\Spec(\oF[[T]])}} & {[\Spec(\oF[[T]][Y,Y^{-1}])/_{(2,4)}\Gm]} & {[\Spec(\oF[[T]][Y,Y^{-1}])/_{(1,2)}\Gm]} \\
										& {[\Spec(\oF[[T]][Y,Y^{-1}])/_{(2,4)}\Gm]} & {[\Spec(\oF[[T]][Y,Y^{-1}])/_{(1,2)}\Gm]}
										\arrow[from=1-1, to=1-2]
										\arrow[from=1-2, to=1-3]
										\arrow["{}", from=1-2, to=2-2]
										\arrow["{}"', from=1-3, to=2-3]
										\arrow[from=2-2, to=2-3]
									\end{tikzcd}\]
									for some $m|4$. The composite of the left horizontal and the left vertical morphism is representable. Using Lemma~\ref{localmu4} implies that the induced homomorphism on the automorphism group schemes is given by $\mu_m\to \mu_4$, with $m=1$ if $e\equiv 0\pmod 4$ and $g=0$ or if $e\equiv 2\pmod 4$ and $g=1$, with $m=2$ if $e\equiv 2\pmod 4$ and $g=0$ or $e\equiv 0\pmod 4$ and $g=1$ and with $m=4$ if $e\equiv1\pmod 4 $ and $g=0$ or $e\equiv 3\pmod4 $ and $g=1$.  The claim in this case follows. 
									 Finally, assume that $f(x)\in K_6$. The choice of line bundle $(\mathcal O(1)_{\PPP(4,6)})|_{V(\iota(f,g)^{\nu})}$ yields the diagram:
									\[\begin{tikzcd}
										{\sqrt[m]{\Spec(\oF[[T]])}} & {[\Spec(\oF[[T]][Y,Y^{-1}])/_{(2,6)}\Gm]} & {[\Spec(\oF[[T]][Y,Y^{-1}])/_{(1,3)}\Gm]} \\
										& {[\Spec(\oF[[T]][Y,Y^{-1}])/_{(2,6)}\Gm]} & {[\Spec(\oF[[T]][Y,Y^{-1}])/_{(1,3)}\Gm]}
										\arrow[from=1-1, to=1-2]
										\arrow[from=1-2, to=1-3]
										\arrow["{}", from=1-2, to=2-2]
										\arrow["{}"', from=1-3, to=2-3]
										\arrow[from=2-2, to=2-3]
									\end{tikzcd}\]
									for some $m|6$. The composite of the left horizontal and the left vertical morphism is representable. Using Lemma~\ref{localmu6} and arguing as in the case $m|4$, implies the third claim. 
								\end{proof} 
								%
\subsection{Orbifold pseudoeffective cone} \label{subsec:orb_peff_cone}
The {\it orbifold pseudoeffective cone} plays the same role in the Batyrev--Manin conjecture for stacks as the pseudoeffective cone plays for varieties.
We recall below its definition and properties most of which are established in \cite{dardayasudabm}. We then establish some properties which will helps us calculate it for $\XXX_0(N)$ (for some that we need).
%
%
\begin{mydef}{\cite[Definition 8.4]{dardayasudabm}}
	A stacky curve on $\XXX_0(N)_{\overline F}$ is a representable morphism $f:\mathcal C\to \XXX_0(N)_{\overline F}$ with $\mathcal C$ a smooth stacky curve. A covering family of stacky curves on $\XXX_0(N)_{\overline F}$ is a pair $$(\pi:\widetilde {\mathcal{C}}\to T, \widetilde f:\widetilde{\mathcal C}\to \XXX_0(N)_{\overline F})$$ of $\overline F$-morphisms of Deligne--Mumford stacks such that:
	\begin{itemize}
		\item $\pi$ is smooth and surjective,
		\item $T$ is an integral scheme of finite type over $\overline F$,
		\item for each point $t\in T(\overline F)$, the morphism $\widetilde f|_{\pi^{-1}(t)}:\pi^{-1}(t)\to \XXX_0(N)_{\overline F} $ is a stacky curve on $\XXX_0(N)_{\overline F}$ and
		\item $\widetilde f$ is dominant.
	\end{itemize}
\end{mydef}
\begin{mydef}
Let $f:\mathcal D\to \XXX_0(N)_{\overline F}$ be a morphism with $\mathcal D$ a stacky curve, and let $p\in\mathcal D(\overline F)$. We define  associated sector  of $f$ at $p$ to be the sector of $\XXX_0(N)$ containing the point $(f(p), \widehat{\mu}\to \Aut(p)\to\Aut(f(p)))$. We denote it by $\psi(p)$.
\end{mydef}
Given $(L,c)\in \NS_{\orb}(\XXX_0(N))_{\RR}$ and $f:\mathcal C\to \XXX_0(N)_{\overline F}$ a stacky curve on $\XXX_0(N)_{\overline F}$, we define an intersection number by
\begin{align*}((L,c),f)&=\deg(f^*L)+\sum_{p\text{ stacky point of }\mathcal C}c(\psi(p))
	\end{align*}
By \cite[Lemma 8.6]{dardayasudabm}, we have that for a covering family $(\pi:\widetilde {\mathcal C}\to T,\widetilde f:\widetilde{\mathcal{C}}\to \XXX_0(N)_{\overline F})$ and any $(L,c)\in \NS_{\orb}(\XXX_0(N)_{\overline F})$, there exists a non-empty open $T_0\subset T$ such that $((L,c), \widetilde f|_{\pi^{-1}(t_0)})$ is constant for $t_0\in T_0(\overline F)$. We define $((L,c),(\pi,\widetilde f)) $ to be the generic value. 
\begin{mydef}
The orbifold pseudo-effective cone of $\XXX_0(N)$ is given by $$\Eff_{\orb}(\XXX_0(N)):=\{(L,c)|\hspace{0,1cm}\forall\widetilde{f}\text{ covering family of stacky curves on $\XXX_0(N)_{\overline F},$ one has }((L,c),\widetilde f)\geq 0)\} $$
\end{mydef}
Given a family of covering curves, there is a closed substack of $\XXX_0(N)_{T} = \XXX_0(N)_{\oF} \times_{\oF} T$ defined by the support of the proper image $\widetilde f_T(\widetilde{\mathcal C})$, so the dimension of the fibers of $\widetilde f_T(\widetilde{\mathcal C})$ is upper semicontinuous.
Hence, there exists an open subscheme $T_0$ of the base $T$ such that the family restricted to $T_0$ surjects to $\XXX_0(N)_{\overline F}$ or degenerates to a point.  
\begin{lem}\label{twotypes}
Let $(\pi:\widetilde{\mathcal C}\to T,\widetilde f:\widetilde{\mathcal C}\to \XXX_0(N)_{\overline F})$ be a covering family of stacky curves on $\XXX_0(N)_{\overline F}$. Then  there exists a non-empty open $T_0\subset T$ such that 1) for $t\in T_0(\overline F)$ one has that $\widetilde f|_{\pi^{-1}(t)}$ surjects onto $\XXX_0(N)_{\overline F}$ for $t\in T_0(\overline F)$ or 2) $\widetilde f(\pi^{-1}(t))$ is a point for $t\in T_0(\overline F)$. In the latter case,  we have that every associated sector  of every stacky point of $\widetilde f_{\pi^{-1}(t)}$ for $t\in T_0(\overline F)$ is the sector $0$ or the sector $1/2$. 
\end{lem}
\begin{lem}\label{partofcone}   The inequalities 
	$L\geq 0, x_{1/2}\geq 0$ define boundaries of the orbifold pseudoeffective cone $\Eff_{\orb}(\XXX_0(N))$
For every $(L,c)$ satisfying these inequalities (that is $L\geq 0$ and $c(x_{1/2})\geq 0$) the intersection number with every covering family $(\pi:\widetilde{\mathcal C}\to T,\widetilde f:\widetilde{\mathcal C}\to \XXX_0(N)_{\overline F})$ which is of type 2) from Lemma \ref{twotypes} is non-negative.
	\end{lem}
	\begin{proof}
	Let $x\in \XXX_0(N)_{\overline F}(\overline F)$ be such that $\Aut(x)=\mu_2$.  By \cite[Proof of Proposition 9.22]{dardayasudabm}, there exists a stacky curve $\mathcal C$ and a morphism  $g:\mathcal C\to(B\mu_2)_{\Spec(\overline F)}$ where every associated sector is the non-trivial sector of $B\mu_2$. Now for the morphism $$\mathcal C\xrightarrow{f}(B\mu_2)_{\overline F}\cong (\Spec(\overline F)\times_{(\XXX_0(N)_{\overline F}\circ \XXX_0(N)^{\rig}_{\overline F})\circ x,\XXX_0(N)^{\rig}_{\overline F}}\XXX_0(N)_{\overline F}\to \XXX_0(N)_{\overline F}$$
	every associated sector is $1/2$. We obtain the  inequality $x_{1/2}\geq 0$. The inequality $L\geq 0$ is included by \cite[Lemma 8.10]{dardayasudabm}. If $(L,c)$ satisfies these inequalities, by Lemma~\ref{twotypes}, the associated sectors of $(\pi,\widetilde f)$ are in the set $\{0,1/2\}$ and the intersection number is non-negative.
	\end{proof}

\begin{rem}
	The orbifold pseudoeffective cone of split toric stacks has been explicitly computed in \cite{orbifold-toric}.
	When $F=\QQ$ and $N \in \{1,2,3,4,6,8,9,12,16,18\}$, the stack $\XXX_0(N)$ is indeed a split toric stack.
\end{rem}

\subsection{Cases $N= 12, 16, 18$}
In this paragraph, we treat the cases $N= 12, 16, 18$. 

\begin{thm}
For $N\in\{ 6,8,9,12,16,18\}$, the stacks $\XXX_0(N)$ satisfy that $$\XXX_0(N)\cong \PP^1\times B\mu_2.$$
\end{thm}
\begin{proof}
For $N$ in the range, we have that $\varepsilon_2(N)=\varepsilon_3(N)=0$. It follows that the rigidification satisfies $\XXX_0(N)^{\rig}\cong\mathbb P^1$. We have that $\XXX_0(N)\cong \sqrt[2]{(J_N^{\rig})^*\Lambda^{\rig}}$.  The degree of $\Lambda^{\rig}$ is $\frac{1}{6}$. 
So we will prove the claim, if we can prove that $(J_N^{\rig})^*(\Lambda^{\rig})$ is a square in the Picard group of $\XXX_0(N)^{\rig}\cong \PP^1$. Its degree is $ \frac{\kappa(N)}{6}$. For $N$ in the range, this is an even number. Hence indeed $(J_N^{\rig})^*(\Lambda^{\rig})$ is a square in $\Pic(\XXX_0(N)^{\rig})$. The claim follows.
\end{proof}
\begin{cor}
	The orbifold pseudoeffective cone of $\PP^1\times B\mu_2$ is given by the inequalities $$\{L\geq 0\}, \{x_{1/2}\geq 0\} $$ with $L$ denoting the variable corresponding to the numerical class of the line bundle and $x_{1/2}$ the variable corresponding to the  only twisted sector.
\end{cor}
\begin{proof}
By Lemma~\ref{partofcone}, we have $\Eff_{\orb}(\XXX_0(N))\subset \{L\geq 0, c\geq 0\}$.
Consider surjective morphism $f:\mathcal C\to \XXX_0(N)\cong \PP^1\times B\mu_2$ with $\mathcal C$ a stacky curve.  We have that $$((L,c),f)=L+bc(x_{1/2})$$ for some $b\geq 0$ an integer. Hence the intersection number is positive. Combining with the last claim of Lemma~\ref{partofcone} completes the proof.
\end{proof}
\begin{thm}\label{thm:N=12,16,18}
The prediction is valid for every of the cases $N\in\{12,16,18\}$ (the other values we have already verified). We have the following values for the $a$- and $b$-invariants:
$$	\begin{tabular}{|c|c|c|}
	\hline
	$N$	& $a$&$b$ \\
	\hline
	12	&  1/6 &1          \\
	\hline
	16	&  1/6&1          \\
	\hline 
	18	&    1/6&1        \\
	\hline
	\end{tabular}$$
\end{thm}
\begin{proof}
	The naive height corresponds to the element $$(\deg(J_N^*\Lambda),6)=(\kappa(N)/2, 6).$$
	We have the following table of values for $\kappa(N)$:
		$$	\begin{tabular}{|c|c|}
		\hline
		$N$	& $\kappa(N)$ \\
		\hline
		12	&  24           \\
		\hline
		16	&  24          \\
		\hline 
		18	&    36         \\
		\hline
		%
	\end{tabular}$$
The canonical element of $\NS_{\orb}(\PP^1\times B\mu_2)$ is given by $(-1,-1)$. We have that: $$ a=\inf\{t|\hspace{0,1cm}t\cdot (\kappa(N)/2,6)+(-1,-1)\in \RR^2_{\geq 0}\}=\begin{cases}
1/6,& N=12\\
1/6,	& N=16 \\
1/6,	& N=18.
\end{cases} $$
Clearly, the $b$-invariant in each of these cases is~$1$. The proof is completed.
\end{proof} 
\subsection{Case $N=7$}
We recall the following fact. Let $\omega$ be a primitive third root of unity.
\begin{lem}
Let $E_0$ be an elliptic curve over~$F$ satisfying that $j(E_0)=0$. One has that $$\End(E_0)=\begin{cases}
\ZZ,	&\text{if }\omega\not\in F\\
\ZZ[\omega],&\text{ otherwise}
\end{cases} $$
\end{lem}
\begin{proof}
One has that $\End((E_0)_{\overline F})=\ZZ[\omega]$, hence $\ZZ\subset\End(E_0)\subset \ZZ[\omega]$.  The absolute Galois group $\Gal(\overline F/F)$ acts on $\End ((E_{0})_{\overline F})$ by $\gamma\cdot \theta=\gamma\circ\theta\circ \gamma^{-1}$. If $\theta$ is represented by $a+b\omega$ with $a,b\in\ZZ$, the element $\gamma\cdot \theta$ corresponds to $$\gamma\cdot(a+b\omega)=a+b\gamma\cdot\omega =a+b(-1-\omega)=(a-b)-b\omega$$ if $\omega\not\in F$ and to
$$\gamma\cdot(a+b\omega)=a+b\omega$$
if $\omega\in F$. We have that $\End(E_0)=\End((E_0)_{\overline F})^{\Gal(\overline F/F)}$ which is thus $\ZZ$ iff $\omega\in F$ and $\ZZ[\omega]$ otherwise. The claim is proven.
\end{proof}
\begin{lem} \label{nmbc7} There exist precisely $3-[F(\omega):F]$ closed points of $\XXX_0(7)$ where the automorphism group is~$\mu_6$ (they are necessarily lying above the point $j=0$ in $\XXX_0(1)$ for the morphism $J_N:\XXX_0(7)\to\XXX_0(1)$.)
\end{lem}
\begin{proof}
Suppose that $\omega\in F$. Then as $\End(E_0)=\ZZ[\omega]$, we deduce that every isogeny of $E_0$ is defined over $F$. By Lemma~\ref{cmsunr}, we have that the number of isogeny structures is $\varepsilon_3(7)=2$. Suppose that $\omega\not\in F$. Then there is no ideal~$\mathfrak a$ of $\ZZ$ satisfying that $N_{\QQ(\omega)/\QQ}(\mathfrak a)=7$, thus there is no isogeny of degree $7$ defined over~$F$. It follows that the number of closed points is $\varepsilon_3(7)/2=1$.  The proof is completed.
\end{proof}
\begin{lem}\label{picx03}
One has that $\Pic(\XXX_0(7)^{\rig})$ is $\mathbb Z$ if $\omega\not\in F$ and $\mathbb Z\oplus \ZZ/3\ZZ$ otherwise. In particular, if a line bundle $L\in\Pic(\XXX_0(7)^{\rig})$ satisfies that $\deg(L)\in 2\deg(\Pic(\XXX_0(7)^{\rig}))$, then it is a square.
\end{lem}
\begin{proof}
We use \cite[Theorem 1.1]{lopez2023picardgroupsstackycurves}. Assume first $\omega\not\in F$. Then the quoted theorem gives us the following push-out commutative diagram, where $j_7^*$ is induced from the coarse moduli space map $j_7:\XXX_0(7)^{\rig}\to \PP^1$:
\[\begin{tikzcd}
	{\mathbb Z^1} & {\mathbb Z^1} \\
	{\Pic(\PP^1)\cong\mathbb Z} & {\Pic(\XXX_0(7)^{\rig})}
	\arrow["{x\mapsto 3x}", from=1-1, to=1-2]
	\arrow["{x\mapsto 2x}"', from=1-1, to=2-1]
	\arrow[from=1-2, to=2-2]
	\arrow["{j_7^*}", from=2-1, to=2-2]
\end{tikzcd}\]
The pushout is isomorphic to $(\ZZ\oplus\ZZ)/\{(3m,-2m)|\hspace{0,1cm}m\in\ZZ\}$ which is $\ZZ$.
Assume that $\omega\in F$. The the quoted theorem gives us the following push-out commutative diagram:
\[\begin{tikzcd}
	{\mathbb Z^2} & {\mathbb Z^2} \\
	{\Pic(\PP^1)\cong\mathbb Z} & {\Pic(\XXX_0(7)^{\rig})}
	\arrow["{(x,y)\mapsto (3x,3y)}", from=1-1, to=1-2]
	\arrow["{(x,y)\mapsto (x+y)}"', from=1-1, to=2-1]
	\arrow[from=1-2, to=2-2]
	\arrow["{j_7^*}", from=2-1, to=2-2]
\end{tikzcd}\]
We again get $\Pic(\XXX_0(7)^{\rig})$ is $\mathbb Z^3/\{(3m,3n,-m-n)|\hspace{0,1cm}m,n\in\ZZ\}$ which is $\ZZ\oplus (\ZZ/3\ZZ) $. Let us prove the last statement. Suppose that $\deg(L)=2\deg(L')=\deg((L')^{\otimes 2})$. We have that $L_0:=L\otimes (L')^{\otimes -2}$ is of degree $0$. There exists an integer $n$ such that $L_0^{\otimes n}$ descends to a line bundle on $\PP^1$. But the degree of the line bundle on $\PP^1$ is $0$, hence it is the trivial line bundle. Hence $L_0^{\otimes n}$ is the trivial line bundle, hence $L_0$ is a $3$-torsion line bundle. But this means that it is a square. Hence $L=(L')^{\otimes 2}\otimes L_0$ is a square too. 
\end{proof}
\begin{thm} Consider the closed point $D:X^2+3Y^2=0$ of $\PP^1$.  One has that $$\XXX_0(7)\cong \sqrt[3]{D}\times B\mu_2.$$
\end{thm}
\begin{proof}
It follows from Lemma~\ref{nmbc7} that the rigidification is $\XXX_0(7)^{\rig}\cong\sqrt[3]{D}$. The degree of the morphism $J_7^{\rig}:\XXX_0(7)^{\rig}\to \XXX_0(1)^{\rig}$ is the same as the degree of $J_7:\XXX_0(7)\to \XXX_0(1)$ which is $\kappa(7)=8$. The morphism $J_7$ when restricted to $\YYY_0(7)$ is \'etale, hence $J_7^{\rig}$ is \'etale when restricted to $J^{\rig}_7:\YYY_0(7)^{\rig}\to \YYY_0(1)$.  The degree of $\mathcal O(1)$ on $\PPP(2,3)$ is $\frac{1}{6}$. The pullback $(J_7^{\rig})^*(\mathcal O(1))$ is of degree $\frac{8}{6}=\frac{4}{3}$. The anticanonical line bundle $K_{\XXX_0(7)^{\rig}}^{-1}$ is of degree $2\cdot \deg(K_{\XXX_0(7)}^{-1})=\frac{2}{3}$. It follows that $(K_{\XXX_0(7)^{\rig}}^{-1})^{\otimes 2}$ and $(J_7^{\rig})^*(\mathcal O(1))$ have the same degree. Hence by Lemma~\ref{picx03}, we have that   $(J_7^{\rig})^*(\mathcal O(1))$ is a square in the Picard group of $\XXX_0(7)^{\rig}$. But $\XXX_0(7)\cong \sqrt[2]{(J_7^{\rig})^*(\mathcal O(1))}$, hence $\XXX_0(7)\cong \sqrt[3]{D}\times B\mu_2$.
\end{proof}

We will implicitly fix an identification $\PP^1\cong X_0(N)$. Let $(\mathcal K_6)_{\overline{F}}=(\mathcal K_6)^+\cup (\mathcal K_6)^-$ be the decomposition of $(\mathcal K_6)_{\overline F(\omega)}$ into two disjoint points. Denote $(K_6)^*$ the image of $(\mathcal K_6)_*$ with $*\in\{+,-\}$ for the coarse moduli space map and $K_6=(K_6)^+\cup (K_6)^-$. Let $V(x)$ denotes the formal neighborhood of a point $x$ on a curve. Given a smooth projective irreducible $\oF$-curve $C$ and a pair $(f,g)\in(\oF(C)-\oF)\times H^1(\oF(C),\mu_2)$ and $(e,i)\in\{0,1,2,3,4,5\}\times\{0,1\}-\{(0,0)\}$, we define $n_{+,e,i}(f,g)$ (respectively, $n_{-,e,i}(f,g)$) to be the number of geometric points $x$ of $C$ which satisfy $f(x)\in (K_6)^+$ (respectively, $f(x)\in (K_6)^-$) and at which the ramification degree of $f$ is congruent to $e$ modulo $6$ and for which $g|_x=1\in H^1(V_x,\mu_2)$ (respectively, $g|_x=0\in H^1(V_x,\mu_2)$) if $i=1$ (respectively, if $i=0$). Specially, we define $n_{-,0,0}(f,g)=n_{+,0,0}(f,g)=0$ and $n_{1/2}(f,g)$ to be the number of geometric points $x$ of $C$ satisfying that $f(x)\not \in K_6$ and above which $g$ is ramified. If $k\equiv e\pmod 6$, by $n_{*,k,j}(f,g)$ we denote $n_{*,e,j}(f,g)$. 

It follows from Lemma~\ref{associated_sector}, that $\gamma(f,g)\in\NS_{\orb}(\XXX_0(N))^*$ is the linear form defined by the representable morphism $\iota(f,g)^{\nu}.$
For a pair $(f,g)$ as before, we denote by $\gamma (f,g)$ the following element of $\NS_{\orb}(\XXX_0(7))^*$:
\begin{itemize}
	\item Assume that $\omega\not\in F$. We set:
	\begin{multline*}
		\gamma(f,g)(L,x):=\deg(f)\cdot L+\bigg(\sum_{\substack{1\leq j\leq 5\\ j\neq 3}}x_{(\mathcal K_6,j/6)}\sum_{*\in\{+,-\}}n_{*,j,0}(f,g)+n_{*,3+j,1}(f,g)\bigg)+\\+x_{1/2}\bigg(n_{1/2}(f,g)+\sum_{*\in\{+,-\}}n_{*,3,0}(f,g)+n_{*,0,1}(f,g)\bigg).
	\end{multline*}
	\item Assume that $\omega\in F$. We set 
	\begin{multline*}\gamma(f,g)(L,x):=\deg(f)\cdot L+\sum_{*\in\{+,-\}}\sum_{\substack{1\leq j\leq 5\\j\neq 3}}x_{((\mathcal K_6)^*,j/6)}(n_{*,j,0}(f,g)+n_{*,3+j,1}(f,g))+\\+ x_{1/2}\bigg(n_{1/2}(f,g)+\sum_{*\in\{+,-\}} n_{*,3,0}(f,g)+n_{*,0,1}(f,g)\bigg).\end{multline*}
\end{itemize}
It follows from Lemma~\ref{associated_sector}, that $\gamma(f,g)\in\NS_{\orb}(\XXX_0(N))^*$ is the linear form defined by the representable morphism $\iota(f,g)^{\nu}.$
\begin{thm}\label{thm:N=7}
Suppose that $N=7$. The $a$-invariant is given by $1/6.$ The $b$-invariant is given by $b=2.$ 
\end{thm}
\begin{proof}
Using notation $(\mathcal K_6)^*$ as before in the case $\omega\in F$ and using it for $\mathcal K_6$ in the case $\omega\not\in F$, one has that $$K_{\XXX_0(7),\orb}=(-1/3, (-1/3)_{((\mathcal K_6)^*,1/6)},(-2/3)_{((\mathcal K_6)^*,1/3)},(-1)_{1/2},(-1/3)_{((\mathcal K_6)^*,2/3)}, (-2/3)_{((\mathcal K_6)^*,5/6)}).$$ 
The naive height corresponds to $$T:=(4, (10)_{((\mathcal K_6)^*,1/6)},(8)_{((\mathcal K_6)^*,1/3)},(6)_{1/2},(4)_{((\mathcal K_6)^*,2/3)}, (2)_{((\mathcal K_6)^*,5/6)}).$$ We have that $x_{1/2}\geq 0$ is an inequality defining cone, so $a\geq 1/6$. Consider the point
\begin{align*}
	O:&=\frac16T+K_{\XXX_0(7),\orb}\\&=(2/3, (5/3)_{((\mathcal K_6)^*,1/6)},(4/3)_{((\mathcal K_6)^*,1/3)},(1)_{1/2},(2/3)_{((\mathcal K_6)^*,2/3)}, (1/3)_{((\mathcal K_6)^*,5/6)})+\\
	&\quad\quad+(-1/3, (-1/3)_{((\mathcal K_6)^*,1/6)},(-2/3)_{((\mathcal K_6)^*,1/3)},(-1)_{1/2},(-1/3)_{((\mathcal K_6)^*,2/3)}, (-2/3)_{((\mathcal K_6)^*,5/6)})\\
	&=(1/3, (4/3)_{((\mathcal K_6)^*,1/6)},(2/3)_{((\mathcal K_6)^*,1/3)},(0)_{1/2},(1/3)_{((\mathcal K_6)^*,2/3)}, (-1/3)_{((\mathcal K_6)^*,5/6)}).
\end{align*}
Consider smooth curve $C$ and $(f,g)\in(\oF(C)-\oF)\times H^1(\oF(C),g)$. 
\begin{multline*}\gamma(f,g)(O)=\frac{1}{3}\cdot\bigg( {\deg(f)}{+4\cdot\bigg(\sum_{*\in\{+,-\}}n_{*,1,0}(f,g)+n_{*,4,1}(f,g)\bigg)}+2\cdot \bigg(\sum_{*\in\{+,-\}}n_{*,2,0}(f,g)+n_{*,5,1}(f,g)\bigg)\\+\bigg(\sum_{*\in\{+,-\}}n_{*,4,0}(f,g)+n_{*,1,1}(f,g)\bigg)-\bigg(\sum_{*\in\{+,-\}}n_{*,5,0}(f,g)+n_{*,2,1}(f,g)\bigg)\bigg)\end{multline*}
Note that $\deg(f)\geq 5n_{+,5,0}(f,g)+2n_{+,2,1}(f,g)$ and that $\deg(f)\geq 5n_{-,5,0}(f,g)+2n_{-,2,1}(f,g)$, hence 
$$2\deg(f)\geq 5n_{+,5,0}(f,g)+5n_{-,5,0}(f,g)+2n_{+,2,1}(f,g)+2n_{-,2,1}(f,g).$$
We deduce $\gamma(f,g)(O)\geq 0.$ In other words, the point $O$ is contained in $\Eff_{\orb}(\XXX_0(7))$. If $\gamma(f,g)(O)=0$, for $*\in\{+,-\}$, we must have that 
{$$n_{*,5,0}(f,g)=n_{*,4,0}(f,g)=n_{*,1,1}(f,g)=n_{*,2,0}(f,g)=n_{*,5,1}(f,g)=n_{*,1,0}(f,g)=n_{*,4,1}(f,g)=0.$$}
We must have that $\deg(f)=n_{+,2,1}(f,g)+n_{-,2,1}(f,g)$ and thus, as $\deg(f)\geq 2\max_{*\in\{+,-\}}\{n_{*,2,1}(f,g)\}$, that $n_{+,2,1}(f,g)=n_{-,2,1}(f,g)$.  In particular, it follows that all $\gamma(f,g)$ for $(f,g)$ as above are of the form 
$$(L,x)\mapsto 2m L + mx_{((\mathcal K_6)_-,5/6)}+mx_{((\mathcal K_6)_+,5/6)} $$ if $\omega\in F$ and 
$$(L,x)\mapsto 2m L + 2mx_{((\mathcal K_6),1/6)}$$ if $\omega\in F$.  
 Let us show that there exists at least one such form $\gamma(f,g)$ for $(f,g)$ as above, not in the span of $(L,x)\mapsto x_{1/2}$. Let's take $f$ to be $f:\PP^1\to X_0(7)\cong\PP^1$, given by $$f(z)=\dfrac{\omega^2\cdot \big(\frac{z}{z-1}\big)^2+\omega}{\big(\frac{z}{z-1}\big)^2+1}.$$
Note that $f=M\circ h^2$, where $h(z)=z/(z-1)$ is of degree $1$ and $M(t)=\frac{\omega^2t+\omega}{t+1}$ is a M\"obius transformation. It follows that $\deg(f)=2$. Moreover, if $f(z)=\omega$, by using that $M(0)=\omega$, we have that $h(z)^2=0$, that is $z=0$. Similarly, if $f(z)=\omega^2$, by using that $M(\infty)=\omega^2$, we obtain $h(z)=\infty$, that is $z=1$.  Choose for $g$ any degree $2$ cover of $\PP^1$, ramified at both $z=0$ and $z=1$. Then $$\gamma(f,g)(L,x)=2L+x_{((\mathcal K_6)_-,5/6)}+x_{((\mathcal K_6)_+,5/6)}$$ if $\omega\in F$ and $$\gamma(f,g)(L,x)=2L+2x_{((\mathcal K_6),5/6)}$$ if $\omega\not \in F$. 
By recalling that $(L,x)\mapsto x_{1/2}$ is a form which vanishes on $O$ and which is in $\Eff_{\orb}(\XXX_0(7))^*$, we obtain that 
\begin{multline*}\Eff_{\orb}(\XXX_0(7))^*\cap \{\lambda\in \NS_{\orb}(\XXX_0(7))^*|\hspace{0,1cm}\lambda(x)=0\}=\\=\{(L,x)\mapsto t_1x_{1/2}+2t_2L+t_2x_{((\mathcal K_6)_-,5/6)}+t_2x_{((\mathcal K_6)_+,5/6)}|\hspace{0,1cm}t_1,t_2\geq 0\}\end{multline*}
if $\omega\in F$ and 
\begin{multline*}\Eff_{\orb}(\XXX_0(7))^*\cap \{\lambda\in \NS_{\orb}(\XXX_0(7))^*|\hspace{0,1cm}\lambda(x)=0\}=\\=\{(L,x)\mapsto t_1x_{1/2}+2t_2L+2t_2x_{((\mathcal K_6),5/6)}|\hspace{0,1cm}t_1,t_2\geq 0\}
\end{multline*}
if $\omega\in F$. We deduce that $$b=\dim(\Eff_{\orb}(\XXX_0(7))^*\cap \{\lambda\in\NS_{\orb}(\XXX_0(7))^*|\hspace{0,1cm}\lambda(O)=0\})= 2.$$
\end{proof}
\subsection{Case $N=10,25$}
Denote by $i$ a chosen square root of $-1$.
\begin{lem}
	Let $E_{1728}$ be an elliptic curve over~$F$ satisfying that $j(E_{1728})=1728$. One has that $$\End(E_{1728})=\begin{cases}
		\ZZ,	&\text{if }i\not\in F\\
		\ZZ[i],&\text{ otherwise}
	\end{cases} $$
\end{lem}
\begin{proof}
	One has that $\End((E_{1728})_{\overline F})=\ZZ[i]$, hence $\ZZ\subset\End(E_0)\subset \ZZ[i]$.  The absolute Galois group $\Gal(\overline F/F)$ acts on $\End ((E_{1728})_{\overline F})$ by $\gamma\cdot \theta=\gamma\circ\theta\circ \gamma^{-1}$. If $\theta$ is represented by $a+b\omega$ with $a,b\in\ZZ$, the element $\gamma\cdot \theta$ corresponds to $$\gamma\cdot(a+bi)=a+b\gamma\cdot i =a+b(-i)=a-bi$$ if $i\not\in F$ and to
	$$\gamma\cdot(a+bi)=a+bi$$
	if $i\in F$. We have that $\End(E_{1728})=\End((E_{1728})_{\overline F})^{\Gal(\overline F/F)}$ which is thus $\ZZ$ iff $i\in F$ and $\ZZ[i]$ otherwise. The claim is proven.
\end{proof}
\begin{lem}
\label{nmbc10} There exist precisely $3-[F(i):F]$ closed points of $\XXX_0(10)$  where the automorphism group is~$\mu_4$ (they are necessarily lying above the point $j=1728$ in $\XXX_0(1)$ for the morphism $J_N:\XXX_0(10)\to\XXX_0(1)$).
\end{lem}
\begin{proof}
	Suppose that $i\in F$. Then as $\End(E_{1728})=\ZZ[i]$, we deduce that every cyclic $10$-isogeny of $E_{1728}$ is defined over $F$. By Lemma~\ref{cmsunr}, we have that this number is $\varepsilon_2(10)=2$. Suppose that $i\not\in F$. Then there is no ideal~$\mathfrak a$ of $\ZZ$ satisfying that $N_{\QQ(i)/\QQ}(\mathfrak a)=7$, thus there is no isogeny of degree $10$ defined over~$F$. It follows that the number of closed points in this case  is $\varepsilon_2(10)/2=1$.  The proof is completed.
\end{proof}
\begin{thm}
	Consider the closed point $K:X^2+Y^2=0$ of $\PP^1$.  One has that $\XXX_0(10)^{\rig} \cong \sqrt[2]{K}$ and that $\XXX_0(10)$ is a non-trivial root stack $\XXX_0(10)\cong\sqrt[2]{(J_{10}^{\rig})^*\Lambda^{\rig}}$.
\end{thm}
\begin{proof}
	It follows from Lemma~\ref{nmbc10} that the rigidification is $\XXX_0(10)^{\rig}\cong\sqrt[2]{K}$. Let us verify that $\XXX_0(10)\cong\sqrt[2]{J_{10}^*(\Lambda^{\rig}})$ is a non-trivial root stack. If we had $\XXX_0(10)\cong \sqrt[2]{K}\times B\mu_2,$ then $\XXX_0(10)$ would have a point with automorphism group scheme $\mu_2\times \mu_2$. But we know that such automorphism group schemes do not occur for points in $\XXX_0(10)$.
\end{proof}
\subsubsection{}{ Before continuing with the calculation of $a$- and $b$-invariants, let us show that $\XXX_0(25)\cong \XXX_0(10)$.
} In fact, both stacks are isomorphic to $\XXX_0(5)$.
\begin{lem} \label{str25} If $i\not\in F$, no elliptic curve $E_{1728}$ with $j(E_{1728})=1728$ admits a cyclic $25$-isogeny. If $i\in F$, there exist precisely two pairs consisting of an elliptic curve with $j$ invariant equal to $1728$ and a cyclic $25$-isogeny, which are not isomorphic over $\overline F$ and such that any other pair consisting of an elliptic curve with $j$ invariant equal to $1728$ and a cyclic $25$-isogeny is isomorphic to one of them.
\end{lem}
\begin{proof}
	Suppose that $i\not\in F$. Then $\End(E_{1728})=\ZZ$. Clearly, the only ideal $\mathfrak a$ with the norm $N_{\QQ(i)/\QQ}(\mathfrak a)=25$ in $\End(E_{1728})$ is $(5)$. It corresponds to the non-cyclic isogeny given by multiplication by $5$ on $E_0$. Suppose that $i\in F$. Then $\End(E_{1728})=\ZZ[i]$ and every cyclic $25$-isogeny is defined over $F$. But as $\varepsilon_2(25)=2$, the second claim follows.
\end{proof}
\begin{cor}
{The stacks $\XXX_0(5)$, $\XXX_0(10)$ and  $\XXX_0(25)$ are isomorphic.} 
\end{cor}
\begin{proof} Combining the proof of Theorem~\ref{aandb}, Lemma~\ref{nmbc10} and Lemma~\ref{str25} we see that the rigidifications of the three stacks are isomorphic to the root stack $\sqrt{K_4}$ where $K_4: X^2+Y^2=0$. Applying the analogous arguments to the ones  in the proof of Lemma~\ref{picx03}, we obtain that $\Pic(\XXX_0(5)^{\rig})\cong \ZZ$ if $i\not\in F$ and $\Pic(\XXX_0(5)^{\rig})\cong \ZZ$ if $i\in F$. We have a diagram $\XXX_0(25)\to\XXX_0(5)\to\XXX_0(1)$ which induces a diagram $\XXX_0(25)^{\rig}\to\XXX_0(5)^{\rig}\to\XXX_0(1)^{\rig}$. The composite morphism is of degree $30$, while the right homomorphism is of degree $6$. It follows that the morphism $\XXX_0(25)^{\rig}\to\XXX_0(5)^{\rig}$ is of degree $5$ and similarly, the morphism $\XXX_0(10)^{\rig}\to\XXX_0(5)^{\rig}$ is of degree $3$. Assume $i\not\in F$. Then the induced homomorphism $$\ZZ/2\ZZ\cong \Pic(\XXX_0(5)^{\rig})/2\Pic(\XXX_0(5)^{\rig})\to\Pic(\XXX_0(25)^{\rig})/2\Pic(\XXX_0(25)^{\rig}) \cong \ZZ/2\ZZ $$ is an isomorphism and similarly $$\ZZ/2\ZZ\cong \Pic(\XXX_0(5)^{\rig})/2\Pic(\XXX_0(5)^{\rig})\to\Pic(\XXX_0(10)^{\rig})/2\Pic(\XXX_0(25)^{\rig}) \cong \ZZ/2\ZZ.$$
Hence $(J^{\rig}_{10})^*(\Lambda^{\rig})(J^{\rig}_{25})^*(\Lambda^{\rig})^{-1}\in 2\Pic(\XXX_0(5)^{\rig})$ and in this case the claim follows. Assume that $i\in F$. An isomorphism $\ZZ\oplus (\ZZ/2\ZZ)\xrightarrow{\sim}\Pic(\XXX_0(5)^{\rig})$ is given by $$(m,n)\mapsto (J_5^{\rig})^*(\Lambda)^{\otimes m}\otimes\mathcal O((K_4)_+-(K_4)_-)^{\otimes n}$$for $m\in\ZZ$ and $n\in\ZZ/2\ZZ$. (As $\OO((K_4)_+-(K_4)_-)^{\otimes 2}=\mathcal O(2(K_4)_+-2(K_4)_-)=\pi^*\mathcal O((K_4)_+-(K_4)_-)=0$ where $\pi:\XXX_0(5)^{\rig}\to X_0(5)$ is the coarse moduli morphism, the notation is unambiguous.) We will show that the image of the class of the only non-trivial torsion line bundle of $\Pic(\XXX_0(5)^{\rig})$ for  $\Pic(\XXX_0(5)^{\rig})/2\Pic(\XXX_0(5)^{\rig})\to\Pic(\XXX_0(25)^{\rig})/2\Pic(\XXX_0(25)^{\rig})$ is the class of the only non-trivial line bundle in $\Pic(\XXX_0(25)^{\rig})$. Obviously, it suffices to show that the only non-trivial torsion line bundle for $\Pic(\XXX_0(5)^{\rig})\to\Pic(\XXX_0(5)^{\rig})$ is the only non-trivial torsion line bundle. It follows from Lemma~\ref{cmsunr} that the morphism $\XXX_0(25)^{\rig}\to\XXX_0(5)^{\rig}$ is \'etale at $K_4$. Hence the image of $\OO((K_4)_+-(K_4)_-)$ is the line bundle $\OO(5(K_4)_+-5(K_4)_-)$, which is the only non-trivial torsion line bundle. Moreover, it is clear that the image of the class of $(J_5^{\rig})^*(\Lambda^{\rig})$ for $\Pic(\XXX_0(5)^{\rig})/2\Pic(\XXX_0(5)^{\rig})\to\Pic(\XXX_0(25)^{\rig})/2\Pic(\XXX_0(25)^{\rig}) $  is the class of $(J_{25}^*)^{\rig}(\Lambda^{\rig})$. Hence we have an isomorphisms $\XXX_0(5)\xrightarrow{\sim}\XXX_0(25)$ and similarly one shows $\XXX_0(5)\xrightarrow{\sim}\XXX_0(25)$. 
	\end{proof}
\subsubsection{} We will fix an implicit identification $X_0(N)\cong\PP^1$. Given an $\oF$-point of a curve~$x$, we denote by $V_x$ the corresponding formal neighborhood. Given a pair $(f,c)\in (\oF(C)-\oF)\times H^1(\oF(C),\mu_2)$ with $C$ smooth projective irreducible $\oF$-curve, we define:
\begin{itemize}
	\item $n_{1/2}(f,g)$ to be the number of geometric points $x$ of $C$ above which $g$ is ramified and which satisfy that $f(x)\not \in K$;
	\item Let $(\mathcal K_4)_{\overline{F}}=(\mathcal K_4)_+\cup (\mathcal K_4)_-$ be the decomposition of $(\mathcal K_4)_{\overline F}$ into two disjoint points. Denote $(K_4)_*$ the image of $(\mathcal K_4)_*$ with $*\in\{+,-\}$ for the coarse moduli space map and $K_4=(K_4)_+\cup (K_4)_-$. We set $n_{+,i,0}(f,g)$ (respectively, $n_{-,1,0}(f,g)$) to be the number of geometric points $x$ of $C$ above which~$ g$ is unramified, which satisfy that $f(x)= ( K_4)_+$ (respectively, $f(x)=( K_4)_-$) and the ramification of $f$ at $x$ is congruent $i$ modulo $4$. 
	\item  We set $n_{+,1,1}(f,g)$ (respectively, $n_{-,1,1}(f,g)$) to be the number of geometric points $x$ of $C$ above which~$g$ is ramified, which satisfy that $f(x)= (K_4)_+$ (respectively, $f(x)=(K_4)_-$) and the ramification of $f$ is congruent $i$ modulo $4$.
\end{itemize}
For a pair $(f,g)$ as above, we denote by $\gamma (f,g)$ the following element of $\NS_{\orb}(\XXX_0(10))^*$:
\begin{itemize}
	\item Assume that $i\not\in F$. We set: \begin{multline*}\gamma(f,g)(L,x):=\deg(f)\cdot L+x_{((\mathcal K_4),1/4)}\cdot\sum_{*\in \{+,-\}} (n_{*,1,0}(f,g)+n_{*,3,1}(f,g))+\\+x_{1/2}\cdot \bigg(n_{1/2}(f,g)+\sum_{*\in\{+,-\}}n_{*,0,1}(f,g)+n_{*,2,1}(f,g)\bigg)+ x_{((\mathcal K_4),3/4)}\cdot\sum_{*\in\{+,-\}}(n_{*,1,1}(f,g)+n_{*,3,0}(f,g)).\end{multline*}
	\item Assume that $i\in F$. We set:
	\begin{multline*}\gamma(f,g)(L,x):=\deg(f)\cdot L+ \sum_{*\in\{+,-\}} x_{((\mathcal K_4)_*,1/4)}\cdot (n_{*,3,0}(f,g)+n_{*,1,1})+\\ + x_{1/2}\cdot \bigg(n_{1/2}(f,g)+\sum_{*\in\{+,-\}}n_{*,0,1}(f,g)+n_{*,2,1}(f,g)\bigg)+\sum_{*\in\{+,-\}}x_{((\mathcal K_4)_*,3/4)}\cdot (n_{*,1,0}(f,g)+n_{*,3,1}(f,g)).
	\end{multline*}
\end{itemize}
It follows from Lemma~\ref{localmu4}, that $\gamma(f,g)\in\NS_{\orb}(\XXX_0(N))^*$ is induced by $\iota(f,g)^{\nu}$.
\begin{thm}\label{thm:N=10,25}
	When $N=10$ we have that $a=1/6$ and that $b=1$. When $N=25$, we also have that $a=1/6$ and that $b=1$.
\end{thm}
\begin{proof}
Using notation $(\mathcal K_4)_*$ as before in the case $i\in F$ and using it for $(\mathcal K_4)$ in the case $i\not\in F$, one has that $$K_{\XXX_0(10),\orb}=(-1/2, (-1/2)_{((\mathcal K_4)_*,1/4)},(-1)_{1/2},(-1/2)_{(\mathcal K_4)_*,3/4}).$$ 
When $N=10$, the naive height corresponds to $$T:=(9, (9)_{((\mathcal K_4)_*,1/4)},(6)_{1/2},(3)_{((\mathcal K_4)_*,3/4)}).$$
We want the infimum of $a$ for which $aT+K_{\XXX_0(10),\orb}\in\Eff_{\orb}(\XXX_0(10))$. Recall that $x_{1/2}\geq 0$ is an equation defining cone, so $a\geq 1/6$. Consider the point \begin{align*}O:&=(1/6)\cdot T+K_{\XXX_0(10)}\\&=(3/2, (3/2)_{((\mathcal K_4)_*,1/4)},(1)_{1/2},(1/2)_{(K_*,3/4)})+(-1/2, (-1/2)_{((\mathcal K_4)_*,1/4)},(-1)_{1/2},(-1/2)_{K_*,3/4})\\&=(1, (1)_{((\mathcal K_4)_*,1/4)},(0)_{1/2},(0)_{(\mathcal K_4)_*,3/4}).
\end{align*} 
Consider $(f,g)$ as before.  
Assume that $\omega\in F$. We have that
\begin{align*}\gamma(f,g)(O)\geq \deg(f)\cdot 1>0.\end{align*}
Hence the inequalities defined by $(f,g)$ as above are all strict. We get $a=1/6$. We also get $$b=\dim(\Eff_{\orb}(\XXX_0(10))^*\cap \{\lambda\in\NS_{\orb}(\XXX_0(10))^*|\hspace{0,1cm}\lambda(O)=0\})=1.$$ 
 When $N=25$, the naive height corresponds to 
$$T':=(15, (9)_{(K_*,1/4)},(6)_{1/2},(3)_{(K_*,3/4)}).$$ We want the infimum of $a$ for which $aT'+K_{\XXX_0(10),\orb}\in\Eff(\XXX_0(10))$. Again, as $x_{1/2}\geq 0$ is an equation defining cone, we have $a\geq 1/6$. Consider the point
\begin{align*}O':&=(1/6)T'+K_{\XXX_0(10)}\\&=(5/2, (3/2)_{((\mathcal K_4)_*,1/4)},(1)_{1/2},(1/2)_{((\mathcal K_4)_*,3/4)})+(-1/2, (-1/2)_{((\mathcal K_4)_*,1/4)},(-1)_{1/2},(-1/2)_{(\mathcal K_4)_*,3/4})\\&=(2, (1)_{((\mathcal K_4)_*,1/4)},(0)_{1/2},(0)_{(\mathcal K_4)_*,3/4}).
\end{align*} 
Consider $(f:C\to \PP^1,g:\widetilde C\to C)$ as before. We have that  
$$\gamma(f,g)(O')\geq \deg(f)\cdot 2>0.$$ We obtain
$$a=1/6\text { and }b=\dim(\Eff_{\orb}(\XXX_0(10))^*\cap \{\lambda\in\NS_{\orb}(\XXX_0(10))^*|\hspace{0,1cm}\lambda(O')=0\})=1.$$ 
The proof is completed.
\end{proof}
\subsection{Case $N=13$}
\begin{lem}
	\label{nmbc13} There exist precisely $3-[F(i):F]$ closed points of $\XXX_0(13)$ with automorphism group scheme $\mu_4$ (they are necessarily lying above the point $j=1728$ in $\XXX_0(1)$ for the morphism $J_N:\XXX_0(13)\to\XXX_0(1)$). There exist precisely $3-[F(\omega):F]$ closed points of $\XXX_0(13)$ with automorphism group scheme $\mu_6$ (they are necessarily lying above the point $j=0$ in $\XXX_0(1)$ for the morphism $J_N:\XXX_0(13)\to\XXX_0(1)$)
\end{lem}
\begin{proof}
An analogous proof as before applies.
\end{proof}
As usual we fix an identification $X_0(N)\cong \PP^1$, with the image of $\mathcal K_4$ and $\mathcal K_6$ under the coarse moduli space denoted by $K_4$ and $K_6$. Up to an automorphism we can identify $K_4$ and $K_6$ with the closed points given in the following lemma. 
\begin{thm} Consider the closed points $K_4:X^2+Y^2=0$ and $K_6=X^2+3Y^2$ of $\PP^1$.  One has that  $\XXX_0(13)^{\rig}\cong \sqrt[2,3]{K_4,K_6},$ where the notation means that $\mu_2$ stackiness is added along $K_4$ and $\mu_3$ stackiness is added along $K_6$. Moreover, one has that $\XXX_0(13)$ is a non-trivial root stack $\XXX_0(13)\cong\sqrt[2]{J_{13}^*(\Lambda^{\rig}})$
	\end{thm}
\begin{proof}
	It follows from Lemma~\ref{nmbc13} that the rigidification is $\XXX_0(13)^{\rig}\cong\sqrt[2,3]{K,D}$. An analogous argument to the one for $N=10$ shows that $\XXX_0(13)$ is a non-trivial gerbe $\mu_2$-gerbe over $\sqrt[2,3]{K,D}$.
\end{proof}
Given a smooth projective $\oF$-curve $C$ and a  pair $(f,g)\in (\oF(C)-\oF)\times H^1(\oF(C),\mu_2)$. By $((\mathcal K_6)_*,a)$ we mean $((\mathcal K_6),a)$ if $\omega\not\in F$ and it has its standard meaning if $\omega \in F$ and by $((\mathcal K_4)_*,a)$ we mean $((\mathcal K_4),a)$ if $i \not\in F$ and it has its standard meaning if $i\in F$. We define $w(1,0)=w(4,1)=1/6$, $w(4,0)=w(1,1)=2/3$, $w(3,0)=w(3,1)=1/2$, $w(2,0)=w(5,1)=1/3$, and $w(5,0)=w(2,1)=5/6$. Similarly, we define $q(3,0)=q(1,1)=1/4$, $q(2,0)=q(2,1)=1/2$ and $q(1,0)=q(3,1)=3/4$. Write $n_{*,(3,0)}$ and $n_{*,(3,1)}$ for $n_{1/2}.$ We define:
\begin{multline*}
	\gamma(f,g)(L,x):=\deg(f)\cdot L+\sum_{*\in\{+,-\}}\sum_{j=1}^5\sum_{r\in w^{-1}(j/6)}n_{*,r}(f,g)x_{((\mathcal K_6)_*,j/6)}+\\+\sum_{*\in\{+,-\}}\sum_{j=1}^3\sum_{r\in q^{-1}(j/4)}n_{*,r}(f,g)x_{((\mathcal K_4)_*,j/4)}.
\end{multline*}
\begin{thm}\label{thm:N=13}
	In the case $N=13$, one has that $a=1/6$ and $b=1$.
\end{thm}
\begin{proof} 
	One has that 
	\begin{multline*}K_{\XXX_0(13),\orb}=\big(1/6,(-1/2)_{((\mathcal K_4)_*,1/4)}, (-1/2)_{((\mathcal K_4)_*,3/4)}, \\(-1)_{1/2},(-2/3)_{((\mathcal K_6)_*,1/6)},(-1/3)_{((\mathcal K_6)_*,1/3)}, (-2/3)_{((\mathcal K_6)_*,2/3)},(-1/3)_{((\mathcal K_6)_*,5/6)}\big).\end{multline*}
	The naive height is given by the element
	\begin{align*}
		T:=(7, (9)_{((\mathcal K_4)_*,1/4)}, (3)_{((\mathcal K_4)_*,3/4)}, (6)_{1/2},(10)_{((\mathcal K_6)_*,1/6)},(8)_{((\mathcal K_6)_*,1/3)}, (4)_{((\mathcal K_6)_*,2/3)},(2)_{((\mathcal K_6)_*,5/6)})
	\end{align*}
	We want the infimum of $a$ for which $aT+K_{\XXX_0(13),\orb}\in\Eff_{\orb}(\XXX_0(13))$. Recall that $x_{1/2}\geq 0$ is an equation defining cone, so $a\geq 1/6$. Consider the point
	\begin{align*}O:&=(1/6)\cdot T+K_{\XXX_0(13),\orb} \\
		&=(4/3, (1)_{((\mathcal K_4)_*,1/4)}, (0)_{((\mathcal K_4)_*,3/4)}, (0)_{1/2},(4/3)_{((\mathcal K_6)_*,1/6)},(2/3)_{((\mathcal K_6)_*,1/3)}, (1/3)_{((\mathcal K_6)_*,2/3)},(-1/3)_{((\mathcal K_6)_*,5/6)}).
	\end{align*}
	Let $C$ be a smooth projective irreducible $\oF$-curve, and $(f,g)$ as usual. We have that 
	$$\gamma(f,g)(O)\geq 4/3\cdot \deg(f)+ (-1/3)\sum_{*\in\{+,-\}}(n_{*,5,0}(f,g)+n_{*,1,1}(f,g)).$$
	We have that $\deg(f)\geq 5 n_{+,5,0}(f,g)+n_{+,1,1}(f,g)$ and $\deg(f)\geq 5 n_{-,5,0}(f,g)+n_{-,1,1}(f,g)$. We obtain \begin{align*}4\deg(f)&\geq 10n_{+,5,0}(f,g)+2n_{+,1,1}(f,g)+10n_{n-,5,0}(f,g)+2n_{-,1,1}(f,g)\\&\geq \sum_{*\in\{+,-\}}(n_{*,5,0}(f,g)+n_{*,1,1}(f,g)).
	\end{align*}
	Hence $\gamma(f,g)(O)\geq 0$ for all $f,g$ as above. Thus $O\in\Eff_{\orb}(\XXX_0(13))$. Hence $a=1/6.$ For the equality $\gamma(f,g)(O)=0$ to hold, one must have $$n_{*,5,0}(f,g)=n_{*,1,1}(f,g)=0$$ for $*\in\{+,-\}$. But as $\deg(f)\geq 1$, we have that $\gamma(f,g)(O)>0$, so the equality is not achieved for any $f,g$ as above. It follows that $$\{\lambda\in \NS_{\orb}(\XXX_0(13))^*|\hspace{0,1cm}\lambda(O)=0, \lambda\in\Eff_{\orb}(\XXX_0(13))^*\}=\{(L,x)\mapsto tx_{1/2}|\hspace{0,1cm}t\geq 0\}$$ which is of dimension $1$. Hence $b=1$. The claim is proven.
\end{proof}

\bibliographystyle{amsalpha}
\bibliography{bibliography.bib}

\providecommand{\bysame}{\leavevmode\hbox to3em{\hrulefill}\thinspace}
\providecommand{\MR}{\relax\ifhmode\unskip\space\fi MR }
\providecommand{\MRhref}[2]{%
  \href{http://www.ams.org/mathscinet-getitem?mr=#1}{#2}
}
\providecommand{\href}[2]{#2}
\begin{thebibliography}{AHPP25}

\bibitem[AGV08]{AGV08}
Dan Abramovich, Tom Graber, and Angelo Vistoli, \emph{Gromov-{W}itten theory of
  {D}eligne-{M}umford stacks}, Amer. J. Math. \textbf{130} (2008), no.~5,
  1337--1398. \MR{2450211}

\bibitem[AH11]{stable_with_twist}
Dan Abramovich and Brendan Hassett, \emph{Stable varieties with a twist},
  Classification of algebraic varieties. Based on the conference on
  classification of varieties, Schiermonnikoog, Netherlands, May 2009.,
  Z{\"u}rich: European Mathematical Society (EMS), 2011, pp.~1--38 (English).

\bibitem[AHPP25]{AHPP25}
Santiago {Arango-Pi{\~n}eros}, Changho {Han}, Oana {Padurariu}, and Sun~Woo
  {Park}, \emph{{Counting 5-isogenies of elliptic curves over $\mathbb{Q}$}},
  arXiv e-prints (2025), arXiv:2504.07750.

\bibitem[AOV08]{Abramovich_Olsson_Vistoli}
Dan Abramovich, Martin Olsson, and Angelo Vistoli, \emph{Tame stacks in
  positive characteristic}, Ann. Inst. Fourier \textbf{58} (2008), no.~4,
  1057--1091 (English).

\bibitem[Bru92]{Brumer}
Armand Brumer, \emph{The average rank of elliptic curves. {I}}, Invent. Math.
  \textbf{109} (1992), no.~3, 445--472. \MR{1176198}

\bibitem[BS24]{boggesssankar}
Brandon Boggess and Soumya Sankar, \emph{Counting elliptic curves with a
  rational {{\(N\)}}-isogeny for small {{\(N\)}}}, J. Number Theory
  \textbf{262} (2024), 471--505 (English).

\bibitem[BT98]{BatyrevTschinkel}
Victor~V. Batyrev and Yuri Tschinkel, \emph{Tamagawa numbers of polarized
  algebraic varieties}, Nombre et r\'epartition de points de hauteur born\'ee,
  Paris: Soci{\'e}t{\'e} Math{\'e}matique de France, 1998, pp.~299--340
  (English).

\bibitem[BV24]{valcritstack}
Giulio Bresciani and Angelo Vistoli, \emph{An arithmetic valuative criterion
  for proper maps of tame algebraic stacks}, Manuscr. Math. \textbf{173}
  (2024), no.~3-4, 1061--1071 (English).

\bibitem[Cad07]{impose_tangency}
Charles Cadman, \emph{Using stacks to impose tangency conditions on curves},
  Am. J. Math. \textbf{129} (2007), no.~2, 405--427 (English).

\bibitem[Dar21]{darda:tel-03682761}
Ratko Darda, \emph{{Rational points of bounded height on weighted projective
  stacks}}, Theses, {Universit{\'e} Paris Cit{\'e}}, September 2021.

\bibitem[DR73]{DeligneRapoport}
Pierre Deligne and M.~Rapoport, \emph{Moduli schemes of elliptic curves},
  Modular {Functions} of one {Variable} {II}, {Proc}. {Int}. {Summer} {School},
  {Univ}. {Antwerp} 1972, {Lect}. {Notes} {Math}. 349, 143-316 (1973)., 1973.

\bibitem[DY24]{dardayasudabm}
Ratko Darda and Takehiko Yasuda, \emph{{B}atyrev-{M}anin conjecture for {D}{M}
  stacks}, Journal of European Mathematical Society (2024), published online
  first.

\bibitem[DY25a]{orbifold-toric}
Ratko {Darda} and Takehiko {Yasuda}, \emph{{Orbifold pseudo-effective cones of
  toric stacks}}, arXiv e-prints (2025), arXiv:2508.20434.

\bibitem[DY25b]{dardayasudabm2}
\bysame, \emph{{The Batyrev-Manin conjecture for DM stacks II}}, arXiv e-prints
  (2025), arXiv:2502.07157.

\bibitem[ESZB23]{ellenberg_satriano_zureick-brown_2023}
Jordan~S. Ellenberg, Matthew Satriano, and David Zureick-Brown, \emph{Heights
  on stacks and a generalized batyrev–manin–malle conjecture}, Forum of
  Mathematics, Sigma \textbf{11} (2023), e14.

\bibitem[FMN10]{mann}
Barbara Fantechi, Etienne Mann, and Fabio Nironi, \emph{Smooth toric
  {D}eligne-{M}umford stacks}, J. Reine Angew. Math. \textbf{648} (2010),
  201--244. \MR{2774310}

\bibitem[Gra00]{Grant}
David Grant, \emph{A formula for the number of elliptic curves with exceptional
  primes}, Compos. Math. \textbf{122} (2000), no.~2, 151--164 (English).

\bibitem[{Lop}23]{lopez2023picardgroupsstackycurves}
Rose {Lopez}, \emph{{Picard Groups of Stacky Curves}}, arXiv e-prints (2023),
  arXiv:2306.08227.

\bibitem[Mal04]{Malle}
Gunter Malle, \emph{On the distribution of {G}alois groups. {II}}, Experiment.
  Math. \textbf{13} (2004), no.~2, 129--135. \MR{2068887}

\bibitem[Maz78]{Mazur78}
B.~Mazur, \emph{Rational isogenies of prime degree (with an appendix by {D}.
  {G}oldfeld)}, Invent. Math. \textbf{44} (1978), no.~2, 129--162. \MR{482230}

\bibitem[Mol23]{Molnar2023}
Grant~S. Molnar, \emph{Counting elliptic curves with a cyclic $m$-isogeny over
  $\mathbb{Q}$}, Ph.d. thesis, Dartmouth College, 2023.

\bibitem[MV23]{7isogeny}
Grant Molnar and John Voight, \emph{Counting elliptic curves over the rationals
  with a 7-isogeny}, Res. Number Theory \textbf{9} (2023), no.~4, 31 (English),
  Id/No 75.

\bibitem[Ols12]{Olsson12}
Martin Olsson, \emph{Integral models for moduli spaces of {{\(G\)}}-torsors},
  Ann. Inst. Fourier \textbf{62} (2012), no.~4, 1483--1549 (English).

\bibitem[Ols16]{Olsson16}
\bysame, \emph{Algebraic spaces and stacks}, American Mathematical Society
  Colloquium Publications, vol.~62, American Mathematical Society, Providence,
  RI, 2016. \MR{3495343}

\bibitem[Pey95]{Peyre}
Emmanuel Peyre, \emph{Hauteurs et mesures de {T}amagawa sur les
  vari\'{e}t\'{e}s de {F}ano}, Duke Math. J. \textbf{79} (1995), no.~1,
  101--218. \MR{1340296}

\bibitem[{Phi}22]{phillips2024pointsboundedheightimages}
Tristan {Phillips}, \emph{{Points of bounded height in images of morphisms of
  weighted projective stacks with applications to counting elliptic curves}},
  arXiv e-prints (2022), arXiv:2201.10624.

\bibitem[PPV20]{countthreeisog}
Maggie Pizzo, Carl Pomerance, and John Voight, \emph{Counting elliptic curves
  with an isogeny of degree three}, Proc. Am. Math. Soc., Ser. B \textbf{7}
  (2020), 28--42 (English).

\bibitem[PS21]{zbMATH07357690}
Carl Pomerance and Edward~F. Schaefer, \emph{Elliptic curves with
  {Galois}-stable cyclic subgroups of order 4}, Res. Number Theory \textbf{7}
  (2021), no.~2, 19 (English), Id/No 35.

\bibitem[Sch74]{Schoeneberg}
B.~Schoeneberg, \emph{Elliptic modular functions. {An} introduction.
  {Translated} from the {German} by {J}. {R}. {Smart} and {E}. {A}.
  {Schwandt}}, Grundlehren Math. Wiss., vol. 203, Springer, Cham, 1974
  (English).

\bibitem[Shi71]{shimura-book}
Goro Shimura, \emph{Introduction to the arithmetic theory of automorphic
  functions}, Kan\^o{} Memorial Lectures, vol. No. 1, Iwanami Shoten
  Publishers, Tokyo; Princeton University Press, Princeton, NJ, 1971,
  Publications of the Mathematical Society of Japan, No. 11. \MR{314766}

\bibitem[Sil09]{Arithmetic_of_Elliptic}
Joseph~H. Silverman, \emph{The arithmetic of elliptic curves}, 2nd ed. ed.,
  Grad. Texts Math., vol. 106, New York, NY: Springer, 2009 (English).

\bibitem[{Sta}26]{stacks-project}
The {Stacks Project Authors}, \emph{The {S}tacks project},
  \url{https://stacks.math.columbia.edu}, 2026.

\bibitem[VZB22]{VZB}
John Voight and David Zureick-Brown, \emph{The canonical ring of a stacky
  curve}, Mem. Am. Math. Soc., vol. 1362, Providence, RI: American Mathematical
  Society (AMS), 2022 (English).

\bibitem[Yan06]{YangModular}
Yifan Yang, \emph{Defining equations of modular curves}, Advances in
  Mathematics \textbf{204} (2006), no.~2, 481--508.

\end{thebibliography}
\end{document}